\documentclass[preprint,11pt,number,nopreprintline]{elsarticle}
\usepackage[letterpaper,margin=1in]{geometry}
\usepackage{amsmath,amssymb,amsthm,mathtools}
\usepackage{booktabs,tabularx,microtype}
\usepackage[hidelinks]{hyperref}
\newtheorem{theorem}{Theorem}[section]
\newtheorem{proposition}[theorem]{Proposition}
\newtheorem{lemma}[theorem]{Lemma}
\newtheorem{corollary}[theorem]{Corollary}
\theoremstyle{definition}
\newtheorem{remark}[theorem]{Remark}
\newcommand{\Z}{\mathbb Z}
\newcommand{\Q}{\mathbb Q}
\newcommand{\F}{\mathbb F}
\newcommand{\K}{K}
\newcommand{\Ht}{\widetilde H}
\newcommand{\Ind}{\operatorname{Ind}}
\newcommand{\Cay}{\operatorname{Cay}}
\newcommand{\tr}{\operatorname{tr}}
\hypersetup{pdftitle={Integral homology lattices of matching and circulant complexes},pdfauthor={Qihang Wang and Weiye Li},pdfsubject={Integral homology actions and equivariant lifting}}
\begin{document}
\begin{frontmatter}
\title{Integral homology lattices of matching and circulant complexes}
\author[pku]{Qihang Wang\corref{corresponding}\fnref{equal}}
\ead{2501110049@stu.pku.edu.cn}
\author[thu]{Weiye Li\fnref{equal}}
\ead{liwy25@mails.tsinghua.edu.cn}
\address[pku]{School of Mathematical Sciences, Peking University, Beijing 100871, China}
\address[thu]{Department of Automation, Tsinghua University, Beijing 100084, China}
\cortext[corresponding]{Corresponding author.}
\fntext[equal]{Qihang Wang and Weiye Li contributed equally to this work and are joint first authors.}
\begin{abstract}
Bouc's hook cycles span the rational top homology of the matching complex of $K_{2k+1}$.
We prove that their integral saturation quotient has exponent exactly $2^{\lfloor k/2\rfloor}$ for every $k\geq1$, and construct divided-power cycles attaining this bound.
For $k\geq2$, the homology lattice is not isomorphic to the standard integral hook lattice as an $\mathfrak S_{2k+1}$-module.
For circulant complexes, we determine the integral dihedral homology modules of $\Ind(C_{3s}^{s-1})$, $\Ind(C_{3s+1}^{s-1})$ and $\Ind(C_{3s+2}^{s-1})$, for $s\geq3$, $s\geq2$ and $s\geq3$, respectively, and classify the least scalar lifts of all primitive equivariant quotients to the rank-one module with trivial rotation and negative reflection.
Tate localization relates prime-order gluing to fixed-point homology under a free, single-degree homology hypothesis.
Explicit cycle subgroups separate saturation defects from obstructions to equivariant sections.
For circulant independence complexes with jumps two and three, we prove an eight-periodic ordinary homotopy classification and a fixed-space obstruction to uniform representation suspensions.
For every fixed nonempty bounded positive jump set, a finite-state construction gives a rational critical-face series.

AI models used: OpenAI GPT-5.6 Sol, Anthropic Claude Fable 5, Grok 4.6, OpenAI GPT-6.1 Sol, OpenAI GPT-6 Astra and Anthropic Claude Fable 5.1.
\end{abstract}
\end{frontmatter}

\section{Introduction}
A family of cycles may span rational homology without generating integral homology.
When a group acts on the complex, two further questions arise: whether a rational decomposition splits over $\Z$, and whether a specified quotient of integral homology admits an equivariant section.
The character of the rational representation does not answer either question.
We study these integral problems for matching complexes and for independence and clique complexes of circulant graphs.

Our starting point is the inclusion of explicit cycles in the simplicial chain complex.
For odd matching complexes, it leads to a sharp denominator theorem for Bouc's hook-cycle lattice.
For circulant complexes, cycle presentations and signed comparison maps determine the integral actions of the original vertex permutations, including the position of particular cycle subgroups.
Fixed-point spaces explain several of the resulting arithmetic obstructions and also show why an ordinary homotopy reduction need not retain the symmetry.

\subsection{Saturation in odd matching complexes}
For the matching complex $\mathcal M_{2k+1}$ of $K_{2k+1}$, let
$A^{(k)}=\ker(\sum:\Z^{2k+1}\to\Z)$ and $S^{(k)}=\bigwedge^k A^{(k)}$.
Since the complex has dimension $k-1$, its top homology $\mathcal H_k=\Ht_{k-1}(\mathcal M_{2k+1};\Z)$ is a free abelian group.
Bouc's hook cycles define an inclusion $\Psi_k:S^{(k)}\to\mathcal H_k$ which becomes an isomorphism over $\Q$, although its integral image can be proper~\cite{BoucMatching}.
We ask for the least positive integer that brings every top cycle into this image after multiplication, or equivalently the exponent of $\mathcal H_k/\Psi_k(S^{(k)})$.

\begin{theorem}\label{thm:matching-main}
For every $k\geq1$, the quotient $\mathcal H_k/\Psi_k(S^{(k)})$ is a finite $2$-group of exponent exactly $2^{\lfloor k/2\rfloor}$.
Divided-power cycles explicitly realize every order $2^j$, $0\leq j\leq\lfloor k/2\rfloor$.
For every $k\geq2$, $\mathcal H_k$ is not isomorphic to $S^{(k)}$ as an integral $\mathfrak S_{2k+1}$-module.
\end{theorem}

The upper bound follows by contracting to a matching complex on four fewer ground vertices and applying a two-edge switching identity.
For sharpness, Pfaffian divisibility produces divided-power cycles with the required denominators. Section~\ref{sec:om-saturation} gives both arguments with the integral normalization of Bouc's map fixed at the outset.
Theorem~\ref{thm:om-exponent} describes their canonical sublattice, and Corollary~\ref{cor:om-first-layer} identifies the submodule generated by the first divided power.
In particular, the quotient contains an element of order four from nine vertices onward.
The exponent quantifies the failure of Bouc's integral cycle span to be saturated; a description of all its invariant factors remains open.

\subsection{Integral actions on circulant complexes}
For the circulant complexes, put
\[
 G(n,s)=\Cay(\Z/n\Z,\{\pm1,\pm s\}),\qquad
 \K(n,s)=\Ind(G(n,s)).
\]
Vertices are represented by $0,\ldots,n-1$, and a simplex is a set containing no pair with difference $\pm1$ or $\pm s$ modulo $n$.
The chosen generator $\tau$ of $C_n$ acts by $x\mapsto x+1$.
We make $\Ht_q(\K(n,s);\Z)$ a $\Z[t]$-module by letting $t$ act as $\tau_*$; thus $t^n=1$.
Let $\rho(x)=-x$ modulo $n$. We write $D_m=\langle\tau,\rho\mid\tau^m=\rho^2=1,\ \rho\tau\rho=\tau^{-1}\rangle$ for the dihedral group of order $2m$.
Unless explicitly stated to be equivariant, homotopy equivalences are ordinary.
We also compare the clique complexes $\mathcal X_s=\operatorname{Cl}(G(3s,s))=\Ind(\overline{G(3s,s)})$ for $s\geq3$. The graph complement is essential: $\mathcal X_s$ and $\K(3s,s)$ are different complexes on the same labeled vertex set.

There are two complementary ways to retain symmetry through a homotopy calculation.
An ordered Morse matching gives an ordinary CW model, even when the matching itself is not invariant.
Signed comparison maps then recover the action on integral homology.
Alternatively, orbit coordinates identify the geometric fixed spaces directly, providing information that homology characters alone cannot supply.

The first uniform presentations arise from the cycle-power complexes $\mathcal Y_s=\Ind(C_{3s}^{s-1})$.
Their triangles collapse to give the ordinary homotopy type.
An integral presentation by graph cycles and triangle boundaries retains both dihedral generators and gives the following parity-dependent decompositions.
Each block in the theorem is the cyclic lattice $\Z[t]/(f)$, where $f$ is its listed annihilator; multiplication by $t$ gives the translation, and the last column specifies the reflection.

\begin{theorem}\label{thm:cyclepoweraction}
Let $s\geq3$ and $\mathcal Y_s=\Ind(C_{3s}^{s-1})$, where $C_n^r$ joins vertices of circular distance at most $r$.
With the labeled actions $T(i)=i+1$ and $R(i)=-i$, its only nonzero reduced integral homology is $H_1$, of rank $(s-1)(3s-2)/2$.
As an integral $D_{3s}$-module,
\[
 H_1(\mathcal Y_s;\Z)\cong
 \begin{cases}
 M_s\oplus(\Pi_{3s}^0)^{(s-3)/2},&s\text{ odd},\\
 B_v\oplus W_v\oplus(\Pi_{6v}^0)^{a_v}\oplus(\Pi_{6v}^1)^{b_v},&s=2v,
 \end{cases}
\]
where all blocks have $T=t$, and their annihilators and reflections are
\[
\begin{array}{c|c|c}
\text{Block}&\text{Annihilator}&Rp(t)\\\hline
M_s&(t-1)(1+t^s+t^{2s})&-p(t^{-1})\\
\Pi_n^\epsilon&t^n-1&-t^\epsilon p(t^{-1})\\
B_v&1-t^v+t^{2v}&-p(t^{-1})\\
W_v&(t-1)(1+t^v+\cdots+t^{5v})&-t^{-v}p(t^{-1})
\end{array}
\]
Here $(a_v,b_v)=(v/2-1,v/2-1)$ for even $v$, and $((v-1)/2,(v-3)/2)$ for odd $v$.
\end{theorem}

For a positive integer $d$ and a specified equivariant quotient $q:H\to Q$, a scalar lift of $d$ is an equivariant homomorphism $a:Q\to H$ such that $qa=d\operatorname{id}_Q$.
The least such positive $d$ measures the obstruction to an equivariant section.
Corollary~\ref{cor:cyclepowerscalar} classifies this minimum for every primitive quotient of $H_1(\mathcal Y_s;\Z)$ to the rank-one module $\Z_-$ with $T=1$ and $R=-1$.
For odd $s\geq5$ the possible minima are $3e$ for $e\mid s$; for $s=2v\geq6$ they are $6e$ for $e\mid v$.
At $s=3,4$ the only minima are three and six.
The proof in Section~\ref{sec:cyclepowerfamily} gives integral cycle coordinates, a complete relation presentation, and a reflection-compatible splitting.

At the adjacent vertex counts $3s+1$ and $3s+2$, the triangles have trivial stabilizers.
Their attaching relations eliminate one or two full chord orbits, giving simpler integral modules.

\begin{theorem}\label{thm:cyclepowerneighbors}
Let $a=1$, $s\geq2$, or $a=2$, $s\geq3$, and put
$n=3s+a$ and $\mathcal Z_{s,a}=\Ind(C_n^{s-1})$.
With the labeled actions $T(i)=i+1$ and $R(i)=-i$, its only nonzero reduced integral homology is $H_1$, of rank $n(s-a-1)/2+1$.
For odd $n$, set $k=(s-a-1)/2$; then
\[
 H_1(\mathcal Z_{s,a};\Z)\cong\Z_-\oplus(\Pi_n^0)^k.
\]
For even $n=2m$, set $k=(s-a-2)/2$; then
\[
 H_1(\mathcal Z_{s,a};\Z)\cong
 H_m\oplus(\Pi_n^0)^{\lfloor k/2\rfloor}\oplus(\Pi_n^1)^{\lceil k/2\rceil},
\]
where the integral block $H_m=\Z[t]/((t-1)(1+t^m))$ has $T=t$ and $Rp(t)=-t^{-m}p(t^{-1})$.
The signed regular blocks $\Pi_n^\epsilon$ and the rank-one module $\Z_-$ have the actions specified above.
\end{theorem}

Corollary~\ref{cor:neighborscalar} determines all primitive scalar quotient lifts for these modules.
The half-turn block has augmentation minimum two, whereas the odd-vertex rank-one summand has minimum one.
The other maps can have larger minima, determined by their coefficients on the regular summands.
These adjacent formulas separate the role of the triangle stabilizers from that of the antipodal edges in the integral action.

The integral decompositions also admit a geometric interpretation.
For a prime-order subgroup $P=\langle\sigma\rangle$, put
\[
 N_P=1+\sigma+\cdots+\sigma^{p-1},\qquad
 \mathcal G_P(L)=L/(L^P+\ker N_P).
\]
This quotient measures the failure of the invariant and noninvariant rational summands to split integrally.
Under the free, single-degree homology hypothesis, classical Tate localization calculates it from the fixed-point space (Theorem~\ref{thm:fixedgluing}).
The clique complexes $\mathcal X_s$ give a parameter-family application, including the residual dihedral action on the gluing quotient (Theorem~\ref{thm:geometricfamily} and Proposition~\ref{prop:residualgluing}).
The localization method is classical; the cycle presentations identify the particular lattices and quotient maps to which it applies.

Write $\zeta_m=e^{2\pi i/m}$ and define
\begin{align*}
 \Phi_6(t)&=t^2-t+1,\qquad \Phi_9(t)=t^6+t^3+1,\\
 L_3&=\Z[t]/((t-1)\Phi_9(t)),\\
 L_2&=\Z[t]/((t^9-1)\Phi_6(t)),\qquad
 L_{14}=\Z[t]/(1+t+\cdots+t^{13}).
\end{align*}
On $\Z[\zeta_m]$ and $\Q(\zeta_m)$, $t$ is multiplication by $\zeta_m$.
On a trivial summand it is the identity.
The notation $\Z[C_9]$ means $\Z[t]/(t^9-1)$ with its usual regular action.
Inflation from $C_d$ to $C_n$, for $d\mid n$, uses the quotient taking the chosen generator to the chosen generator.

The finite circulant examples exhibit the same obstructions in higher homological degrees.
The translation lattices of $\K(18,4)$, $\K(18,5)$ and $\K(27,5)$ are $L_3,L_2,L_3$.
The factor-kernel sums have indices three and four in the two lattice types, while the two occurrences of $L_3$ have opposite reflection signs (Theorem~\ref{thm:dihedral}).
Fixed-point homology explains the defects for the eighteen-vertex examples and the three-factor quotient of $\K(14,2)$.
For $\K(27,5)$ the geometric subgroup $\langle T^9\rangle$ acts trivially on homology; Lemma~\ref{lem:centralcomparison} instead passes to the residual action on its fixed space to recover the effective index-three gluing.
The eighteen-case ordinary homotopy and translation-module table appears in Appendix~\ref{app:finiteclassification}. Three additional benchmarks in Theorem~\ref{thm:three-benchmarks} are treated by direct reductions; their thirteen-vertex homology is the integral cyclic augmentation lattice, and their two fourteen-vertex spheres have opposite translation signs (Theorem~\ref{thm:cyclic-action}).
The $L_{14}$ representation agrees with the earlier circular-arc nerve representation, as detailed in Proposition~\ref{prop:composite}.

A specified cycle subgroup gives a finer question than the abstract type of the ambient lattice.
For $G(18,2)$, twenty-seven classical adjacent-pair cross-cycles generate a primitive rank-twenty-three subgroup in rank-twenty-five homology.
The quotient is a free rank-two permutation module, but its least positive equivariant scalar lift is nine.
For $G(22,2)$, the analogous family spans rational homology and has integral index two.
Proposition~\ref{prop:tilingdihedral} identifies its position as a parity kernel on a norm-quotient summand together with an entire regular summand.
These calculations distinguish a saturation defect from the failure of an equivariant section of a free quotient.

The clique family makes the dependence on the quotient map particularly explicit.
Triangle boundaries give $H_1(\mathcal X_s;\Z)\cong M_s$, with negative inversion as reflection.
Its augmentation to $\Z_-$ has least scalar lift three for every $s\geq3$.
Before filling the triangles the graph homology splits integrally, but its primitive scalar quotient maps realize every divisor of the vertex count as a least lift.
Proposition~\ref{prop:flaggraphs} determines the fixed-lattice image for each map and identifies which maps descend through triangle filling.
An ambient direct-sum decomposition therefore does not settle the lifting question for a chosen quotient.

The jump-three families distinguish integral dihedral lattices with equal rational characters by their coinvariant reflections. Appendix~\ref{app:guide} locates these and the further action comparisons.

\subsection{Periodic reductions and fixed spaces}
An integral action calculation requires an underlying homology calculation.
For circulant graphs with bounded jumps, an ordered matching can sometimes be organized into local blocks, yielding ordinary suspension periods.
The distinction between such a period and an equivariant one is already visible in the family with jumps two and three.

For a nonempty finite positive jump set $J$, put $G_n(J)=\Cay(\Z/n\Z,\{\pm j:j\in J\})$ for $n>\max J$.
For the family $X_n=\Ind(G_n(\{2,3\}))$, the local criterion of Theorem~\ref{lb-criterion} reduces an infinite induction to two finite boundary identities.
It gives $X_{n+8}\simeq\Sigma^2X_n$ and the complete residue-dependent sphere multiplicities and dimensions in Corollary~\ref{lb-eight-periodic}.
For an arbitrary fixed bounded jump set, Theorem~\ref{thm:rational-critical-series} instead constructs a finite automaton and a rational critical-face series, without a duplication hypothesis.
This counts Morse cells; attaching maps remain a separate question.
Theorem~\ref{thm:family-cyclic-character} determines the rational translation module of $X_n$.
The full cyclic fixed-space profile in Theorem~\ref{thm:fixed-profile} is stronger: it shows that no positive uniform ordinary period can be refined to a representation suspension for all common cyclic subgroups (Theorem~\ref{thm:no-equivariant-period}).
These results explain why ordinary periodicity, a homology action and an equivariant homotopy type require different arguments.

The proofs are organized as follows.
Section~\ref{sec:om-saturation} proves the matching saturation theorem.
Sections~\ref{sec:normgluing} and~\ref{sec:cyclepowerfamily} treat fixed-point gluing and the three cycle-power presentations; Section~\ref{sec:framework} gives the low-rank trace argument.
The local and finite-state Morse constructions occupy Sections~\ref{lb-section}--\ref{sec:finite-state}, followed by the fixed-space obstruction in Section~\ref{sec:fixed-spaces} and the direct reductions and cyclic actions in Sections~\ref{sec:benchmarks}--\ref{sec:symmetry}.
The remaining sections develop the finite integral action calculations, reflections and embedded cross-cycle subgroups.
Section~\ref{sec:reconstruction} specifies their reconstruction, and Section~\ref{sec:further-questions} discusses the unresolved integral classifications.

\subsection{Related work}
The chain constructions used here belong to discrete Morse theory.
Forman's theorem gives the ordinary CW reduction \cite{Forman}, while the signed algebraic comparison maps follow the framework of Sk\"oldberg and of J\"ollenbeck and Welker \cite{Skoldberg,JW}.
Freij develops equivariant discrete Morse theory for group homotopy types \cite{Freij}.
Here the matching may be noninvariant: we retain the group action on the original simplicial chains and transfer it to homology through integral comparison maps.
The local block criterion in Section~\ref{lb-section} separates a finite Morse matching calculation from the induction needed for a parameter family. Its periodic application provides ordinary reductions; the signed comparison maps supply the integral actions.

Bouc constructs the hook cycle map and identifies the rational top representation of odd matching complexes~\cite[Lemma~8, Proposition~5 and Section~3.3.2]{BoucMatching}.
He also notes explicitly, for $k=2$, that the cycle map does not induce an isomorphism of integral $\mathfrak S_5$-modules, so the integral cycle span need not be the entire homology lattice~\cite[Section~3.3.2]{BoucMatching}.
Theorem~\ref{thm:matching-main} determines the sharp power-of-two denominator for that specified geometric cycle span, with explicit witnesses; the cycle construction and rational decomposition are prior results.
Below the top degree, the integral homology of matching complexes has torsion: Bouc found $3$-torsion, and Shareshian and Wachs showed that the bottom nonvanishing homology is an elementary $3$-group for all sufficiently large $n$~\cite{ShareshianWachs}; Wachs surveys these complexes~\cite{WachsSurvey}.
The saturation question studied here concerns the free top lattice and the prime $2$, and is independent of that torsion.
Danz and Hofmann describe integral hook forms, including their exterior-power realization and local classifications~\cite{DanzHofmannHooks}.
Their later work resolves the remaining exterior-degree-two case in an earlier conjectural classification~\cite[Theorem~3.7]{DanzHofmannZeta}.
Kapadia's characteristic-two hook filtrations and subsequent decomposition results give additional modular context~\cite[Theorem~5.1]{KapadiaHooks} and~\cite{KapadiaHooksII}.
Our first-layer construction realizes a modular exterior-algebra module inside the geometric saturation quotient. Determining this geometric lattice is distinct from classifying all integral or modular hook forms.

Circulant graph complexes have been studied from several complementary viewpoints.
Adamaszek classifies clique complexes of cycle powers, equivalently independence complexes of circular complete graphs \cite{Adamaszek}, and Nilakantan and Shukla treat a regular bipartite family \cite{NS}.
Earl, Vander Meulen and Van Tuyl study independence complexes of well-covered circulants; Table~1 of their public preprint includes $G(14,6)$ and $G(15,5)$ in a classification by vertex decomposability, shellability and related properties \cite{Earl}.
Rinaldo and Romeo use cyclic symmetry to study their reduced Euler characteristics \cite{RR}.
These results concern the graph class directly, whereas our comparisons retain the integral modules of the specified vertex actions.
The recurrence for independence complexes of cycle powers in Adamaszek's Theorem~1.1 requires $n\geq5r+4$ \cite{CyclePowers}.
For the families $\mathcal Y_s$ and $\mathcal Z_{s,a}$, $r=s-1$ and the stated vertex counts lie outside that range. Remark~6.6 of the public preprint leaves the smaller initial regime aside and points to other methods for ordinary types.
The clique classification in \cite{Adamaszek} concerns $\operatorname{Cl}(C_n^r)$, equivalently independence complexes of complementary circular complete graphs.
Here $\mathcal Y_s$ and $\mathcal Z_{s,a}$ take the opposite complement direction.
Our triangle-collapse proofs give their ordinary types directly; Theorems~\ref{thm:cyclepoweraction} and~\ref{thm:cyclepowerneighbors} determine their full integral labeled dihedral actions.
These three families do not exhaust the cycle powers.
The preprint of Rather, Pirzada and Aijaz studies edge rings of complements of cycle powers through the induced homology of $\Ind(\overline{C_n^m})=\operatorname{Cl}(C_n^m)$ \cite[Section~2 and Theorem~3.10]{RatherCyclePowers}. Their complex again has the opposite complement direction; the graded Betti numbers studied there are distinct from the integral actions considered here.

Square-grid independence complexes also have extensive Morse and Euler-characteristic treatments \cite{Jonsson,Thapper}.
Matching trees provide another systematic reduction method \cite{BousquetMelouLinussonNevo2008}.
For width-three hard-square cylinders, the eight-column recurrence has suspension exponent six \cite[Theorem~1.1(c)]{AdamaszekHardSquares}; Okura gives a simple-homotopy refinement \cite[Corollary~1.5]{Okura2019}.
The ordinary three-row cylinder classification is due to Iriye \cite{Iriye2012}; its formula, recalled and strengthened in \cite[Corollary~1.5 and the preceding paragraph]{Okura2019}, has the same modulo-eight multiplicities $5,3,1$ as Corollary~\ref{lb-eight-periodic}, with a different dimension shift.
These are comparisons of patterns; no graph transformation or Alexander duality with $X_n$ is established here.
The two fourteen-vertex and thirteen-vertex benchmarks also have known contexts.
The recursion for the cycle square in \cite{CyclePowers} begins at $n=14$, so Section~\ref{sec:p683} gives a direct reduction at $n=13$.
For $G(14,5)$, write even and odd vertices as $2i$ and $2j+1$ in $\Z/7\Z$.
Adjacency is $j-i\in\{0,2,4,6\}$; multiplying both indices by $4$ gives $\{0,1,2,3\}$, identifying this graph with $G_7^4$ of \cite{NS}.
Their family theorem already gives $S^2$ at $7=1(4+1)+2$.
The benchmarks supply explicit reductions and labeled cyclic actions, rather than priority claims for these homotopy types.

The closest representation-theoretic precedent is the work of Adamaszek, Adams, Frick, Peterson and Previte-Johnson on nerve complexes of circular arcs \cite{Nerve}.
Their $N(n,k)$ is the nerve of the evenly spaced closed arcs $[i/n,(i+k)/n]$ in $\mathbb R/\mathbb Z$.
They give explicit integral homology bases and compute both dihedral generators in Section~7 of the public preprint.
In particular, their Propositions~6.1 and~7.2 already realize $L_{14}$ on $N(28,14)$ with the same dihedral action after shifting the cyclic module generator.
Proposition~\ref{prop:composite} identifies this agreement.
The odd-dimensional wedges of seven or eleven spheres in Theorem~\ref{thm:main} have different ordinary types from their uniform-arc nerves.
Adams, Lagoda, Moy, Sadovek and De Saha obtain an $S^1$-equivariant homotopy model for metric thickenings of the full circle \cite[Theorem~7.1]{PersistentEquivariant}; this concerns a different object and a stronger equivalence than the finite integral homology identifications considered here.

The arithmetic of polynomial factors is also classical.
Mahatab and Sampath study integer Chinese remainder maps, resultant determinants and Smith invariants for complete cyclotomic factorizations of $t^n-1$ \cite{MS}.
Our factor-ideal calculations use this framework.
Lemma~\ref{lem:crtdual} relates the factor-kernel embedding to polynomial reduction by a perfect integral pairing, so that Proposition~\ref{prop:flagfactorquotients} can apply their coprime tensor isomorphisms and prime-power Smith formula to $M_s$.
For equivariant lifting, we use the classical induction adjunction and the description of equivariant maps by fixed elements \cite[Section~3.3, equation~(3.6)]{BrownLectures}.
The calculation for the specified eighteen-vertex quotient determines the lifting ideal $(9,3(1+t))$ by finding the image of the entire fixed lattice.

Rotation traces require both independent-set counts and orientation signs.
The circular counting formula in \cite[equation~(1.1)]{MansourSun}, including its attribution of the earlier $m=1$ case, supplies the face counts.
White's cyclic-sieving theorem concerns rotations of fixed-size independent sets in powers of cycles \cite[Theorem~1]{White}.
Those unsigned set counts enter a different stage from the signed Lefschetz calculation and the integral basis construction.

Cross-cycles from induced matchings are established constructions \cite[Section~2]{SpecialCycles}.
Jonsson constructs integral cross-cycles from induced matchings and maximal faces on grid quotients, distinguishing existence from completeness as a homology basis \cite[Section~1.2 and Proposition~3.1]{JonssonCross}.
His remark following Proposition~5.1 also poses a subgroup-rank question for balanced-tiling families.
The adjacent-pair gap families in Sections~\ref{sec:tiling} and~\ref{sec:crossfamily} are specified on the cycle square itself.
Their index-two parity subgroup and primitive subgroup with a free permutation quotient provide two finite comparisons of this embedding problem.
The ambient cyclic module in the twenty-two-vertex case agrees with a restriction of the earlier nerve representation; the additional coordinates locate the chosen cycle family inside it.

Supersymmetric-fermion models provide related cohomology counts and spectral comparisons for periodic zig-zag ladders \cite[Section~IV.D]{HuijseLadders}.
Huijse and Schoutens determine graded cohomology dimensions on square tori with periods $(m,-m)$ and $(v_1,v_2)$, where $v_1+v_2=3p>0$ \cite[Theorem~1]{HuijseSchoutens}.
For periods $(5,-5),(1,2)$, the map $(x,y)\mapsto2x-y\pmod {15}$ identifies the quotient graph with $G(15,2)$.
For $G(16,2)$ the periods are $(8,0),(1,2)$.
Huijse's thesis gives quarter-filling cross-cycles in equation~(7.8), but the accompanying footnote excludes completeness at total sizes $8,12,16$ \cite[Section~7.4.1, pp.~122--123]{HuijseThesis}.
In both graph identifications, horizontal translation becomes $\tau^2$, and fermion number $f$ corresponds to simplicial degree $f-1$.
The cited complex cohomology counts therefore do not specify an integral orbit basis for one-step translation; on ordinary dual cohomology the left action corresponding to a homology matrix $T$ is $(T^{-1})^T$.

The seventeen- and eighteen-site systems at fermion number four already occur in the thesis \cite[Figure~7.9, p.~125]{HuijseThesis}.
Huijse, Halverson, Fendley and Schoutens study periodic nearest- and next-nearest-neighbor exclusion chains, including their ground states, momenta and spectral flow \cite[Figure~4 and Table~I]{HuijseCritical}.
Chepiga, Min\'a\v r and Schoutens treat the same exclusion constraint and nearby Hamiltonians, with a periodic four-site product-state pattern \cite[Section~2 and equations~(18)--(19) of the public preprint]{ChepigaLadder}.
Their equation~(10) counts open-boundary ground states and is separate from the periodic calculations here.
Likewise, the periodic zig-zag counting formula (7.5) in the thesis is numerically checked through thirty-four sites, while the accompanying proof has a restricted divisibility hypothesis \cite[Section~7.2.2, pp.~116--117]{HuijseThesis}.
Our cycle-square comparisons identify primitive integer orbit bases and actual dihedral summands, in addition to these dimension and momentum precedents.

The jump-three comparisons use several classical constructions.
Adamaszek and Adams recall the cyclic-graph classification in Theorem~4.3 and give the cross-polytopal maps and integral bases used here in Lemma~8.2 and Proposition~8.4 \cite{RipsCircle}.
After relabeling vertices and transporting the rotation generator, these apply to $\K(5(2\ell+1),3)$.
Section~\ref{sec:rankfour} derives the geometric reflection and distinguishes the two integral types by their order-five coinvariants.
For even vertex counts, Jonsson's bipartite suspension construction is credited in Barmak's public Theorem~3.7; the star-cluster cover follows from his Lemmas~3.2--3.3 \cite{Barmak}.
Nilakantan and Shukla classify the ordinary types of this entire family \cite[Theorem~1.2, $d=4$]{NS}.
Section~\ref{sec:evenfamily} transports the cover-preserving actions through the integral connecting map, then recovers one-step translation from its square, its centralizer and its signed trace.
The resulting integral types agree with the relevant nerve formulas \cite[Theorem~3.6 and Propositions~6.1, 7.2--7.3 of the public preprint]{Nerve}.
The odd sphere branches use Adamaszek's clique classification \cite[Corollary~6.7]{Adamaszek} and a direct signed quotient trace.
The modulo map $i\mapsto i\bmod n$ in Theorem~8.5 of the public nerve preprint \cite{Nerve} intertwines neither the chosen rotations nor the chosen reflections unless $n$ divides $k$, so its ordinary homotopy equivalence alone does not transport these actions.

\subsection{Other complexes of circulant graphs}
Brimkov's clique-number calculation gives the three triangular conditions for degree-four circulants used in Proposition~\ref{prop:flaggraphs} \cite[Section~3.1 of the public conference text]{Brimkov}.
The ordinary branches of our clique comparison follow this face structure; oriented cycles and attaching boundaries then determine their integral actions and quotient maps.
Goyal, Shukla and Singh study clique complexes of line graphs, including four-regular circulant input graphs \cite[Section~4 of the public preprint]{GoyalShuklaSingh}.
Their complex is $\operatorname{Cl}(L(G))$, rather than $\operatorname{Cl}(G)$ or $\Ind(G)$.
Tang and Yau use Fourier symbol matrices and translation eigenspaces for GLMY path homology of circulants \cite{TangYau}.
Their comparison of the path homology of the digraph on $\Z/7\Z$ with steps $1,3$ against the clique homology of the underlying graph distinguishes the two theories, and the calculations over a root-of-unity field do not give the integral triangle-boundary presentation studied here.

\section{Saturation of top cycles in odd matching complexes}
\label{sec:om-saturation}

Let $\mathcal M_n$ be the matching complex of the complete graph on the
labeled vertex set $\{0,\ldots,n-1\}$. Its simplices are collections of
pairwise disjoint edges. Throughout this section we use augmented simplicial chains: $C_{-1}=\Z$ and $\partial_0$ sends every vertex to $1$. Thus $\ker\partial_0$ represents reduced homology when $k=1$. Put $n=2k+1$ and
\begin{align*}
 U^{(k)}&=\mathbb Z^{2k+1},&
 A^{(k)}&=\ker\bigl(\textstyle\sum:U^{(k)}\longrightarrow\mathbb Z\bigr),\\
 Q^{(k)}&=U^{(k)}/\mathbb Z(e_0+\cdots+e_{2k}),&
 S^{(k)}&=\bigwedge^k A^{(k)}.
\end{align*}
The integral identification of $S^{(k)}$ with the hook Specht lattice $(k+1,1^k)$ is classical~\cite[Section~5.5(b), equation~(3)]{DanzHofmannHooks}. The dot product gives a perfect pairing $A^{(k)}\times Q^{(k)}\to\mathbb Z$.
Write $q_i$ for the image of $e_i$ in $Q^{(k)}$ and $f_i=e_i-e_0$ for
$1\leq i\leq2k$; these are dual bases when $i\ne0$.
For an oriented maximum matching
$m=[\{a_1,b_1\},\ldots,\{a_k,b_k\}]$, set
\[
 b_m=(q_{a_1}+q_{b_1})\wedge\cdots\wedge(q_{a_k}+q_{b_k}),\qquad
 \Psi_k(x)=\sum_m\langle x,b_m\rangle m.
\]
The sum uses one orientation for each maximum matching. Reversing that
orientation reverses both factors, so the definition is independent of
the chosen chain basis. All these maps respect the actual permutation
action of $\mathfrak S_{2k+1}$ on the labels.

\begin{lemma}[The integral normalization of the cycle map]
\label{lem:om-normalization}
The map $\Psi_k:S^{(k)}\to C_{k-1}(\mathcal M_{2k+1};\mathbb Z)$ takes
values in cycles and is Bouc's integral hook cycle map~\cite{BoucMatching}, with no
additional integer factor. More precisely, if
$I=\{i_1<\cdots<i_k\}\subset\{1,\ldots,2k\}$ and
$\{r_1<\cdots<r_k\}$ is its complement, then
\[
 \Psi_k(f_{i_1}\wedge\cdots\wedge f_{i_k})
 =\sum_{\pi\in\mathfrak S_{\{0,i_1,\ldots,i_k\}}}
 \operatorname{sgn}(\pi)
 [\{r_1,\pi(i_1)\},\ldots,\{r_k,\pi(i_k)\}].
\]
In particular every displayed basis cycle is primitive in the integer
chain group.
\end{lemma}
\begin{proof}
Fix an oriented $(k-1)$-edge matching $\sigma$, leaving three ground
vertices $a,b,c$. Orient each extension by placing its new edge last.
Up to the common boundary sign, the coefficient of $\sigma$ in
$\partial\Psi_k(x)$ is the pairing with
\[
 b_\sigma\wedge\bigl((q_a+q_b)+(q_b+q_c)+(q_c+q_a)\bigr)
 =2b_\sigma\wedge(q_a+q_b+q_c)=0.
\]
The last equality follows from $\sum_iq_i=0$: the sum of the three
remaining vectors is minus the sum of the edge-sum vectors already
occurring in $b_\sigma$. This also proves the assertion for the
augmented boundary when $k=1$.

For the stated formula, a coefficient is the determinant whose entries
are $\langle f_{i_t},q_a+q_b\rangle$. An edge with both endpoints in the
complement of $\{0\}\cup I$ gives a zero column. If a matching has no
such edge but leaves a vertex of that complement unmatched, it has one
edge entirely in $\{0\}\cup I$ and uses all vertices of $\{0\}\cup I$.
In this case the sum of its determinant columns is zero, since each
$f_{i_t}$ has coordinate sum zero on that set. The only remaining
matchings pair every $r_s$ with a distinct element of $\{0\}\cup I$.
They are indexed uniquely by the displayed permutations, with
$\pi(0)$ unmatched. In the written edge order their determinant is
$\operatorname{sgn}(\pi)$: it is the determinant of $\pi$ on the
augmentation lattice of $\{0\}\cup I$. Thus there are exactly
$(k+1)!$ distinct summands with coefficients $1$ or $-1$, proving both
the normalization and primitivity.
\end{proof}

Bouc's rational top-homology theorem~\cite[Proposition~4, Lemma~8 and Proposition~5]{BoucMatching} identifies the rational extension
of this classical map with the irreducible hook representation
$(k+1,1^k)$. The saturation problem concerns the integral lattice inside this rational representation. Since there are no chains above
degree $k-1$, put
\[
 \mathcal H_k=\widetilde H_{k-1}(\mathcal M_{2k+1};\mathbb Z)
     =\ker\partial_{k-1},\qquad
 \widehat S^{(k)}=\Psi_{k,\mathbb Q}^{-1}(\mathcal H_k),\qquad \Gamma_k=\widehat S^{(k)}/S^{(k)}.
\]
Thus $\Gamma_k$ measures the saturation of the specified integral Bouc
cycles inside the actual top homology. The next theorem determines
its exponent for every odd number of ground vertices.

\begin{theorem}[Uniform exponent of the saturation quotient]
\label{thm:om-exponent}
For every $k\geq1$, $\Gamma_k$ is a finite $2$-group of exponent exactly
\[
                         2^{\lfloor k/2\rfloor}.
\]
Equivalently, this is the least positive integer $d$ for which every
integral top cycle becomes an integral linear combination of the
specified Bouc cycles after multiplication by $d$. More explicitly,
if
\[
 \omega=\sum_{1\leq i<j\leq2k}f_i\wedge f_j,\qquad
 \omega^{[j]}=\frac{\omega^j}{j!},\qquad\omega^{[0]}=1,
\]
then the lattice
\[
 N^{(k)}=\sum_{j=0}^{\lfloor k/2\rfloor}
  2^{-j}\omega^{[j]}\wedge\bigwedge^{k-2j}A^{(k)}
\]
is $\mathfrak S_{2k+1}$-stable and depends only on the canonical
inverse symplectic bivector $\overline\omega$ on $A^{(k)}/2A^{(k)}$, not
on its chosen integral lift. It satisfies
\begin{equation}
\label{eq:om-two-inclusions}
 \sum_{j=0}^{\lfloor k/2\rfloor}
  2^{-j}\omega^{[j]}\wedge\bigwedge^{k-2j}A^{(k)}
 \ \subseteq\ \widehat S^{(k)}\ \subseteq\ 2^{-\lfloor k/2\rfloor}S^{(k)}.
\end{equation}
For each $0\leq j\leq\lfloor k/2\rfloor$, the class of
\[
 2^{-j}\omega^{[j]}\wedge f_1\wedge\cdots\wedge f_{k-2j}
\]
has order exactly $2^j$ in $\Gamma_k$. For $k\geq2$, $\mathcal H_k$ and $S^{(k)}$ are
not isomorphic as integral $\mathfrak S_{2k+1}$-modules.
\end{theorem}

\begin{lemma}[Two-edge lattice identity]
\label{lem:om-two-edge}
Let $V=\mathbb Z^m$, where $m\geq5$, with basis $u_1,\ldots,u_m$.
The subgroup of $\bigwedge^2V$ generated by
\[
 (u_a+u_b)\wedge(u_c+u_d),\qquad a,b,c,d\text{ distinct},
\]
contains $2\bigwedge^2V$.
\end{lemma}
\begin{proof}
In the quotient by this subgroup write $w_{ab}$ for the class of
$u_a\wedge u_b$. Adding the three pairings on four distinct labels,
always putting the pair containing $a$ first, gives
\[
 2(w_{ab}+w_{ac}+w_{ad})=0.
\]
For fixed $a$, comparison of triples among the other $m-1\geq4$
labels shows that $2w_{ab}$ is independent of $b\ne a$; denote it by
$t_a$. The displayed relation gives $3t_a=0$. Skew symmetry gives
$t_a=-t_b$ for distinct $a,b$. Using three distinct labels then gives
$2t_a=0$. Hence $t_a=0$, proving the assertion.
\end{proof}

\begin{proof}[Proof of Theorem~\ref{thm:om-exponent}]
\emph{The upper bound.} We prove the upper inclusion by induction in steps of two. It
is a statement about alternating forms: if a rational alternating
$k$-form $x$ on $Q^{(k)}$ has integer values on every $b_m$, then
$2^{\lfloor k/2\rfloor}x$ is integer-valued on
$\bigwedge^kQ^{(k)}$. For $k=0$ use the one-vertex set and the empty
matching, so the assertion is simply $x\in\mathbb Z$. For $k=1$,
$q_a+q_b=-q_c$ on three labels, so the assertion is immediate.

Suppose $k\geq2$. Choose two disjoint edges $ab,cd$ and let $W$ be
the other $2k-3$ ground labels. Contract $x$ with their edge-sum
vectors and restrict the remaining arguments to $\mathbb Z^W$:
\[
 y(v_1,\ldots,v_{k-2})
   =x(q_a+q_b,q_c+q_d,v_1,\ldots,v_{k-2}).
\]
This descends to $\mathbb Z^W/\mathbb Z\sum_{i\in W}e_i$,
because in $Q^{(k)}$ the sum over $W$ is
$-(q_a+q_b)-(q_c+q_d)$, which contraction annihilates. Its values on
maximum matchings of $W$ are integers: append the two fixed edges to
obtain a maximum matching of the original ground set. The induction
hypothesis therefore gives
\[
 2^{\lfloor(k-2)/2\rfloor}
 x(q_a+q_b,q_c+q_d,q_{i_1},\ldots,q_{i_{k-2}})\in\mathbb Z
\]
whenever all the displayed ground labels are distinct.

Now fix distinct $i_1,\ldots,i_{k-2}$. There are $k+3\geq5$ other
labels. On their free lattice consider the rational alternating
bilinear form obtained by fixing the last $k-2$ arguments of
$2^{\lfloor(k-2)/2\rfloor}x$. Its values on all disjoint two-edge
generators are integers. Lemma~\ref{lem:om-two-edge} shows that
twice this form is integer-valued on every wedge of two basis
vectors. Thus $2^{\lfloor k/2\rfloor}x$ is integer-valued on every
basis wedge of $Q^{(k)}$. The perfect pairing proves the upper inclusion
in~\eqref{eq:om-two-inclusions}.

\emph{Divided-power cycles.} For the lower inclusion, the divided powers $\omega^{[j]}$ are
integral: their coefficients are the principal Pfaffians of the
alternating matrix with all upper-triangular entries equal to one.
Modulo two, $\omega(q_a,q_b)=1$ for every pair of distinct ground
labels, including label zero since $q_0=-\sum_{i>0}q_i$ and $2k$
is even. Consequently
\[
 \omega(q_a+q_b,q_c+q_d)\equiv0\pmod2
\]
for disjoint edges. Evaluate
$\omega^{[j]}\wedge v$, where
$v\in\bigwedge^{k-2j}A^{(k)}$, on $b_m$. The usual exterior product
expansion is a sum of Pfaffians on $2j$ chosen edge-sum arguments,
multiplied by the complementary evaluations of $v$. Each Pfaffian
is a sum of products of $j$ even entries. Every term is therefore
divisible by $2^j$. This proves that
$2^{-j}\omega^{[j]}\wedge v$ lies in $\widehat S^{(k)}$.

\emph{Canonicity.} First note that the Gram matrix of the dot
product on $A^{(k)}/2A^{(k)}$ in the $f_i$ basis has zero diagonal and ones
off the diagonal. Its square is the identity, so
$\overline\omega$ is its invariant inverse bivector. Any other
integral lift, including $g\omega$ for
$g\in\mathfrak S_{2k+1}$, has the form $\omega+2\eta$ with
$\eta\in\bigwedge^2A^{(k)}$. Its divided powers satisfy
\[
 \frac{(\omega+2\eta)^{[j]}}{2^j}
 =\sum_{a+b=j}\frac{\omega^{[a]}}{2^a}\wedge\eta^{[b]}.
\]
Every $\eta^{[b]}$ is integral, again by the Pfaffian formula. Thus
the new lift generates a sublattice of $N^{(k)}$. Interchanging the two
lifts gives equality, proving both lift independence and invariance.

\emph{Sharpness.} The principal Pfaffian on any ordered set of $2j$ indices is one.
Indeed expansion along its first row reduces this to the alternating
sum of $2j-1$ ones, starting and ending with a plus sign. Choose any
$2j$ indices disjoint from $\{1,\ldots,k-2j\}$. The corresponding
wedge coefficient of
$\omega^{[j]}\wedge f_1\wedge\cdots\wedge f_{k-2j}$ is $1$ or
$-1$. Its class after division by $2^j$ therefore has exactly that
order modulo $S^{(k)}$. Taking $j=\lfloor k/2\rfloor$, and using the
upper inclusion, proves the exact exponent.

\emph{Integral nonisomorphism.} The rational hook representation is absolutely irreducible,
so an integral equivariant isomorphism $S^{(k)}\to \mathcal H_k$ would, under
$\Psi_{k,\mathbb Q}$, identify $\widehat S^{(k)}$ with $cS^{(k)}$ for some nonzero
rational $c$. Since $S^{(k)}\subseteq \widehat S^{(k)}$, this forces $c=\pm1/d$ for
a positive integer $d$. Lemma~\ref{lem:om-normalization} shows that
each $\Psi_k(f_{i_1}\wedge\cdots\wedge f_{i_k})$ is primitive even
in the full chain lattice, hence cannot be divided by any $d>1$ in
$\mathcal H_k$. Thus $d=1$. For $k\geq2$ the order-two class already
constructed shows $\widehat S^{(k)}\ne S^{(k)}$, a contradiction.
\end{proof}

\begin{proof}[Proof of Theorem~\ref{thm:matching-main}]
The rational cycle isomorphism identifies $\mathcal H_k/\Psi_k(S^{(k)})$ with $\widehat S^{(k)}/S^{(k)}$.
The exponent and the sharp classes are given by Theorem~\ref{thm:om-exponent}, whose final argument also proves the asserted integral nonisomorphism.
\end{proof}

\begin{corollary}[The first divided-power submodule]
\label{cor:om-first-layer}
For $k\geq2$, the image in $\Gamma_k$ of
$\frac12\omega\wedge\bigwedge^{k-2}A^{(k)}$ is a canonical
$\mathfrak S_{2k+1}$-submodule naturally isomorphic to
\[
 \operatorname{im}\left(
  \overline\omega\wedge-:
  \bigwedge^{k-2}(A^{(k)}/2A^{(k)})\longrightarrow\bigwedge^k(A^{(k)}/2A^{(k)})
 \right).
\]
Its dimension over $\mathbb F_2$ is
\[
 r_k=\sum_{j=1}^{\lfloor k/2\rfloor}
             (-1)^{j-1}\binom{2k}{k-2j}.
\]
\end{corollary}
\begin{proof}
The class of $\frac12\omega\wedge v$ depends precisely on the
reduction of $\omega\wedge v$ modulo $2S^{(k)}$, giving the asserted
image. On $A^{(k)}/2A^{(k)}$ the dot product is a nondegenerate alternating
form: its Gram matrix in the $f_i$ basis has zero diagonal and ones
off the diagonal, and its square is the identity. Thus
$\overline\omega$ is its invariant inverse bivector. This proves
that the image is canonical and equivariant, although the displayed
integral lift $\omega$ depends on the basis.

Choose a symplectic basis $u_i,v_i$, $1\leq i\leq k$. Multiplication
by $\overline\omega=\sum_i u_i\wedge v_i$ defines a differential
of degree two on the exterior algebra, since its square is zero in
characteristic two. This complex is the tensor product of the $k$
complexes $\bigwedge\langle u_i,v_i\rangle$ with differential
$1\mapsto u_i\wedge v_i$. Each factor has homology only in degree
one, of dimension two. The tensor product therefore has homology
only in degree $k$. It is exact in degrees below $k$. Recursively
subtracting the rank of the preceding differential from
$\dim\bigwedge^{k-2}(A^{(k)}/2A^{(k)})$ gives the stated alternating sum.
\end{proof}

The corollary identifies the submodule generated by the first divided
power; it does not identify the whole subgroup $\Gamma_k[2]$ of classes
annihilated by two.

The theorem gives the exact exponent and explicit classes realizing
every indicated power of two. It does not assert equality in the
left inclusion of~\eqref{eq:om-two-inclusions}, or determine all
invariant factors of $\Gamma_k$. Those stronger statements remain open here. In particular, the
elementary $2$-group quotients in the first small examples cannot
persist in this form: for $k\geq4$, $\Gamma_k$ contains an element of
order four. Bouc's cycle construction, rational hook theorem, and
observation that the integral cycle span can be proper are established
in~\cite{BoucMatching}. The exponent formula concerns the actual
homology lattice selected by the matching complex; it does not
classify every integral form of the rational hook representation.

\section{Fixed-point geometry and norm gluing}\label{sec:normgluing}
The integral decomposition into invariant and noninvariant summands has a geometric obstruction.
We first express it through the classical norm map, then apply Tate localization to the fixed-point space.
This is an application of the fixed-point method of Swan and the standard complete-resolution formalism \cite{Swan,BrownGroups,BrownLectures}; the argument is included to specify the hypotheses and the particular quotient being calculated.

\begin{lemma}\label{lem:normgluing}
Let $P=\langle\sigma\rangle$ be cyclic of prime order $p$ and let $L$ be a finitely generated free abelian $P$-module.
Set $N_P=1+\sigma+\cdots+\sigma^{p-1}$, $A_P=L^P$ and $D_P=\ker N_P$.
Then $A_P\cap D_P=0$, and norm induces an isomorphism
\[
 \mathcal G_P(L):=L/(A_P\oplus D_P)\cong N_PL/pA_P.
\]
In particular $\mathcal G_P(L)$ is an elementary abelian $p$-group, with exact sequences
\begin{align}
 0&\longrightarrow\mathcal G_P(L)\longrightarrow A_P/pA_P
 \longrightarrow\widehat H^0(P,L)\longrightarrow0,\label{eq:normexact}\\
 0&\longrightarrow\mathcal G_P(L)\longrightarrow D_P/(\sigma-1)D_P
 \longrightarrow\widehat H^{-1}(P,L)\longrightarrow0.\label{eq:normminusone}
\end{align}
Here $\widehat H^0(P,L)=A_P/N_PL$ and $\widehat H^{-1}(P,L)=D_P/(\sigma-1)L$.
\end{lemma}
\begin{proof}
Norm is multiplication by $p$ on $A_P$, so torsion-freeness gives $A_P\cap D_P=0$.
Since $N_P^2=pN_P$, the identity $px=N_Px+(p-N_P)x$ puts $pL$ in $A_P+D_P$.
The map $[x]\mapsto N_Px\bmod pA_P$ is surjective onto $N_PL/pA_P$.
If $N_Px=pa$ with $a\in A_P$, then $x-a\in D_P$, proving injectivity and \eqref{eq:normexact}.
For \eqref{eq:normminusone}, use $[x]\mapsto(\sigma-1)x\bmod(\sigma-1)D_P$.
If $(\sigma-1)x=(\sigma-1)d$ for $d\in D_P$, then $x-d\in A_P$; its cokernel is $D_P/(\sigma-1)L$.
The Tate-group descriptions are the standard cyclic norm formulas \cite[Sections~6.3--6.4]{BrownLectures}.
\end{proof}

Vanishing of $\widehat H^0(P,L)$ gives the largest possible norm gluing, rather than an integral splitting: then $\mathcal G_P(L)=A_P/pA_P$.
For example this occurs for the regular lattice $\Z[P]$, whose norm gluing is $\Z/p$.

\begin{theorem}[Fixed-point calculation of norm gluing]\label{thm:fixedgluing}
Let a cyclic group $P$ of prime order $p$ act simplicially on a finite complex $X$.
Suppose $\Ht_i(X;\Z)=0$ for $i\ne q$ and $L=\Ht_q(X;\Z)$ is free of rank $h$.
Put $F=|X|^P$ and
\[
 b_0=\sum_{i\equiv q\ (\mathrm{mod}\ 2)}\dim_{\F_p}\Ht_i(F;\F_p),\qquad
 b_1=\sum_{i\not\equiv q\ (\mathrm{mod}\ 2)}\dim_{\F_p}\Ht_i(F;\F_p).
\]
Use augmented reduced homology, so $\Ht_{-1}(\varnothing;\F_p)=\F_p$.
Then
\begin{equation}\label{eq:geometricgluing}
 \mathcal G_P(L)\cong(\Z/p)^g,
 \qquad g=\frac{h-b_0-(p-1)b_1}{p}.
\end{equation}
Moreover $\dim_{\F_p}\widehat H^0(P,L)=b_0$.
If $b_0=0$, every primitive $P$-equivariant quotient $u:L\to\Z$ with trivial target action satisfies $u(L^P)=p\Z$ and has least positive scalar lift $p$.
\end{theorem}
\begin{proof}[Proof of Theorem~\ref{thm:fixedgluing}]
The proof uses the classical Tate localization construction \cite{Swan,BrownGroups}, with complete resolutions as in \cite[Section~6.4]{BrownLectures}.
We give its chain argument, including the empty-fixed-set convention.
After barycentric subdivision, a stabilized simplex is a chain of faces of distinct dimensions, so its stabilizer fixes it pointwise.
Thus $F$ is a subcomplex and, because $|P|=p$, every simplex outside $F$ has trivial stabilizer.
Orient an orbit representative and transport its orientation.
The bounded relative augmented chain complex $\widetilde C_*(X)/\widetilde C_*(F)$ is consequently degreewise free over $\Z[P]$.

Let $\widehat P_*$ be the standard two-periodic complete free resolution of the trivial module.
Place chain degree $i$ in cohomological degree $-i$ and form the total complex of $\operatorname{Hom}_{\Z[P]}(\widehat P_*,\widetilde C_*(X))$.
Each total degree has finitely many chain-degree summands.
Filtering by chain degree shows that the relative total complex is acyclic, since free $\Z[P]$-modules have zero Tate cohomology.
The inclusion of fixed chains therefore induces an isomorphism on total cohomology.
Filtering first by chain homology gives a spectral sequence with terms
$\widehat H^a(P,\Ht_i(X;\Z))$ in total degree $a-i$.
There is only the row $i=q$, so total degree $j-q$ has cohomology $\widehat H^j(P,L)$.

The action on fixed chains is trivial.
Its Tate differential alternates between zero and multiplication by $p$.
Pair resolution columns $2a-1$ and $2a$, whose differential is multiplication by $p$. Each pair is a translate of the cone of $p$ on the fixed chain complex. Since the chain groups are free abelian, this cone is quasi-isomorphic to the quotient complex modulo $p$ in column $2a$. The translates therefore occur in even degrees.
Hence, as $\F_p$-vector spaces,
\begin{equation}\label{eq:tatetotalparity}
 \widehat H^j(P,L)\cong
 \bigoplus_{i\equiv q-j\ (\mathrm{mod}\ 2)}\Ht_i(F;\F_p).
\end{equation}
In particular $\dim\widehat H^0(P,L)=b_0$.

Every nonidentity element of $P$ has fixed space $F$.
The simplicial Lefschetz formula gives trace $(-1)^q\widetilde\chi(F)=b_0-b_1$ on $L\otimes\Q$.
Averaging this character yields
\[
 \operatorname{rank}L^P=\frac{h+(p-1)(b_0-b_1)}p.
\]
Subtracting $b_0$ in \eqref{eq:normexact} proves \eqref{eq:geometricgluing}.
If $b_0=0$, then $L^P=N_PL$.
For a primitive equivariant quotient $u$, one has $u(N_Px)=pu(x)$ and $u(L)=\Z$.
Thus $u(L^P)=p\Z$, which gives both the lower bound and an invariant element lifting $p$.
\end{proof}

For an involution, \eqref{eq:geometricgluing} reads
$2g=h-\sum_i\dim_{\F_2}\Ht_i(F;\F_2)$:
the excess over the fixed-point Betti sum measures this integral gluing.
The single-degree hypothesis is essential to the spectral-sequence collapse.
Composite-order actions can have proper nontrivial stabilizers outside the fixed subcomplex, so the same localization argument does not apply to them.
The quotient $\mathcal G_P(L)$ also differs from the sum-of-all-factor-kernels quotient when more than two rational factors are present.

The same localization argument preserves a commuting group action.
Centralizer-equivariant Smith theory also appears in Treumann's sheaf-theoretic formulation \cite[Section~4.5 of the public preprint]{Treumann}.
The following elementary consequence is useful when a geometric subgroup acts trivially on homology.
\begin{lemma}[Centralizer comparison]\label{lem:centralcomparison}
Let a finite group $G$ act simplicially on a finite complex $X$, and let $K\leq Z(G)$ have prime order $p$.
Suppose $\Ht_*(X;\Z)$ is free and concentrated in degree $q$, with lattice $L$, and $K$ acts trivially on $L$.
Suppose the fixed space $F=|X|^K$ has free reduced integral homology concentrated in degree $d$, with lattice $M$, and $q\equiv d\pmod2$.
Then $L/pL\cong M/pM$ as $\F_p[G/K]$-modules.
For every subgroup $P\leq G/K$ of order $p$, the norm-gluing quotients $\mathcal G_P(L)$ and $\mathcal G_P(M)$ are isomorphic as $\F_p[N_{G/K}(P)]$-modules.
Here $G/K$ acts on $L$ by the homologically trivial $K$ hypothesis and on $M$ by its geometric action on $F$.
\end{lemma}
\begin{proof}
Barycentric subdivision makes $F$ a subcomplex; the relative augmented chain groups outside $F$ are free over $\Z[K]$.
Use the two-periodic complete resolution $\widehat P_*$ of $\Z$ over $\Z[K]$ and the total complex $\operatorname{Hom}_{\Z[K]}(\widehat P_*,\widetilde C_*(X))$, placing chain degree $i$ in cohomological degree $-i$.
Since $G$ centralizes $K$, it acts on coefficients alone, commuting with both resolution differentials and all filtrations.
The inclusion of fixed chains induces a $G$-equivariant isomorphism on total cohomology because the relative complex is degreewise free over $\Z[K]$ and bounded in chain degree.
Filtering by chain homology has just the $q$ row for $X$ and the $d$ row for $F$; it yields, in total degree $-q$, an isomorphism
$\widehat H^0(K,L)\cong\widehat H^{d-q}(K,M)$ compatible with $G$.
Both coefficient lattices have trivial $K$ action, so the left side is $L/pL$, while even periodicity identifies the right side with $M/pM$ without any centralizer twist.
These coefficient actions factor through $G/K$, proving the first assertion.

For any lattice $V$ free over $\Z$ with a $P$-action, its invariant lattice $A=V^P$ is primitive, since it is the kernel of $\sigma-1$ on a torsion-free lattice.
Therefore $A/pA\hookrightarrow V/pV$, and the norm model identifies
$\mathcal G_P(V)=N_PV/pA\cong\operatorname{im}(\overline N_P:V/pV\to V/pV)$.
The comparison $L/pL\cong M/pM$ intertwines the residual norm operators.
Taking their images proves the second assertion; the image is stable under the normalizer since conjugation permutes the summands of the norm.
\end{proof}

\subsection{The clique family}
\begin{theorem}\label{thm:geometricfamily}
Let $s\geq3$ and $L=H_1(\mathcal X_s;\Z)$, with vertex translation $T$.
For $P_3=\langle T^s\rangle$,
\[
 L^{P_3}=N_{P_3}L,\qquad \mathcal G_{P_3}(L)\cong\Z/3.
\]
Every primitive $P_3$-equivariant quotient $u:L\to\Z$, with trivial target action, satisfies $u(L^{P_3})=3\Z$.
Consequently its least positive scalar lift is three.
For every prime $p\ne3$ dividing $s$, set $P_p=\langle T^{3s/p}\rangle$. Then
\[
 \widehat H^0(P_p,L)\cong\F_p,\qquad
 \mathcal G_{P_p}(L)\cong(\Z/p)^{2s/p}.
\]
\end{theorem}
\begin{proof}[Proof of Theorem~\ref{thm:geometricfamily}]
Proposition~\ref{prop:flagfamily} gives $h=2s+1$, $q=1$, and the full triangle list
$\{i,i+s,i+2s\}$, $0\leq i<s$, with no higher faces.
For $P_3=\langle T^s\rangle$, these triangles are the vertex orbits.
A fixed point has invariant support, so the fixed space consists of their $s$ barycenters.
Thus $b_0=0$ and $b_1=s-1$.
Theorem~\ref{thm:fixedgluing} gives $g=1$, $\widehat H^0(P_3,L)=0$, and the asserted image of every primitive scalar quotient.
This also gives a geometric proof of the full dihedral augmentation minimum in Proposition~\ref{prop:flagfamily}.
Indeed $L^{P_3}$ has rank one and maps injectively to the rank-one target; equivariance forces translation and reflection on it to have the target actions, and its generator maps to $\pm3$.

If $p\ne3$ divides $s$, an invariant simplex must contain a full orbit of $p$ vertices.
For $p>3$ there is none, since all cliques have size at most three.
For $p=2$, the orbit is a half-turn pair with difference $3s/2$, which is neither $\pm1$ nor $\pm s$ modulo $3s$.
Again there is no invariant simplex, so the fixed space is empty.
The augmented convention gives $b_0=1$, $b_1=0$, and \eqref{eq:geometricgluing} yields $g=2s/p$ and $\widehat H^0(P_p,L)\cong\F_p$.
\end{proof}

The norm calculation also retains the residual dihedral action, including its nonsemisimple part in characteristic $p$.
\begin{proposition}\label{prop:residualgluing}
Let $s\geq3$, $L=H_1(\mathcal X_s;\Z)$, and $p\ne3$ be a prime dividing $s$.
Put $r=s/p$ and $P=\langle T^{3r}\rangle$.
As a module for the labeled residual group $D_{3s}/P\cong D_{3r}$,
\begin{equation}\label{eq:residualgluing}
 \mathcal G_P(L)\cong\F_p[t]/(1+t^r+t^{2r}),\qquad
 T v(t)=tv(t),\quad Rv(t)=t^{-1}v(t^{-1}).
\end{equation}
Translation has exact order $3r$ on this quotient.
For the order-three subgroup, $\mathcal G_{P_3}(L)\cong\F_3$ has $T=1$ and $R=-1$.
\end{proposition}
\begin{proof}
Use the geometric polynomial presentation in Proposition~\ref{prop:flagfamily} and write $N_j=1+t^j+t^{2j}$.
For $p\ne3$, $K=N_s/N_r$ is monic integral of degree $2r(p-1)$ and
\[
 \gcd((t-1)N_s,t^{3r}-1)=(t-1)N_r.
\]
For example, a root of $t^{3r}-1$ is also a root of $N_s$ exactly when it is a root of $N_r$, since $p$ is coprime to three.
Monic division therefore gives the entire integral invariant lattice $L^P=KL$.
Put $a'=T^{-r(p-1)}K(T)w$.
Its translates identify $L^P$ with the algebraic module $\Z[t]/((t-1)N_r)$, also for $r=1,2$.
Reciprocity $K(t^{-1})=t^{-2r(p-1)}K(t)$ gives $Ra'=-a'$.
Thus this identification has translation $t$ and reflection negative inversion.

The polynomial norm factors as
\[
 N_P(t)=1+t^{3r}+\cdots+t^{3r(p-1)}=K(t)E(t),\qquad
 E(t)=1+t^r+\cdots+t^{r(p-1)}.
\]
Norm commutes with $T$ and $R$, since $T$ centralizes $P$ and $R$ conjugates its generator to its inverse.
Under the resulting equivariant norm model $\mathcal G_P(L)\cong N_PL/pL^P$ from Lemma~\ref{lem:normgluing}, the quotient is consequently the ideal generated by $E$ in $L^P/pL^P$; the shift defining $a'$ changes this generator by a unit only.
In characteristic $p$, $E=(t^r-1)^{p-1}$ is coprime to $N_r$, since $N_r$ is three when $t^r=1$.
Its image in $\F_p[t]/((t-1)N_r)$ is therefore exactly the augmentation ideal $(t-1)$.
Multiplication by $t-1$ identifies that ideal with $\F_p[t]/N_r$.
On the generator $g=(t-1)a'$ of this norm-image model one has $Rg=t^{-1}g$, proving \eqref{eq:residualgluing}.
Here $g$ is a class in $N_PL/pL^P$, transported by the norm isomorphism; its literal class in $L/(L^P+\ker N_P)$ would be zero.

Write $r=p^a u$ with $p\nmid u$.
Then $N_r=N_u^{p^a}$ in $\F_p[t]$, and a primitive $3u$th root is a root of $N_u$ of multiplicity one.
Thus $N_r\mid t^k-1$ requires both $3u\mid k$ and $p^a\mid k$.
Conversely $N_r\mid t^{3r}-1$, proving the exact rotation order.
For $P_3$, its invariant lattice has rank one, with $T=1$ and $R=-1$, and its reduction modulo three is the norm gluing.
\end{proof}

When $p\mid r$, the polynomial in \eqref{eq:residualgluing} has repeated factors.
The description retains the resulting unipotent part of the translation action; the additive group $(\Z/p)^{2r}$ alone does not record it.

\subsection{The independence-complex examples}
For a subgroup acting on graph vertices, form the orbit graph by joining two orbits if an edge joins their vertices, and put a loop on an orbit containing an edge.
The fixed-point space of the independence complex is homeomorphic to the independence complex of this looped orbit graph: invariant barycentric coordinates are constant on each orbit, and the normalized orbit barycenters give the parametrization. Proposition~\ref{prop:orbit-fixed} gives the full coordinate proof for circulants and arbitrary cyclic subgroups.

\begin{corollary}\label{cor:fixedexamples}
For the indicated geometric subgroups, norm gluing is
\[
 \begin{array}{c|c|c|c}
 X& P&\text{Fixed-space type}&\mathcal G_P(\Ht_q(X;\Z))\\\hline
 \K(14,2)&\langle T^2\rangle\cong C_7&\varnothing&\Z/7\\
 \K(14,2)&\langle T^7\rangle\cong C_2&S^1&(\Z/2)^6\\
 \K(18,4)&\langle T^6\rangle\cong C_3&\text{three contractible components}&\Z/3\\
 \K(18,5)&\langle T^9\rangle\cong C_2&\bigvee^7 S^1&(\Z/2)^2\\
 \K(27,5)&\langle T^9\rangle\cong C_3&\bigvee^7 S^1&0
 \end{array}
\]
The product of the first two natural norm quotients identifies the full factor-kernel quotient of $L_{14}$ with $\Z/7\oplus(\Z/2)^6$.
For the two eighteen-vertex rows these are exactly the factor-kernel quotients in \eqref{eq:index3} and \eqref{eq:index2}.
\end{corollary}
\begin{proof}
For $\K(14,2)$, each orbit of $\langle T^2\rangle$ contains an edge of circular difference two, so both orbit vertices are looped and the fixed space is empty.
The half-turn orbit graph is $C_7^2$.
Its independence complex has seven vertices and seven edges of circular distance three, forming one seven-cycle, with no triangles.
Since $(h,q)=(13,2)$, formula~\eqref{eq:geometricgluing} gives $g=(13-6)/7=1$ and $g=(13-1)/2=6$.

For the three remaining rows the orbit graphs are $G(6,4)$, $G(9,5)$ and $G(9,5)$, respectively, with no loops.
The first independence complex consists of the three opposite pairs, hence three disjoint edges.
Multiplication by $2$ identifies the second orbit graph with the square of the nine-cycle.
Its independence complex has nine vertices, eighteen edges and the three triangles $\{i,i+3,i+6\}$, $0\leq i<3$.
These triangles have disjoint edges and can each be collapsed along one edge.
The resulting graph is connected, with nine vertices and fifteen edges, so it is a wedge of seven circles.
The free homology ranks and degrees in Theorem~\ref{thm:main} are $(h,q)=(7,3),(11,3),(7,5)$.
Formula~\eqref{eq:geometricgluing} gives $g=1,2,0$.
In the $\K(18,4)$ row, $L^P=A_3$ and $\ker N_P=D_3$; in the $\K(18,5)$ row, $L^P=A_2$ and $\ker N_P=D_2$, by their displayed rational factors and torsion-freeness.

For the three-factor identification write $L=L_{14}=\Z[t]/(\Phi_2\Phi_7\Phi_{14})$ and $U_i=\ker\Phi_i(T)$.
The actual subgroup invariant lattices are $L^{\langle T^2\rangle}=U_2$ and $L^{\langle T^7\rangle}=U_7$.
The respective norm kernels are $\ker(\Phi_7\Phi_{14})(T)$ and $\ker(\Phi_2\Phi_{14})(T)$.
Thus $U_2+U_7+U_{14}$ is contained in both denominator lattices and gives natural dihedral-equivariant surjections from the full factor quotient $Q$ onto the two norm quotients.
Their product is surjective because the two target orders are coprime: multiplying a preimage by an integer that is one modulo one target exponent and zero modulo the other isolates that component.
The established order $|Q|=7\cdot2^6$ in \eqref{eq:index14} equals the product of these orders, so the product map is an isomorphism.
It identifies the full primary components through the natural quotient maps, not merely by comparing group orders.
\end{proof}

In the last row $T^9$ acts trivially on homology, although the geometric subgroup is nontrivial.
The order-three subgroup of the effective homology group $C_9$ is a different action; it cannot be substituted for this geometric subgroup.
Instead, the residual action on its fixed space supplies a geometric calculation of that factor gluing.

\begin{corollary}[The effective gluing of $\K(27,5)$]\label{cor:effective27}
Put $X=\Ind(G(27,5))$, $L=\Ht_5(X;\Z)$, and $G=C_{27}=\langle T\rangle$.
Let $K=\langle T^9\rangle$, $F=|X|^K$, and $P=\langle T^3K\rangle\leq G/K$.
Then $\mathcal G_P(\Ht_5(X;\Z))\cong\Z/3$.
It is the full two-factor quotient $L/(\ker(T-1)+\ker\Phi_9(T))$.
\end{corollary}
\begin{proof}
Theorem~\ref{thm:main} gives free reduced homology concentrated in degree five, with $\operatorname{rank}L=7$ and $T^9=1$.
Corollary~\ref{cor:fixedexamples} identifies the orbit graph of $K$ as $G(9,5)$ and gives free reduced homology of rank seven in degree one for $F$.
Multiplication by two on its orbit labels identifies this graph with $C_9^2$, and takes $T$ to the square of the usual translation.
Under this labeling, $T^3$ becomes the sixth power of the usual translation; it generates the same subgroup as its third power.
The $P$ fixed space in $F$ consists of the three barycenters of the independent triples $\{i,i+3,i+6\}$.
Theorem~\ref{thm:fixedgluing} applied to $F$ gives $\mathcal G_P(H_1(F;\Z))\cong\Z/3$.
Lemma~\ref{lem:centralcomparison} transfers this norm gluing to $L$, because $5-1$ is even.
Finally the already established rational factors of $L$ are $t-1$ and $\Phi_9$.
For $\sigma=T^3$, its invariant rational subspace is the $t=1$ factor, and its norm is $1+T^3+T^6=\Phi_9(T)$.
Torsion-freeness gives $L^P=\ker(T-1)$ and $\ker N_P=\ker\Phi_9(T)$, identifying the stated quotient.

\end{proof}

The centralizer comparison transfers a module modulo three.
The integral lattice and the geometric reflection in Theorem~\ref{thm:dihedral} still come from the signed-cycle calculation.

\section{An integral action theorem for a cycle-power family}\label{sec:cyclepowerfamily}
Let $C_n^r$ denote the graph whose edges join vertices of circular distance at most $r$, and put
\[
 \mathcal Y_s=\Ind(C_{3s}^{s-1}),\qquad s\geq3.
\]
Its one-skeleton is the complement of $C_{3s}^{s-1}$: its edge differences are $s,\ldots,\lfloor3s/2\rfloor$.
In particular $\mathcal Y_s$ differs from the degree-four clique complex $\mathcal X_s$.
Throughout this section $T(i)=i+1$ and $R(i)=-i$ on the labeled vertices.
For $n\geq3$, define the signed regular lattices
\[
 \Pi_n^\epsilon=\Z[t]/(t^n-1),\qquad
 T p(t)=tp(t),\quad R p(t)=-t^\epsilon p(t^{-1}),\quad \epsilon=0,1.
\]
Also put $N_s(t)=1+t^s+t^{2s}$ and
\[
 M_s=\Z[t]/((t-1)N_s(t)),\qquad R p(t)=-p(t^{-1}).
\]
For $v\geq2$, set
\begin{align*}
 B_v&=\Z[t]/(1-t^v+t^{2v}),& R p(t)&=-p(t^{-1}),\\
 W_v&=\Z[t]/\bigl((t-1)(1+t^v+\cdots+t^{5v})\bigr),& R p(t)&=-t^{-v}p(t^{-1}).
\end{align*}
All these actions satisfy the indicated dihedral relations.
Theorem~\ref{thm:cyclepoweraction} determines the entire integral action, including reflection, for this family.
The ordinary homotopy type used in its proof follows from an elementary triangle collapse.

\begin{proof}[Proof of Theorem~\ref{thm:cyclepoweraction}]
An independent set in $C_{3s}^{s-1}$ has consecutive circular gaps at least $s$.
Consequently the only triangles of $\mathcal Y_s$ are
$\{i,i+s,i+2s\}$, $0\leq i<s$, and there are no higher faces.
Each triangle edge belongs to exactly one triangle.
Choose a directed Hamiltonian edge step $d_0$ as follows:
\[
 d_0=(3s-1)/2\quad(s\text{ odd});\qquad
 d_0=\begin{cases}-(3v-1),&v\text{ even},\\3v-2,&v\text{ odd},\end{cases}
 \quad(s=2v).
\]
Each step is a unit modulo $3s$, is $1$ modulo three, and belongs to an edge orbit different from the triangle-edge orbit.
For the second even case $v\geq3$; when $v$ is even, $v\geq2$ suffices.
Collapsing each triangle along one of its edges preserves this Hamiltonian cycle, so the resulting graph is connected.
There are $3s(s+1)/2$ edges before collapsing and $s$ triangles.
Thus the collapse gives a wedge of $(s-1)(3s-2)/2$ circles.

We now work on the original oriented chains, without requiring the collapse to preserve the action.
Write $e_{d,i}=[i,i+d]$, with subscripts modulo $n=3s$, and let
$u=\sum_{i=0}^{n-1}e_{d_0,i}$.
For each other directed edge step $d$, let $a_d$ be the integer in $\{1,\ldots,n-1\}$ congruent to $d d_0^{-1}$ modulo $n$, and define
\[
 c_{d,i}=e_{d,i}-\sum_{j=0}^{a_d-1}e_{d_0,i+jd_0}.
\]
These are cycles, with
\begin{equation}\label{eq:cyclepowerchords}
 T c_{d,i}=c_{d,i+1},\qquad R c_{d,i}=-c_{d,-i-d},\qquad Ru=-u.
\end{equation}
The full edge orbits give $n$ such cycles apiece.
If $n$ is even, use only $0\leq i<n/2$ for the half-turn orbit.
Together with $u$, these cycles form an integral basis of graph homology: subtracting their coefficients eliminates all non-Hamiltonian edges from any cycle, and the remainder is an integer multiple of $u$.
Independence follows by inspecting those edges, then one Hamiltonian edge.

For the triangle step $d=s$, one has $a_s=s$ because $d_0\equiv1\pmod3$.
Put $c=c_{s,0}$.
The triangle boundaries are therefore the translates of
\begin{equation}\label{eq:cyclepowertriangle}
 N_s(T)c+u.
\end{equation}
This accounts for all attaching relations.
When $s$ is odd, all other edge orbits are full.
Eliminating $u$ in \eqref{eq:cyclepowertriangle}, and using $Tu=u$, gives precisely the relation $(t-1)N_s c=0$.
For $w=t^s c$, \eqref{eq:cyclepowerchords} gives $Rw=-w$.
There are $(s-3)/2$ other non-Hamiltonian edge orbits.
Their generators have reflection $Rc_d=-t^{-d}c_d$; a shift $t^a c_d$ replaces the exponent by $-2a-d$.
Since $n$ is odd, each exponent can be made zero modulo $n$.
This proves the odd decomposition over $\Z$, with the two specified generators of the group.

Suppose now that $s=2v$.
The half-turn chord $h=c_{3v,0}$ satisfies
$c_{3v,i+3v}=-c_{3v,i}-u$; the two base paths together traverse the Hamiltonian cycle once.
Hence, after eliminating $u$, the complete presentation of the special two-generator part is
\begin{equation}\label{eq:cyclepowerevenpresentation}
 (t-1)N_{2v}c=0,\qquad (1+t^{3v})h=N_{2v}c,
 \qquad Rc=-t^{-2v}c,\quad Rh=-t^{-3v}h.
\end{equation}
The relations $t^{6v}c=c$ and $t^{6v}h=h$ follow from these two relations, so no periodicity relation is omitted.
For example, multiplying the second relation by $t^{3v}-1$ gives the relation for $h$, since $t^{3v}-1$ is divisible by $t-1$.

Set $x=t^v$, $A=1+x+x^2$, $B=1-x+x^2$, and $C=1+x$.
Then $N_{2v}=AB$, $1+x^3=BC$, and $A-xC=1$.
The integral change of generators
\[
 z=Ch-Ac,\qquad w=h-xc,
 \qquad c=Cw-z,\quad h=Aw-xz
\]
is invertible over $\Z[t,t^{-1}]$.
The two relations in \eqref{eq:cyclepowerevenpresentation} become
\[
 Bz=0,\qquad (t-1)ABCw=0.
\]
Indeed the second transformed relation follows after subtracting a multiple of $Bz=0$, and conversely it recovers $(t-1)ABc=0$.
This establishes a direct sum, rather than merely annihilating polynomials for two classes.
The reflection formulas in the new generators are
\[
 Rz=-x^{-4}z,\qquad
 Rw=-x^{-1}w+x^{-3}(x-1)z.
\]
Replacing $z$ by $z'=t^{-2v}z$ and $w$ by $w'=w-xz$ gives
$Rz'=-z'$ and $Rw'=-t^{-v}w'$.
For the latter equality, the coefficient of $z$ in $R(w-xz)$ is
\[
 x^{-3}(x-1)+x^{-5}=-(x-1)+x=1\pmod B,
\]
where $x^3=-1$ and $x^6=1$ modulo $B$.
Thus $R(w-xz)=-x^{-1}w+z=-x^{-1}(w-xz)$.
These replacements preserve the direct sum because $B$ divides $(t-1)ABC$.
Since $ABC=1+x+\cdots+x^5$, they give the blocks $B_v$ and $W_v$ with their stated reflections.

The remaining full edge orbits have steps $2v,\ldots,3v-1$, excluding the triangle step and the Hamiltonian step.
The latter has odd absolute step.
The phase in \eqref{eq:cyclepowerchords} can be shifted to zero for an even step and to one for an odd step, producing $\Pi_{6v}^0$ and $\Pi_{6v}^1$.
Counting these steps gives $(a_v,b_v)=(v/2-1,v/2-1)$ for even $v$, and $((v-1)/2,(v-3)/2)$ for odd $v$.
This proves the even decomposition and completes the full integral action calculation.
\end{proof}

\begin{corollary}\label{cor:cyclepowerscalar}
Let $\Z_-$ have $T=1$ and $R=-1$.
For every primitive $D_{3s}$-equivariant quotient $q:H_1(\mathcal Y_s;\Z)\to\Z_-$, the least positive integer $d$ for which $d\operatorname{id}_{\Z_-}$ has an equivariant lift is as follows.
For odd $s$, write $q=\alpha\operatorname{aug}_{M_s}+\sum_j\beta_j\operatorname{aug}_{\Pi_{3s}^0}$; then
\[
 d=3\gcd(\alpha,s).
\]
For $s=2v$, the map vanishes on $B_v$ and has coefficients $\alpha$ on $W_v$ and $\beta_j$ on the regular blocks; then
\[
 d=6\gcd(\alpha,v).
\]
For odd $s\geq5$, the possible minima are exactly $3e$ with $e\mid s$.
For $v\geq3$, they are exactly $6e$ with $e\mid v$.
At $s=3$ and $s=4$ the only minima are three and six, respectively.
\end{corollary}
\begin{proof}
An equivariant scalar map on each displayed cyclic block is a multiple of augmentation; $B_v$ has no such nonzero map since its annihilator is one at $t=1$.
The other annihilators vanish there, and each reflection has negative augmentation, so all the stated maps exist.
Primitivity is exactly $\gcd(\alpha,\beta_1,\ldots)=1$.
The entire $T$-fixed lattice of $M_s$ is generated by $N_s$, whose augmentation is three.
That of $W_v$ is generated by $1+t^v+\cdots+t^{5v}$, whose augmentation is six.
The block $B_v$ has no $T$-fixed element; a regular block has fixed generator $1+t+\cdots+t^{n-1}$, of augmentation $n$.
These assertions follow by monic division of $(t-1)p$ by the relevant annihilator.
Every such fixed generator has reflection $-1$.
Consequently the full dihedral lifting images are
\[
 (3\alpha,3s\beta_1,\ldots)\Z
 \quad\text{and}\quad
 (6\alpha,6v\beta_1,\ldots)\Z,
\]
respectively, where the notation denotes the ideal generated by the entries.
Primitivity reduces their positive generators to the asserted gcd formulas; Bezout gives an element attaining each generator.
With a regular block present, choose $\alpha=e$ and one $\beta_j=1$ to realize every stated divisor value.
At the two exceptional parameters no regular block is present and $\alpha=\pm1$.
\end{proof}

The baseline divisibilities have a fixed-point explanation.
The order-three subgroup $\langle T^s\rangle$ fixes exactly the $s$ triangle barycenters of $\mathcal Y_s$.
For even $s=2v$, the order-two subgroup $\langle T^{3v}\rangle$ fixes exactly the $3v$ half-turn edge barycenters.
There can be no further fixed supports, since an invariant simplex must be a union of vertex orbits and has at most three vertices.
Theorem~\ref{thm:fixedgluing} gives zero $\widehat H^0$ in both cases, forcing every primitive scalar quotient to take the subgroup invariant lattice onto $3\Z$ or $2\Z$.
Thus the minima are multiples of three in the odd case and six in the even case.
The full block decomposition, rather than these divisibilities alone, supplies the gcd dependence and the optimal lifts in Corollary~\ref{cor:cyclepowerscalar}.

\subsection{The two adjacent vertex counts}
The triangle relations simplify at $n=3s+1$ and $3s+2$.
The resulting presentations give Theorem~\ref{thm:cyclepowerneighbors}; the blocks are ordinary graph-cycle lattices with the specified geometric actions.

\begin{proof}[Proof of Theorem~\ref{thm:cyclepowerneighbors}]
The one-skeleton has circular edge distances $s,\ldots,\lfloor n/2\rfloor$.
An independent set has all circular gaps at least $s$, so it has at most three vertices in the stated range.
For $a=1$, the triangles have gaps $(s,s,s+1)$ and form one full orbit of length $n$.
For $a=2$, there are two full orbits: type A has gaps $(s,s,s+2)$ and type B has gaps $(s,s+1,s+1)$.
The three gaps are not all equal, since $3\nmid n$, so none of these orbits has a nontrivial stabilizer.
Every triangle belongs to one of the listed orbits.

Choose the Hamiltonian edge step $d_0=s$ for $a=1$, and $d_0=s+1$ for $a=2$.
Its inverse modulo $n$ is respectively $-3$ or $3$.
Use the cycles $c_{d,i}$ and $u$ constructed in \eqref{eq:cyclepowerchords}, with this new $n,d_0$.
Their integral graph-cycle basis and their translation and reflection formulas have the same proof: elimination of the non-Hamiltonian edge coefficients leaves an integer multiple of $u$.
The half-turn chords, when present, again satisfy
\begin{equation}\label{eq:neighborhalfturn}
 (1+t^m)h=-u,\qquad Rh=-t^{-m}h,\qquad m=n/2.
\end{equation}
These are graph-cycle identities before any triangle is filled.

For $a=1$, the base-path length for step $s+1$ is $n-2$.
The triangle with consecutive steps $s,s,s+1$ therefore has boundary
\[
 c_{s+1,i+2s}+u.
\]
Its translates impose precisely $c_{s+1,j}=-u$ for every $j$.
For $a=2$, the base-path lengths for steps $s$ and $s+2$ are $n-2$ and four.
The boundaries of type B and type A, respectively, are
\begin{equation}\label{eq:neighbortriangles}
 c_{s,i}+u,\qquad
 c_{s,i}+c_{s,i+s}+c_{s+2,i+2s}+2u.
\end{equation}
For the first boundary, the two Hamiltonian edges complete a base path of total length $n$.
For the second, the total length of the three paths is $2(n-2)+4=2n$, giving the coefficient two on $u$.
Thus all attaching relations are equivalent to $c_{s,j}=-u$ and $c_{s+2,j}=0$ for every $j$.
These substitutions are compatible with $T$ and $R$ and have coefficient one, so they eliminate entire chord orbits integrally.
They impose no further relation on the remaining graph-cycle generators.

The ordinary collapse also makes this last assertion transparent at the level of homology groups.
For $a=1$, every edge of distance $s+1$ is in exactly one triangle; collapse each triangle along that edge.
For $a=2$, each distance-$s+2$ edge occurs in exactly one type-A triangle and no type-B triangle.
Collapse all type-A triangles first; then each distance-$s$ edge is in exactly one remaining type-B triangle, and collapse those.
The Hamiltonian edge orbit remains in both cases, so the resulting graph is connected.
There are $n(s+a+1)/2$ edges initially and $an$ triangles.
The graph after collapse has first Betti number
\[
 n(s+a+1)/2-an-n+1=n(s-a-1)/2+1.
\]
It supplies free homology in degree one and no other reduced homology; the preceding chain relations retain the original group action without requiring the collapse to be equivariant.

The remaining full edge-orbit steps are
\[
 s+a+1,\ldots,\lfloor(n-1)/2\rfloor.
\]
If $n$ is odd, $u$ generates $\Z_-$ and each remaining full orbit is a signed regular lattice.
Its reflection phase can be centered to zero, as in the odd case of Theorem~\ref{thm:cyclepoweraction}.
There are $(s-a-1)/2$ such orbits, proving the odd formula.

If $n=2m$ is even, retain the half-turn orbit and eliminate $u$ using \eqref{eq:neighborhalfturn}.
The complete relations on its generator are
$(t-1)(1+t^m)h=0$ and $t^{2m}h=h$.
The second follows from the first, so this is exactly $H_m$, with the reflection in \eqref{eq:neighborhalfturn}.
The other full orbits are independent regular summands; shifting their generators centers the reflection phase to their step parity.
The first remaining step $s+a+1$ is odd, since $n$ is even.
The $k=(s-a-2)/2$ remaining steps therefore contribute $\lceil k/2\rceil$ odd phases and $\lfloor k/2\rfloor$ even phases, as asserted.
The four endpoint cases have no regular summands and are included in these formulas.
\end{proof}

\begin{corollary}\label{cor:neighborscalar}
For a primitive $D_n$-equivariant quotient $q:H_1(\mathcal Z_{s,a};\Z)\to\Z_-$, write $\alpha$ for its coefficient on the first summand of Theorem~\ref{thm:cyclepowerneighbors}, and $\beta_j$ for its augmentation coefficients on the regular summands.
Then the least positive scalar lift is
\[
 d=\begin{cases}
 \gcd(\alpha,n),&n\text{ odd},\\
 2\gcd(\alpha,m),&n=2m.
 \end{cases}
\]
When a regular summand is present, all divisors of $n$ occur in the odd case and all $2e$, $e\mid m$, occur in the even case.
At $(s,a)=(2,1),(3,2)$ only one occurs; at $(3,1),(4,2)$ only two occurs.
\end{corollary}
\begin{proof}
Every equivariant map on a cyclic summand is a multiple of augmentation; on $\Z_-$ it is a scalar.
All these maps have the required reflection action, and primitivity is exactly $\gcd(\alpha,\beta_1,\ldots)=1$.
For odd $n$, the entire $T$-fixed lattice has generators $u$ and the regular norms, so its image is the ideal $(\alpha,n\beta_1,\ldots)$.
For even $n$, monic division gives $H_m^T=\Z(1+t^m)$, with augmentation two; together with the regular norms it gives image $(2\alpha,n\beta_1,\ldots)$.
Every displayed fixed generator has reflection $-1$.
Consequently this image is the full equivariant scalar lifting ideal.
Primitivity gives the two gcd formulas, and Bezout supplies lifts attaining them.
Choosing one $\beta_j=1$ and $\alpha$ equal to a chosen divisor realizes all listed minima when a regular summand is present; otherwise $\alpha=\pm1$.
\end{proof}

For even $n$, the geometric half-turn fixes exactly the $m$ antipodal edge barycenters.
No triangle is fixed, because an invariant support is a union of two-element vertex orbits.
Theorem~\ref{thm:fixedgluing} therefore gives $\widehat H^0(C_2,H_1)=0$ and explains the common factor two in Corollary~\ref{cor:neighborscalar}.
The order-three triangle stabilizers of $\mathcal Y_s$ disappear at the two adjacent vertex counts; the actual presentations determine the resulting change of lattices and lifting ideals.

\section{Orbit traces and low-rank integral actions}\label{sec:framework}
The trace calculation and the integral lattice calculation answer different questions.
The former is a standard application of the alternating chain-trace identity; the latter requires additional arithmetic information.
We first state the trace calculation in a form that also explains the low-rank entries of Theorem~\ref{thm:main}.

\begin{proposition}\label{prop:orbittrace}
Let $G$ be a finite simple graph and $g$ a graph automorphism.
Form a possibly looped quotient graph $Q_g$ on the vertex orbits of $g$: two orbits are adjacent whenever an edge of $G$ joins them, and an orbit has a loop if it contains an edge.
An independent set of a looped graph cannot contain a vertex with a loop.
Then the reduced Lefschetz number of $g$ is
\[
 \sum_i(-1)^i\tr(g_*\mid\Ht_i(\Ind(G);\Q))
 =\widetilde\chi(\Ind(Q_g)).
\]
Here the complex consisting only of the empty face has reduced Euler characteristic $-1$.
If reduced homology is concentrated in degree $q$, its trace is therefore $(-1)^q\widetilde\chi(\Ind(Q_g))$.
\end{proposition}
\begin{proof}
A face fixed setwise by $g$ is a union of vertex orbits, and the fixed independent faces correspond exactly to independent sets of $Q_g$.
Suppose such a face has $m$ vertices in $a$ orbits.
The orientation sign of $g$ on that simplex is $(-1)^{m-a}$.
Its contribution to the alternating simplicial trace is
$(-1)^{m-1}(-1)^{m-a}=(-1)^{a-1}$.
The empty face contributes $-1$ in the augmented complex.
Summing gives the reduced Euler characteristic of the quotient independence complex.
The alternating chain-trace identity gives the same sum on homology.
This identity follows by trace additivity in the cycle and boundary exact sequences: the boundary terms cancel between adjacent degrees.
\end{proof}

For $G(n,s)$ and $g=\tau^k$, put $d=\gcd(n,k)$ and write $Q_k=Q_g$.
Its vertices are the residue classes modulo $d$.
For one-step translation, $d=1$ and $Q_1$ has a loop, so the reduced Lefschetz number is $-1$.
More generally, $Q_k$ has a loop at every vertex if $d$ divides either $1$ or $s$.
Otherwise it is the simple Cayley graph with the nonzero jumps $\pm1,\pm s$ reduced modulo $d$, with repeated edges removed.
Thus quotient independence complexes determine the rational rotation character, but need not determine an integral lattice.

\begin{proposition}\label{prop:lowrank}
Let $G$ be a simple graph on $\Z/n\Z$ with at least one edge, invariant under one-step translation.
Suppose $\Ht_*(\Ind(G);\Z)$ is free, concentrated in one degree $q>0$, and has rank $r\in\{1,2\}$.
The integral translation module is determined as follows:
\[
\begin{array}{c|c|c}
 r & q & \text{module}\\ \hline
 1 & q\text{ odd} & \Z[t]/(t-1)\\
 1 & q\text{ even} & \Z[t]/(t+1)\\
 2 & q\text{ even} & \Z[t]/(t^2+t+1)\\
 2 & q\text{ odd} & \Z[t]/(t^2-t+1).
\end{array}
\]
The action orders in the last three rows are respectively $2,3,6$, and hence divide $n$.
\end{proposition}
\begin{proof}
No nonempty independent face is fixed setwise by one-step translation: the only nonempty invariant vertex subset is the full vertex set, which contains an edge.
The reduced alternating chain trace is $-1$.
Hence the trace $\epsilon$ in degree $q$ is $(-1)^{q+1}$.
In rank one this directly gives multiplication by $\epsilon$.

In rank two let $T$ be the finite-order integral action matrix.
Its determinant cannot be $-1$: a real finite-order matrix of dimension two with negative determinant has eigenvalues $1,-1$ and trace zero.
Thus $\det T=1$, and the characteristic polynomial is $t^2-\epsilon t+1$.
It is $\Phi_3$ or $\Phi_6$, and is irreducible over $\Q$.
Cayley--Hamilton makes the lattice a torsion-free rank-one module over the corresponding quadratic ring.
Torsion-freeness follows by embedding in the rationalization, where every nonzero ring element acts invertibly.
Both rings are the Eisenstein ring $\Z[\zeta_3]=\Z[\zeta_6]$, up to the chosen generator.
This ring is Euclidean: rounding the two coefficients in the basis $1,\zeta_3$ leaves a remainder $a+b\zeta_3$ with $|a|,|b|\le1/2$ and norm $a^2-ab+b^2\le3/4<1$.
Consequently a finitely generated torsion-free rank-one module over it is free.
The displayed polynomials give the exact orders, which must divide $n$.
\end{proof}

These propositions assemble standard trace and quadratic-ring arguments.
They require neither an invariant matching nor an equivariant wedge decomposition.
For larger rank we will compute the geometric action on actual integral cycles; knowing all rotation traces alone would still leave the gluing of rational factors unresolved.

\section{Local blocks and periodic Morse reductions}\label{lb-section}

The ordinary reductions used in the action calculations admit a local formulation.
For a graph with bounded edge differences, we isolate the matching near the ends of a block and give a finite condition under which the block can be duplicated.
The resulting criterion propagates critical faces and their dimensions through an infinite family.
We then prove the eight-periodic reduction for the jumps $\{2,3\}$, including the complete finite part of the argument.

\subsection{Ordered toggles}
Let $K$ be a finite simplicial complex with ordered vertices
$v_0,\ldots,v_{N-1}$.  Starting with the nonempty faces, at stage $i$ match
each available pair $\sigma,\sigma\cup\{v_i\}$ with $v_i\notin\sigma$.
Delete both faces from the available set before the next stage.  Pairs
formed at the same stage are disjoint.  If $a_i$ is the indicator of the
available set after stage $i$, extended by zero to all other vertex subsets,
then
\begin{equation}\label{lb-toggle-recurrence}
 a_i(x)=a_{i-1}(x)\bigl(1-a_{i-1}(x\mathbin\triangle\{v_i\})\bigr).
\end{equation}
Initially $a_{-1}$ is the indicator of $K\setminus\{\varnothing\}$.

\begin{lemma}\label{lb-acyclic}
The ordered-toggle matching is acyclic.  If its critical faces are one
vertex and $m$ faces of a single dimension $d>0$, then
$K\simeq\bigvee^m S^d$.
\end{lemma}
\begin{proof}
Orient matched incidences upward and all other incidences downward.
A directed cycle would have equally many upward and downward edges.
Two upward edges cannot be consecutive, since a face belongs to at most
one matching pair; hence the edges of such a cycle alternate.  Choose the
smallest pivot $v_i$ used by an upward edge of the cycle and start just
after that edge.  Eventually a downward edge removes $v_i$, say
$\tau\longrightarrow\tau\setminus\{v_i\}$.
The upward edge immediately before it has a different pivot, because a
matched edge is not oriented downward.  That pivot is therefore later
than $v_i$.  The upward edge immediately after it likewise has a different
pivot, since otherwise it would reverse the same downward edge.
Thus both $\tau$ and $\tau\setminus\{v_i\}$ were available at stage $i$.
They would have been paired then, a contradiction.

Discrete Morse theory~\cite{Forman} gives a CW model with precisely the
critical cells.  Under the stated profile it has one vertex and $m$
$d$-cells.  For $d=1$ this is a wedge of circles.  For $d>1$ every attaching
map lands in the unique vertex, so the model is a wedge of $d$-spheres.
\end{proof}

\subsection{A finite criterion}
Let $J$ be a nonempty finite set of positive integers, let $w=\max J$,
and put
\[
 G_n(J)=\Cay(\Z/n\Z,\{\pm j:j\in J\}).
\]
Fix $p\geq w$.  A local word $P\mid R\mid L$ has respective lengths
$p,w,w$, with integer coordinates
\[
 P:\ 0,\ldots,p-1,\qquad R:\ p,\ldots,p+w-1,
 \qquad L:\ -w,\ldots,-1.
\]
Within each block, bits are written in increasing coordinate order.
Two coordinates are adjacent when their absolute difference lies in $J$.
For $z\in\{0,1\}$, start with all independent local words, omitting the
zero word exactly when $z=0$, and perform the toggles at $0,\ldots,p-1$.
Write $\mathcal L_p(J,z)$ for the survivors.  Define
\[
 \Delta_p(P\mid R\mid L)=PP\mid R\mid L,
\]
with one exception: when $z=0$, the singleton at coordinate zero maps to
the singleton at coordinate zero, rather than to a two-element word.
On cyclic words the same notation duplicates the first $p$ positions and
shifts all subsequent positions by $p$, with the same singleton exception.

\begin{theorem}[Local block criterion]\label{lb-criterion}
Assume the two finite identities
\begin{equation}\label{lb-local-identity}
 \mathcal L_{2p}(J,z)=\Delta_p\mathcal L_p(J,z)
 \qquad(z=0,1).
\end{equation}
If $\mathcal C(n,J)$ is the critical set of the natural-order matching on
$\Ind(G_n(J))\setminus\{\varnothing\}$, then for every $n\geq p+3w$,
\[
 \mathcal C(n+p,J)=\Delta_p\mathcal C(n,J).
\]
In general, duplication raises the dimension of a positive critical face by the number of its selected vertices in the first block.
Suppose moreover that the critical profile at $n$ consists of the vertex
$\{0\}$ and $m$ faces of a single positive dimension $d$, each containing
exactly $t$ vertices among $0,\ldots,p-1$.  Then
\[
 \Ind(G_{n+p}(J))\simeq\bigvee^m S^{d+t}
 \simeq\Sigma^t\Ind(G_n(J)).
\]
\end{theorem}
\begin{proof}
Write a cyclic word as $P\mid R\mid Y\mid L$, where $R,L$ have length
$w$.  The inequalities $p\geq w$ and $n\geq p+3w$ imply $|Y|\geq w$.
There is no edge from $R$ to $L$: both arcs between these boundary blocks
have more than $w$ edges.  Every edge meeting $P$ is represented in the
displayed local graph, and $P$ has no edge to $Y$.

Fix $Y$.  Independence inside $Y$ and compatibility of $Y$ with $R,L$
give a Boolean condition $A_Y(R,L)$ independent of $P$.
On a boundary choice with $A_Y=0$ there are no global faces.  On a choice
with $A_Y=1$, the initial global indicator as a function of $P$ is precisely
the local indicator with $z=1$ when $Y$ is nonempty and $z=0$ otherwise.
Since the first $p$ toggles alter only $P$, the same description holds at
every stage of \eqref{lb-toggle-recurrence}.  Duplicating $P$ leaves
$Y,R,L$ and their compatibility condition unchanged.  Equation
\eqref{lb-local-identity} therefore gives a bijection between the global
survivors after $p$ old pivots and after $2p$ new pivots.

After the first pivot, every surviving face containing zero is the singleton
$\{0\}$: a larger such face has already been paired with its nonempty
deletion of zero.  The singleton has no surviving partner at any subsequent
stage.  All other survivors omit zero, and the ordinary duplication map
commutes with every later toggle:
\[
 \Delta_p(\sigma\mathbin\triangle\{j\})
 =\Delta_p(\sigma)\mathbin\triangle\{j+p\},\qquad j\geq p.
\]
The inverse on the image deletes the second copy of $P$ and shifts the
remaining positions back.  Thus two survivors are toggle partners at $j$
if and only if their images are partners at $j+p$.  Induction through the
remaining pivots proves the asserted critical-set equality.  Each positive
critical face acquires exactly $t$ vertices.  Lemma~\ref{lb-acyclic}
then gives both wedge descriptions and the suspension equivalence.
\end{proof}

The criterion is sufficient: it does not assert the existence of a suitable
$p$ for every jump set.  Its finite hypotheses use $p+2w$ and $2p+2w$
Boolean coordinates, irrespective of $n$.
For comparison, $J=\{1\}$ and $p=3$ give the classical three-periodic suspension pattern for cycle independence complexes~\cite[Proposition~11.17, pp.~194--195]{Kozlov2008}.

\subsection{The jumps two and three}
The following specialization includes the finite core of the proof, so that
the infinite family does not depend on an unspecified matching table.

\begin{corollary}[Eight-periodic family]\label{lb-eight-periodic}
Let $X_n=\Ind(\Cay(\Z/n\Z,\{\pm2,\pm3\}))$ and write
$n=8q+r\geq6$, with $0\leq r\leq7$.  Then
\[
 X_n\simeq
 \begin{cases}
 \displaystyle\bigvee^5 S^{2q-1},&r=0,\\
 S^{2q-1},&r=1,2,3,\\
 \displaystyle\bigvee^3 S^{2q},&r=4,\\
 S^{2q+1},&r=5,6,7.
 \end{cases}
\]
In particular $X_{n+8}\simeq\Sigma^2X_n$ for every $n\geq6$.
\end{corollary}
\begin{proof}
Use $J=\{2,3\}$, $w=3$ and $p=8$ in Theorem~\ref{lb-criterion}.
For clarity we specify the whole finite calculation.  For any of the local
graphs $H$ in that theorem, initialize the indicator on every binary word by
\begin{equation}\label{lb-initial-indicator}
 a_{-1}(x)=
 \left(\prod_{\{u,v\}\in E(H)}(1-x_ux_v)\right)
 \left(1-(1-z)\prod_u(1-x_u)\right).
\end{equation}
Apply \eqref{lb-toggle-recurrence} at the processed positions in order.
The resulting fourteen-coordinate survivor table is
\begin{center}
\begin{tabular}{@{}lll@{}}
\toprule
$P$&allowed $R$&allowed $L$\\
\midrule
$00011000$&$000,100,010,110,001,011$&$000$\\
$00001100$&$000,010,001,011$&$100$\\
$00000110$&$000,001$&$010,110$\\
$00010001$&$000,100$&$001$\\
$00000011$&$000$&$011$\\
\bottomrule
\end{tabular}
\end{center}
Each row represents all combinations in its last two columns, for a total
of seventeen words.  For $z=0$ add
$10000000\mid000\mid000$; for $z=1$ add no word.
Applying the same two formulas to the twenty-two-coordinate graph, with
pivots $0,\ldots,15$, gives precisely this table with $P$ replaced by $PP$,
and the unchanged singleton when $z=0$.  These are equalities of complete
survivor sets, not just of their cardinalities.  They are the finite
identities \eqref{lb-local-identity}.

Here is also the complete base specification.  For $6\leq n\leq24$, use
\eqref{lb-initial-indicator} on the cyclic graph with $z=0$ and apply all
$n$ pivots.  Besides $\{0\}$, its survivors are the words in
\eqref{lb-critical-words} below, with $n=8q+r$.  Put
\[
 B_1=00011000,\quad B_2=00001100,\quad B_3=00000110,
 \quad B_4=00010001,\quad B_5=00000011.
\]
Using concatenation and word powers, the positive critical words are
\begin{equation}\label{lb-critical-words}
\begin{array}{c|l}
r&\text{positive critical words}\\ \hline
0&B_j^q\quad(1\leq j\leq5)\\
1,2,3&B_1^q0^r\\
4&B_2^q0100,\ B_3^q0010,\ B_4^q0001\\
5&B_5^q00011\\
6&B_5^q000011\\
7&B_5^q0000011.
\end{array}
\end{equation}
The empty power is allowed for $n=6,7$.  Equivalently, their
multiplicity--dimension pairs in the full base window are
\[
\begin{array}{c|rrrrrrrrrr}
n&6&7&8&9&10&11&12&13&14&15\\ \hline
m&1&1&5&1&1&1&3&1&1&1\\
d&1&1&1&1&1&1&2&3&3&3
\end{array}
\]
\[
\begin{array}{c|rrrrrrrrr}
n&16&17&18&19&20&21&22&23&24\\ \hline
m&5&1&1&1&3&1&1&1&5\\
d&3&3&3&3&4&5&5&5&5.
\end{array}
\]
Equations \eqref{lb-initial-indicator} and \eqref{lb-toggle-recurrence}
specify each entry using only Boolean multiplication and subtraction on a
fixed finite domain. The complete local survivor sets and the base-word
comparisons are the finite part of the proof; they are specified by the
displayed graph, initial indicator and toggle recurrence.

The generic threshold is $p+3w=17$.  The eight base cases
$17,\ldots,24$ meet every residue class modulo eight, and each positive
critical word begins with a block having two selected positions.  Theorem
\ref{lb-criterion} therefore duplicates its initial block, preserves the
multiplicity and raises the dimension by two.  Induction proves
\eqref{lb-critical-words} for every $n\geq25$; the finite base window
covers all smaller $n\geq6$.  Lemma~\ref{lb-acyclic} gives the displayed
homotopy types.  Comparing the formulas at $n$ and $n+8$ gives the claimed
suspension equivalence, including $6\leq n\leq16$ below the generic
threshold.
\end{proof}

\begin{corollary}\label{cor:p687-family}
Multiplication by $3$ modulo $14$ identifies the jumps
$\{\pm1,\pm4\}$ with $\{\pm2,\pm3\}$. Hence
$\K(14,4)\simeq X_{14}\simeq S^3$.
\end{corollary}
\begin{proof}
The image of $\{\pm1,\pm4\}$ is $\{\pm3,\pm12\}=\{\pm2,\pm3\}$.
Apply Corollary~\ref{lb-eight-periodic} at $14=8+6$.
\end{proof}
This relabeling identifies the ordinary complexes; its effect on a labeled
translation must be retained when computing actions.

\section{A finite-state description for every bounded jump set}
\label{sec:finite-state}

For a nonempty finite positive jump set $J$, write
$G_n(J)=\Cay(\Z/n\Z,\{\pm j:j\in J\})$ with vertices
$0,\ldots,n-1$ in their natural order.  The local duplication identity is sufficient for a particularly simple
recurrence.  The ordered matching has a finite-state description even when
that identity fails.  This description counts critical faces; it does not
compute their attaching maps or, in general, their homology.
Finite automata for Morse critical cells and rational Morse-number series
already occur in the work of Hersh and Welker on affine-semigroup
posets~\cite[Theorem~1.2 and Sections~6,8]{HershWelker2003}.
Here we construct a Boolean frontier directly for ordered-toggle matchings
on every fixed bounded-jump circulant family.

\begin{theorem}[Rational critical-face series]
\label{thm:rational-critical-series}
Fix a nonempty finite set $J$ of positive jumps and put $w=\max J$.
For $n>w$, let $\mathcal C(n,J)$ be the nonempty critical faces of the
natural ordered-toggle matching on $\Ind(G_n(J))$.  The binary words of
these faces, over all $n>w$, form a regular language.  Their bivariate
series
\[
 F_J(x,y)=\sum_{n>w}x^n\sum_{\sigma\in\mathcal C(n,J)}y^{|\sigma|}
\]
is rational.  More precisely, there are finite matrices and vectors over
$\Z[y]$, depending only on $J$, such that
\[
 F_J(x,y)=E_J(x,y)+\frac{yx^{3w}}{1-x}
       +x^{3w}U_J(y)(I-xM_J(y))^{-1}V_J(y),
\]
where $E_J$ accounts exactly for $w<n<3w$.  One may use at most
$2^{2^{2w}}$ frontier states.  In particular, the critical-face polynomials
satisfy a constant-coefficient linear recurrence over $\Z[y]$ for all
sufficiently large $n$.
\end{theorem}

\begin{proof}
We first include the empty face in the matching.  At pivot zero this pairs
$\varnothing$ with $\{0\}$.  Removing that one pair instead leaves
$\{0\}$ critical, and does not change any other subsequent pair: every
two-element partner of $\{0\}$ was already paired at pivot zero, or is
not independent.  Thus the nonempty-face matching has precisely the same
other critical faces as the augmented matching, with the extra singleton
$\{0\}$.  This explains the term $yx^{3w}/(1-x)$ in the large-$n$ series.

Assume $n\ge3w$.  At the stage immediately before pivot $i$, fix the
already read prefix $b_0\cdots b_{i-1}$.  For $0\le i\le n-2w$, the
unprocessed suffix has at least $2w$ vertices.  Let $q$ be its first $w$
bits, and let $v$ be its last $w$ bits.  These blocks are disjoint.
The survivor indicator on the suffix is
\begin{equation}\label{eq:frontier-factor}
 I_i(s) f_i(q,v),
\end{equation}
where $I_i$ is the independence indicator of the subgraph induced by the
unprocessed suffix and $f_i:\{0,1\}^{2w}\to\{0,1\}$ is a Boolean
function.  Initially $f_0=1$.

For $i<n-2w$, write the actual bit at pivot $i$ as $b$, and the next
$w$ bits as $q=(q_0,\ldots,q_{w-1})$.  For a possible bit $c$, its
compatibility with future vertices is
\[
 c_i(c,q,v)=\prod_{j\in J}(1-cq_{j-1})
 \prod_{\substack{j\in J\\j>i}}(1-cv_{w+i-j}).
\]
The second product represents the cyclic edges to the terminal boundary;
it is empty for $i\ge w$.  The two evaluations before the toggle use
\[
 h_c(q,v)=f_i(c,q_0,\ldots,q_{w-2},v)c_i(c,q,v).
\]
For $w=1$, the middle list is empty.  The survivor recurrence gives
\begin{equation}\label{eq:frontier-transition}
 f_{i+1}(q,v)=h_b(q,v)\bigl(1-h_{1-b}(q,v)\bigr).
\end{equation}
Indeed, the induced-suffix independence indicator after removing vertex
$i$ is common to both toggle variants.  If it is zero, both vanish; if it
is one, the two variants are exactly $h_b$ and $h_{1-b}$.  This proves
\eqref{eq:frontier-factor} and \eqref{eq:frontier-transition} by induction.
Edges to earlier vertices have already entered $f_i$, so none is omitted.

The possible Boolean functions constitute a finite state set of cardinality
$D=2^{2^{2w}}$.  The first $w$ pivots use the explicit, time-dependent
transitions above.  All later pivots before $n-2w$ use the same two
transitions, indexed by $b=0,1$, since the wraparound product is then
empty.  Let $M_J(y)$ give weight $y^b$ to the transition reading bit $b$,
adding the weights if the two bits yield the same state.  Let $U_J(y)$
be the row vector obtained from $f_0=1$ after the first $w$ transitions,
with the same weights.

At $i=n-2w$, exactly $2w$ bits remain.  Their induced graph has only
the ordinary line-distance edges in $J$: the shortest possible wraparound
distance is $n-(2w-1)>w$.  For each state $f$, initialize the finite table
by $I_{n-2w}(s)f(s)$ and perform the final $2w$ toggles.  Let $V_J(y)_f$
be the sum of $y^{|s|}$ over its surviving terminal words.  This terminal
rule depends on $J$ and $f$ but not on $n$.  Thus the augmented critical-face
polynomial is exactly
\[
 U_J(y) M_J(y)^{n-3w} V_J(y).
\]
Summing this expression over $n\ge3w$ proves the displayed rational
formula.  The omitted lengths are a finite polynomial $E_J$.

The same construction is a finite automaton: read the first $w$ bits with
the initial transitions, read any number of bits with the common
transitions, and accept a terminal block of length $2w$ according to the
finite rule.  A finite automaton can guess where this terminal block begins;
its prescribed length makes the split unique for every accepted word.
The extra singleton language and the finitely many short-length words are
regular as well.  Finally, the characteristic polynomial of
$\operatorname{diag}(M_J(y),1)$
gives the stated eventual linear recurrence by Cayley--Hamilton.
\end{proof}

The bound on states is deliberately uniform rather than practical.  It gives
a finite construction for each fixed jump set without requiring any block
duplication identity.  A rational critical-face series alone does not imply
periodic homotopy types: Morse differentials and dimensions still matter.

\section{Fixed spaces and the obstruction to equivariant periodicity}
\label{sec:fixed-spaces}

Put $Y_n=|\Ind(G_n(\{2,3\}))|$ for $n\geq4$; for $n\geq6$ this is $|X_n|$ from Corollary~\ref{lb-eight-periodic}.  For a divisor $m$ of $n$, let $P_m^{(n)}$ be the
subgroup of $C_n$ generated by translation through $g=n/m$ vertices.
Thus $P_m^{(n)}$ has order $m$, and its vertex orbits are the residue classes
modulo $g$.  Superscripts denote fixed spaces in the geometric realization,
not subcomplexes on fixed vertices.

\begin{proposition}[Orbit-coordinate fixed spaces]
\label{prop:orbit-fixed}
Let $J$ be a nonempty finite positive jump set and assume $n>\max J$.  For
$m\mid n$ and $g=n/m$, form the graph on $\Z/g\Z$ by reducing the jumps
in $J$ modulo $g$, retaining a loop when a reduced jump is zero.
Then the $P_m^{(n)}$-fixed space of $|\Ind(G_n(J))|$ is homeomorphic to the
independence complex of this graph.  In particular, the fixed space is empty
if $g$ divides any member of $J$.
\end{proposition}

\begin{proof}
Represent a point by nonnegative vertex coordinates $(\lambda_v)$ with sum
one and independent support.  The point is $P_m^{(n)}$-fixed precisely when its
coordinates are constant on each orbit.  Set
$\mu_a=\sum_{v\equiv a\pmod g}\lambda_v$.  The inverse map assigns
$\lambda_v=\mu_a/m$ on that orbit.  An orbit contains an internal edge
precisely when $g$ divides a jump; two distinct orbits have an edge between
them precisely when their residues differ by a reduced jump.  Consequently
these two linear maps restrict to mutually inverse homeomorphisms of the
stated realizations.  A zero reduced jump produces a loop at every residue,
so there is no nonempty independent support in that case.
\end{proof}

\begin{theorem}[Complete cyclic fixed-space profile]
\label{thm:fixed-profile}
For every $n\ge6$, every $m\mid n$, and $g=n/m$,
\[
 Y_n^{P_m^{(n)}}\cong
 \begin{cases}
  \varnothing,&g=1,2,3,\\
  \text{four points},&g=4,\\
  S^1,&g=5,\\
  Y_g,&g\ge6,
 \end{cases}
\]
where the $g=5$ identification is a homeomorphism and the homotopy type of
$Y_g$ is given by Corollary~\ref{lb-eight-periodic}.  The full translation action
on $Y_n$ is free precisely when $n=6$, $n=9$, or $n$ is prime.
\end{theorem}

\begin{proof}
Proposition~\ref{prop:orbit-fixed} applies to $J=\{2,3\}$.  Its quotient
has a loop exactly for $g=1,2,3$.  For $g=4$ it is $K_4$, and for $g=5$
it is a five-cycle whose independence complex is another five-cycle.  For
$g\ge6$ it is $G_g(\{2,3\})$.  The action is free if and only if every nontrivial
subgroup has empty fixed space, equivalently every proper divisor of $n$
is at most three.  For composite $n\ge6$, let $p$ be its least prime
divisor.  This condition gives $n/p\le3$ and $p\le n/p$, whence
$p\le3$ and $n\le9$.  The composite values satisfying the condition are
exactly $6$ and $9$.  Prime $n$ has only the proper divisor one.
\end{proof}

For a real representation $V$ of a finite group $H$, write
$\Sigma^V X=X*S(V)$, where $S(V)$ is its unit sphere.  The zero-dimensional
representation means no suspension.  When $X^H$ is nonempty, on $H$-fixed spaces this construction
is the ordinary $\dim V^H$-fold suspension of $X^H$; if $V^H=0$, it is
$X^H$ itself.  In particular, it preserves the total rank of reduced
homology of a nonempty fixed space.

\begin{theorem}[No uniform equivariant suspension period]
\label{thm:no-equivariant-period}
Fix any positive multiple $p$ of eight, write $p=2^a b$ with $b$ odd, and
put $m=2^a$.  For every positive integer $q$, set $n=8mq$ and identify
the order-$m$ translation subgroups of $C_n$ and $C_{n+p}$ with the same
abstract group $H=C_m$.  There is no $H$-equivariant homotopy equivalence
\[
 Y_{n+p}\simeq_H\Sigma^V Y_n
\]
for any real $H$-representation $V$.  Nevertheless,
$Y_{n+p}\simeq\Sigma^{p/4}Y_n$ nonequivariantly.
\end{theorem}

\begin{proof}
The two fixed spaces are $Y_{8q}$ and $Y_{8q+b}$ by
Theorem~\ref{thm:fixed-profile}.  The first has total reduced Betti number
five.  The second has total reduced Betti number one, because $b$ is odd.
Taking fixed spaces of an equivariant homotopy equivalence would preserve
these numbers, and a representation suspension cannot change them.  This
contradiction holds for every $V$, even if it depends on $q$.
Iterating the ordinary eight-periodic equivalence gives the last assertion.
\end{proof}

Any uniform ordinary suspension period must be divisible by eight:
Corollary~\ref{lb-eight-periodic} gives total reduced Betti number five precisely
at indices divisible by eight, and suspension preserves that number.
Applying a proposed period to arbitrarily large such indices forces its
divisibility by eight.  Hence the preceding theorem covers every positive
uniform ordinary period.

There is also a small-subgroup obstruction to the original step itself.
For every $n=8q$, identify the two half-turn subgroups with $C_2$.
Their fixed spaces are $Y_{4q}$ and $Y_{4q+4}$.  For $q\ge2$ their total
reduced Betti numbers alternate between five and three; for $q=1$ they
are four points and $Y_8$, with total reduced Betti numbers three and five.
Thus the eight-step recurrence fails to be a representation suspension even
after restriction to the half-turn.  Larger ordinary periods cannot remove
the obstruction for all common cyclic subgroups, as the theorem shows.

\section{Three finite benchmarks}\label{sec:benchmarks}

The following reductions connect the periodic family to the finite integral-action calculations. The cyclic augmentation lattice on thirteen vertices will also provide a small model for the distinction between an ordinary homotopy type and its symmetry.

\begin{theorem}[Three fixed chorded cycles]\label{thm:three-benchmarks}
There are homotopy equivalences
\[
 \Ind(G(13,2))\simeq\bigvee^{12}S^2,
 \qquad
 \Ind(G(14,4))\simeq S^3,
 \qquad
 \Ind(G(14,5))\simeq S^2.
\]
The associated face vectors are, respectively,
\[
 (13,52,65,13),\qquad
 (14,63,112,77,14),\qquad
 (14,63,112,91,42,14,2).
\]
All reduced integral homology groups are torsion-free.
\end{theorem}

The proof is given in the next two subsections. Their translation actions are determined in Theorem~\ref{thm:cyclic-action}.

\subsection{The thirteen-vertex square of a cycle}\label{sec:p683}

Set $K=\Ind(G(13,2))$.  If the vertices of an independent set are listed
cyclically, every gap between consecutive selected residues is at least three.
Conversely, this gap condition is sufficient.  Up to rotation, the nonempty
faces are exactly those in Table~\ref{tab:p683orbits}; all entries are read
modulo $13$.

\begin{table}[ht]
\centering
\caption{Cyclic face orbits for $\Ind(G(13,2))$.}
\label{tab:p683orbits}
\begin{tabular}{clcc}
\toprule
dimension & representatives & orbits & faces\\
\midrule
$0$ & $0$ & $1$ & $13$\\
$1$ & $03,04,05,06$ & $4$ & $52$\\
$2$ & $036,037,047,038,048$ & $5$ & $65$\\
$3$ & $0369$ & $1$ & $13$\\
\bottomrule
\end{tabular}
\end{table}

The corresponding cyclic gap types are $3+10,4+9,5+8,6+7$ in dimension
one; $3+3+7,3+4+6,4+3+6,3+5+5,4+4+5$ in dimension two; and
$3+3+3+4$ in dimension three.

No independent face has five vertices, because five gaps of size at least
three cannot sum to $13$.  Each displayed orbit has length $13$: a nonempty
proper subset of a prime cycle has no nontrivial rotational stabilizer.  This
proves the first face vector in Theorem~\ref{thm:three-benchmarks}.

For $i\in\Z/13\Z$, define
\[
 Q_i=\{i,i+3,i+6,i+9\},\qquad
 B_i=Q_i\setminus\{i+3\}=\{i,i+6,i+9\}.
\]

\begin{lemma}\label{lem:p683collapse}
Each $B_i$ is contained in the unique tetrahedron $Q_i$.  Hence the thirteen
pairs $(B_i,Q_i)$ form a cyclic orbit of elementary collapses.
\end{lemma}

\begin{proof}
Translation reduces the assertion to $i=0$.  A fourth vertex compatible with
$B_0=\{0,6,9\}$ must avoid
\[
 \{11,12,0,1,2\}\cup\{4,5,6,7,8\}
 \cup\{7,8,9,10,11\}.
\]
The only remaining residue is $3$.  Thus $B_0$ has the unique coface $Q_0$,
and translation proves the claim for every $i$.
\end{proof}

After these collapses, the remaining complex satisfies $K\simeq L$; it is
two-dimensional and has face
vector $(13,52,52)$; its triangle orbits are represented by
$036,047,038,048$.  Apply the ordered-toggle matching with order
$0,1,\ldots,12$.  Table~\ref{tab:p683matching} gives the exact audit.  The
two pair columns count vertex--edge and edge--triangle pairs created at the
pivot; the survivor columns give the unmatched profile afterward.

\begin{table}[ht]
\centering
\caption{Ordered-toggle audit after the cyclic collapses.}
\label{tab:p683matching}
\begin{tabular}{crrrrr}
\toprule
pivot & $0\leftrightarrow1$ & $1\leftrightarrow2$ & vertices & edges & triangles\\
\midrule
0&8&12&5&32&40\\ 1&1&6&4&25&34\\ 2&1&5&3&19&29\\
3&0&1&3&18&28\\ 4&1&2&2&15&26\\ 5&1&4&1&10&22\\
6&0&3&1&7&19\\ 7&0&2&1&5&17\\ 8&0&0&1&5&17\\
9&0&1&1&4&16\\ 10&0&4&1&0&12\\ 11&0&0&1&0&12\\
12&0&0&1&0&12\\
\bottomrule
\end{tabular}
\end{table}

The unique critical vertex is $\{0\}$, and there are twelve critical
triangles.  The acyclicity and CW consequence in Lemma~\ref{lb-acyclic} give
\[
 \Ind(G(13,2))\simeq\bigvee^{12}S^2.
\]
The first collapse orbit is equivariant, but the subsequent ordered matching
is not; no equivariant wedge decomposition is asserted.

\subsection{Two fourteen-vertex certificates}\label{sec:p687p688}

For the remaining graphs, the faces were obtained by checking every one of
the $2^{14}$ vertex subsets against the defining edge set.  Oriented boundary
maps use increasing vertex order and the usual sign $(-1)^j$ on deletion of
the $j$th vertex.  The complete integer audit verified $\partial^2=0$.

\subsubsection{The graph \texorpdfstring{$G(14,4)$}{C14(1,4)}}

Corollary~\ref{cor:p687-family} already gives the homotopy type as the
$n=14$ member of the eight-periodic family.  The following
instance-level certificate is a logically independent finite
cross-check and records the parity-first matching used in the benchmark.

There are $281$ independent faces including the empty face, with
\[
 f=(14,63,112,77,14),\qquad \chi=0,\qquad \widetilde{\chi}=-1.
\]
Use the parity-first order
\begin{equation}\label{eq:p687order}
 0,2,4,6,8,10,12,1,3,5,7,9,11,13.
\end{equation}
The matching has $139$ pairs.  Table~\ref{tab:p687counts} lists the number
whose lower cell has dimension $0,1,2,3$ at each nonzero pivot.

\begin{table}[ht]
\centering
\caption{Per-pivot pair counts for $G(14,4)$.}
\label{tab:p687counts}
\begin{tabular}{crrrr@{\qquad}crrrr}
\toprule
$v$&$0$&$1$&$2$&$3$&$v$&$0$&$1$&$2$&$3$\\
\midrule
0&9&24&22&5 & 2&3&13&15&4\\
4&1&8&12&4 & 6&0&3&5&1\\
8&0&1&2&0 & 1&0&0&1&0\\
3&0&0&1&0 & 9&0&1&2&0\\
11&0&0&1&0 & 13&0&0&1&0\\
\bottomrule
\end{tabular}
\end{table}

The pivots $10,12,5,7$ create no pairs.  The only critical faces are the
vertex $\{0\}$ and the tetrahedron $\{5,7,10,12\}$.  By
Lemma~\ref{lb-acyclic},
\[
 \Ind(G(14,4))\simeq S^3.
\]
In particular $\Ht_3\cong\Z$, all other reduced integral homology groups
vanish, and there is no torsion.

\subsubsection{The graph \texorpdfstring{$G(14,5)$}{C14(1,5)}}

There are $339$ independent faces including the empty face, with
\[
 f=(14,63,112,91,42,14,2),\qquad \chi=2,\qquad \widetilde{\chi}=1.
\]
For the natural order $0,1,\ldots,13$, the matching has $168$ pairs.
Table~\ref{tab:p688counts} records the number of pairs by lower-cell
dimension.  The omitted pivots $10,11,12,13$ create no pairs.

\begin{table}[htbp]
\centering
\caption{Per-pivot pair counts for $G(14,5)$.}
\label{tab:p688counts}
\begin{tabular}{crrrrrr}
\toprule
$v$&$0$&$1$&$2$&$3$&$4$&$5$\\
\midrule
0&9&24&26&15&6&1\\
1&3&15&22&15&6&1\\
2&0&3&3&0&0&0\\
3&1&2&2&0&0&0\\
4&0&2&2&0&0&0\\
5&0&2&2&0&0&0\\
6&0&0&1&0&0&0\\
7&0&1&1&0&0&0\\
8&0&1&1&0&0&0\\
9&0&0&1&0&0&0\\
\bottomrule
\end{tabular}
\end{table}

The only critical faces are $\{0\}$ and the triangle $\{3,6,9\}$.
Lemma~\ref{lb-acyclic} gives
\[
 \Ind(G(14,5))\simeq S^2.
\]
Thus $\Ht_2\cong\Z$, every other reduced integral homology group vanishes,
and there is no torsion.  

\begin{remark}
The parity-first order in \eqref{eq:p687order} and the natural order in the
$s=5$ case behave very differently.  Corollary~\ref{lb-eight-periodic} supplies a
parameter-family rule only after the first graph is relabeled as
$G_{14}(\{2,3\})$ and embedded in the specific family $G_n(\{2,3\})$.  Nothing here
extrapolates the $s=5$ certificate or classifies arbitrary jump pairs.
\end{remark}

\begin{proof}[Proof of Theorem~\ref{thm:three-benchmarks}]
The cyclic gap enumeration gives the thirteen-vertex face vector, and the collapse orbit followed by the displayed ordered matching leaves one vertex and twelve triangles. The two fourteen-vertex enumerations and matchings leave one vertex and one cell in dimension three or two, respectively. Lemma~\ref{lb-acyclic} gives the three wedge types. Their integral reduced homology is free, with the displayed ranks and degrees.
\end{proof}

\section{Perfectness, cyclic characters, and quotient geometry}
\label{sec:symmetry}

The family and benchmark reductions do more than identify homotopy types.
They attain the Morse lower bounds exactly.  Moreover, cyclic symmetry
constrains their Betti numbers, determines their top-homology actions, and
gives a small quotient model for the prime benchmark.

\begin{proposition}[Perfectness and optimality]\label{prop:perfect}
For every complex in Corollary~\ref{lb-eight-periodic} and Theorem~\ref{thm:three-benchmarks}, the
displayed reduction is a perfect discrete Morse reduction.  Consequently no
acyclic matching on its face poset can have fewer critical cells in any
dimension.
\end{proposition}

\begin{proof}
For $G_n(\{2,3\})$, the proof of Corollary~\ref{lb-eight-periodic} identifies the natural-order critical
profile as one vertex together with exactly the cells that generate the
displayed reduced homology.  Thus equality holds in every weak Morse
inequality.

For $G(13,2)$, adjoining the thirteen elementary-collapse pairs to the
ordered matching on the remaining complex gives an acyclic matching on the
original face poset; this is the standard concatenation of a collapse sequence
with a Morse matching on the terminal subcomplex.  Its critical profile is
$(1,0,12,0)$.  The two fourteen-vertex profiles are $(1,0,0,1,0)$ and
$(1,0,1,0,0,0,0)$.

By Theorem~\ref{thm:three-benchmarks}, these benchmark profiles agree dimension by dimension with
the corresponding Betti numbers over every field.  The weak Morse
inequalities require the number of critical $i$-cells in any acyclic matching
to be at least the $i$th Betti number.  Equality here proves both perfectness
and the asserted global optimality.
\end{proof}

Let $\Gamma$ be a finite group of order $q>1$, and write
$\operatorname{Aug}_{\Q}(\Gamma)$ for the augmentation subrepresentation of
$\Q[\Gamma]$.

\begin{proposition}[Face-free action constraint]\label{prop:face-free}
Suppose that $\Gamma$ acts simplicially on a connected finite complex $K$, no
nonidentity element fixes a nonempty simplex setwise, and
$\Ht_i(K;\Q)=0$ for $i\ne d$.  Put
$\beta=\dim_{\Q}\Ht_d(K;\Q)$.  If $d$ is odd, then
\[
 \beta\equiv1\pmod q,
 \qquad
 \Ht_d(K;\Q)
 \cong \Q\oplus
 \Q[\Gamma]^{\oplus(\beta-1)/q}.
\]
If $d$ is even, then $\beta\ge q-1$, $\beta\equiv-1\pmod q$, and
\[
 \Ht_d(K;\Q)
 \cong \operatorname{Aug}_{\Q}(\Gamma)\oplus
 \Q[\Gamma]^{\oplus((\beta+1)/q-1)}.
\]
\end{proposition}

\begin{proof}
For $1\ne g\in\Gamma$, no oriented simplex is carried to itself as a basis
element, even up to sign, so $g$ has trace zero on every simplicial chain
group.  Since $K$ is connected, $\widetilde H_0(K;\Q)=0$.  Applying the Hopf
trace identity to the ordinary simplicial chain complex gives
\[
 0=1+(-1)^d
 \operatorname{tr}\bigl(g_*\mid\Ht_d(K;\Q)\bigr).
\]
Thus the homology character has value $\beta$ at the identity and
$\varepsilon=(-1)^{d+1}$ elsewhere.

For odd $d$ this character is
$\mathbf{1}+((\beta-1)/q)\operatorname{reg}_{\Gamma}$; for even $d$ it is
$-\mathbf{1}+((\beta+1)/q)\operatorname{reg}_{\Gamma}$.  Taking the inner
product with the trivial character shows that the displayed coefficients are
integers and, in the even case, that $(\beta+1)/q-1\ge0$.  Maschke
semisimplicity and equality of characters identify the rational
representations.  Finally,
$\operatorname{reg}_{\Gamma}-\mathbf{1}$ is the augmentation
representation.
\end{proof}

As elsewhere, $T$ denotes translation by one on the labeled vertices.

For a cyclic group $C_m$, write $\operatorname{Aug}_{\Q}(C_m)$ for the
augmentation subrepresentation of $\Q[C_m]$.  If $m$ divides $n$, a
$\Q[C_m]$-module is regarded as a $\Q[C_n]$-module through the quotient
$C_n\twoheadrightarrow C_m$.

\begin{theorem}[Cyclic character of the family]
\label{thm:family-cyclic-character}
Let $d_n$ and $m_n$ be the sphere dimension and multiplicity in Corollary~\ref{lb-eight-periodic}.  As a
rational representation of cyclic translation,
\[
 \Ht_{d_n}(X_n;\Q)\cong
 \begin{cases}
  \Q\oplus\Q(\zeta_8),&n\equiv0\pmod8,\\
  \operatorname{Aug}_{\Q}(C_4),&n\equiv4\pmod8,\\
  \Q,&n\not\equiv0,4\pmod8.
 \end{cases}
\]
Here a generator acts on $\Q(\zeta_8)$ by multiplication by $\zeta_8$;
the first two representations are inflated through $C_n\twoheadrightarrow
C_8$ and $C_n\twoheadrightarrow C_4$, respectively.
\end{theorem}

\begin{proof}
Fix $a\in\Z/n\Z$ and put $g=\gcd(n,a)$, with $\gcd(n,0)=n$.  Translation
by $a$ has $g$ vertex orbits, each of length $\ell=n/g$.  Let $B_g(t)$ count
sets of these orbits whose union is an independent face, by number of chosen
orbits.  If $g=1,2,3$, one of the jumps $2,3$ lies within every orbit, so
$B_g(t)=1$.  For $g\geq4$, a union of orbits is independent precisely when
the corresponding residues form an independent set in the loopless quotient
with jumps $2$ and $3$.  Thus
\[
 B_4(-1)=-3,\qquad B_5(-1)=1,
\]
because the quotient graphs are $K_4$ and $C_5$, while for $g\geq6$
Corollary~\ref{lb-eight-periodic} gives
\begin{equation}\label{eq:orbit-polynomial}
 B_g(-1)=(-1)^{d_g+1}m_g,
\end{equation}
where $m_g$ and $d_g$ are its sphere multiplicity and dimension.

It remains to relate this orbit count to the character.  An invariant face
formed from $t$ orbits has $t\ell$ vertices.  Translation restricts to one
$\ell$-cycle on each chosen orbit, so its orientation sign is
$(-1)^{t(\ell-1)}$.  Including the alternating chain sign, this face
contributes
\[
 (-1)^{t\ell-1}(-1)^{t(\ell-1)}=(-1)^{t+1}
\]
to the ordinary simplicial Lefschetz number.  Consequently that number is
$1-B_g(-1)$.  Since $X_n$ is connected and has reduced homology only
in degree $d_n$, the Hopf trace identity yields
\begin{equation}\label{eq:family-character}
 \operatorname{tr}\!\left(T^a_*\mid
 \Ht_{d_n}(X_n;\Q)\right)
   =(-1)^{d_n+1}B_g(-1).
\end{equation}

Substituting the four residue cases from
Corollary~\ref{lb-eight-periodic} into~\eqref{eq:family-character} gives the
following characters.  For $8\mid n$, the value is $5$ when $8\mid a$,
$-3$ when $a\equiv4\pmod8$, and $1$ otherwise.  This is the character of
$\Q\oplus\Q(\zeta_8)$ inflated from $C_8$.  For $n\equiv4\pmod8$, the
value is $3$ when $4\mid a$ and $-1$ otherwise, which is the character of
$\operatorname{Aug}_{\Q}(C_4)$ inflated from $C_4$.  In every other residue
class the homology has rank one and the generator has trace $1$, so the
representation is trivial.  Rational representations of a finite cyclic group
are determined by their characters, completing the proof.
\end{proof}

Write $\operatorname{Aug}_{\Z}(C_{13})$ for the integral augmentation ideal
in $\Z[C_{13}]$.

\begin{theorem}[Cyclic actions for the benchmark graphs]\label{thm:cyclic-action}
Under cyclic translation,
\begin{align*}
 \Ht_2(\Ind(G(13,2));\Z)
   &\cong \operatorname{Aug}_{\Z}(C_{13}),\\
 T_*\mid \Ht_3(\Ind(G(14,4));\Z)&=+1,\\
 T_*\mid \Ht_2(\Ind(G(14,5));\Z)&=-1.
\end{align*}
In particular, the two fourteen-vertex spheres carry different orientation
characters under the same cyclic generator.
\end{theorem}

\begin{proof}
For the thirteen-vertex complex, every nonidentity power of $T$ is
transitive because $13$ is prime.  Since the full vertex set is not
independent, Proposition~\ref{prop:face-free} applies with $q=13$ and even
$d=2$.  It forces $\beta_2\ge12$ and $\beta_2\equiv-1\pmod {13}$; the value
$12$ from Theorem~\ref{thm:three-benchmarks} is therefore the smallest possible, and the
rational module is the augmentation representation.

It remains to identify the integral lattice.  Orient simplices by increasing
vertex order and set
\begin{align*}
 z={}&[036]-[037]+[047]-[049]+[069]-[147]+[149]-[169]\\
    &+[1,6,11]-[1,7,11]-[3,6,11]+[3,7,11].
\end{align*}
Direct boundary calculation gives $\partial z=0$ and
$\sum_{j=0}^{12}T^jz=0$ at chain level.  Order the $65$ triangles
lexicographically and form the $65\times25$ matrix whose first thirteen
columns are the tetrahedron boundaries and whose remaining columns are
$z,T z,\ldots,T^{11}z$.  On the zero-based rows
\[
\begin{split}
 0,1,2,3,5,6,7,10,15,16,17,18,20,21,22,25,30,31,33,\\
 35,36,37,40,45,50
\end{split}
\]
the determinant is $-1$.  Hence these columns generate a primitive rank-$25$
sublattice of the two-cycle lattice.  Theorem~\ref{thm:three-benchmarks} gives
$\operatorname{rank}\Ht_2=12$ and $\Ht_3=0$, so the cycle lattice itself
has rank $13+12=25$.  A primitive sublattice of $\Z^{65}$ cannot lie properly
in a rank-equal lattice, so the two lattices are equal.  Consequently the
classes $[z],[T z],\ldots,[T^{11}z]$ form an integral homology basis,
with
\[
 T^{12}[z]= -\sum_{j=0}^{11}T^j[z].
\]
Thus the module is
$\Z[C_{13}]/(1+T+\cdots+T^{12})$, which multiplication by $T-1$
identifies with $\operatorname{Aug}_{\Z}(C_{13})$.

For either fourteen-vertex complex, the generator $T$ is transitive, so
it fixes no nonempty simplex setwise.  Since each complex is connected, the
Hopf trace calculation in the proof
of Proposition~\ref{prop:face-free} gives trace $+1$ in degree three and $-1$
in degree two.  Each relevant integral homology group has rank one, so its
automorphism is determined by that sign.
\end{proof}

\begin{corollary}[The cyclic quotient]\label{cor:p683-quotient}
The translation action on $\Ind(G(13,2))$ is free on its realization.  The
quotient has cell vector $(1,4,5,1)$, Euler characteristic one, fundamental
group $C_{13}$, and vanishing positive-degree rational homology.
\end{corollary}

\begin{proof}
A fixed point would lie in the relative interior of a simplex fixed setwise,
which does not exist.  The face-orbit counts are those of
Table~\ref{tab:p683orbits}.  Theorem~\ref{thm:three-benchmarks} makes the covering complex
simply connected, so the quotient has fundamental group $C_{13}$.  Over
$\Q$, the quotient chain complex is obtained by coinvariants.  These are
exact for the semisimple algebra $\Q[C_{13}]$, and the augmentation
representation has zero coinvariants.  The asserted rational homology
follows.
\end{proof}

\section{Finite CW reductions}\label{sec:cw}
Orient a simplex by increasing vertex order. In positive degrees use
\begin{equation}\label{eq:boundary}
 \partial[v_0,\ldots,v_q]=\sum_{i=0}^q(-1)^i[v_0,\ldots,\widehat v_i,\ldots,v_q].
\end{equation}
In the ordinary chain complex, the boundary of a vertex is zero.
For a facet $\alpha$ of a simplex $\beta$, write $[\beta:\alpha]=\pm1$ for the coefficient of $\alpha$ in $\partial\beta$.
Order nonempty independent sets by cardinality and then lexicographically, and mark them available.
For $v=0,1,\ldots,n-1$, inspect this ordered list: whenever $\sigma$ is available, $v\notin\sigma$, and $\sigma\cup\{v\}$ is available, pair the two and mark both unavailable.
Denote this ordinary matching by $M$.

Orient unmatched Hasse incidences from the larger face to the smaller, and matched incidences in the reverse direction.
Lemma~\ref{lb-acyclic} proves acyclicity for this ordered matching. For each specified graph, topological sorting of the whole directed Hasse graph also verifies the resulting finite matching independently.
The vertex $\{0\}$ is critical.
Equivalently, one may include the empty face, which pairs with $\{0\}$, and remove this pair afterwards.
All pair counts below use the ordinary convention.

\begin{proposition}\label{prop:finite}
The finite matching has the following data; total face counts include the empty face.
The critical faces are given in Appendix~\ref{app:cells}.
\begin{center}
\begin{tabular}{@{}lrrl@{}}
\toprule
Complex & Faces & Ordinary pairs & Ordinary critical cells\\
\midrule
$\K(14,2)$ &211&98& one vertex and thirteen $2$-cells\\
$\K(14,6)$ &337&167& one vertex and one $1$-cell\\
$\K(15,4)$ &434&215& one vertex and two $2$-cells\\
$\K(15,5)$ &394&194& one vertex, three $2$-cells, one $3$-cell\\
$\K(16,4)$ &673&334& one vertex, one $2$-cell, two $3$-cells\\
$\K(17,4)$ &953&475& one vertex and one $3$-cell\\
$\K(18,4)$ &1495&743& one vertex and seven $3$-cells\\
$\K(18,5)$ &1647&817& one vertex and eleven $3$-cells\\
$\K(18,6)$ &1297&645& one vertex, three $3$-cells, two $4$-cells\\
$\K(19,4)$ &2205&1101& one vertex and one $3$-cell\\
$\K(19,5)$ &2205&1101& one vertex and one $3$-cell\\
$\K(21,4)$ &5016&2506& one vertex and two $4$-cells\\
$\K(22,5)$ &8275&4134& one vertex, three $4$-cells, two $5$-cells\\
$\K(22,6)$ &7877&3933& one vertex, three $3$-, four $4$-, two $5$-cells\\
$\K(23,5)$ &11593&5795& one vertex and one $3$-cell\\
$\K(24,4)$ &16782&8389& one vertex and two $4$-cells\\
$\K(25,4)$ &25011&12503& one vertex, one $4$-cell, two $5$-cells\\
$\K(27,5)$ &60079&30035& one vertex and seven $5$-cells\\
\bottomrule
\end{tabular}
\end{center}
For $\K(25,4)$, in the appendix's critical-cell order, the degree-five Morse differential is
\begin{equation}\label{eq:row}
 \partial^M_5=\begin{pmatrix}-1&1\end{pmatrix}:\Z^2\longrightarrow\Z.
\end{equation}
\end{proposition}
\begin{proof}
These are finite exhaustive matching and signed-boundary computations, specified as follows. The algorithms provide a mathematical reconstruction; Section~\ref{sec:reconstruction} distinguishes this specification from the reported execution records and states their availability.
Set $F_{-1}=\{\varnothing\}$ and
\[
 F_v=F_{v-1}\cup\{\sigma\cup\{v\}:\sigma\in F_{v-1},\ \sigma\cap N(v)=\varnothing\}
 \quad(0\le v<n),
\]
where $N(v)$ is the neighborhood of $v$ in $G(n,s)$.
This enumerates independent sets uniquely by their largest vertex.
Sort $F_{n-1}\setminus\{\varnothing\}$, construct $M$, and count its critical complement.
Topological sorting visits every face.
The $f$-vectors for $\K(18,4)$ and $\K(18,5)$, recording numbers of nonempty simplices by dimension, are respectively
\[
 (18,117,348,504,360,129,18),\qquad
 (18,117,348,504,378,189,72,18,2).
\]
For the degree-five differential of $\K(25,4)$ in \eqref{eq:row}, sum signed alternating paths from each critical five-cell to the critical four-cell.
A downward incidence $\beta\to\alpha$ has weight $[\beta:\alpha]$, and a matched upward incidence has weight $-[\beta:\alpha]^{-1}$.
The sums are $-1$ and $1$.
All incidences used are $\pm1$.
An independent recursive enumeration verifies the full face partition and checks every directed Hasse incidence against a topological ordering.
The seven additional pure-dimensional cases use the same full-domain check; the five mixed-dimensional cases are treated in Section~\ref{sec:mixed}.
Execution tests are summarized in Section~\ref{sec:reconstruction}.
The finite algorithm, critical lists, and signed path rule determine the data without a stored certificate.
\end{proof}

Forman's theorem gives the corresponding ordinary Morse CW models.
\subsection{Mixed-dimensional Morse models}\label{sec:mixed}
Cell counts alone do not determine the homology of a mixed-dimensional Morse model.
In the critical-face orders of Appendix~\ref{app:mixed}, its nonzero differentials are
\begin{align}
 \partial^M_3(\K(15,5))&=\begin{pmatrix}1\\0\\1\end{pmatrix},
 &\partial^M_3(\K(16,4))&=\begin{pmatrix}1&-1\end{pmatrix},\label{eq:mixedsmall}\\
 \partial^M_4(\K(18,6))&=\begin{pmatrix}0&0\\0&1\\-1&0\end{pmatrix},
 &\partial^M_5(\K(22,5))&=\begin{pmatrix}-1&0\\-1&0\\0&-1\end{pmatrix},\label{eq:mixedmiddle}\\
 \partial^M_4(\K(22,6))&=\begin{pmatrix}-1&-1&1&1\\-1&0&1&0\\0&-1&0&1\end{pmatrix},
 &\partial^M_5(\K(22,6))&=\begin{pmatrix}1&0\\0&1\\1&0\\0&1\end{pmatrix}.\label{eq:mixedlarge}
\end{align}
All remaining differentials vanish.
The signed alternating-path rule in the proof of Proposition~\ref{prop:finite} gives every displayed entry.
An independent check uses the backward projection recurrence, rather than the forward path sums, and compares every integer entry.
It also verifies the full face partition, acyclic Hasse graph, and square-zero identities.
An isolated reconstruction reproduces all data.
Mutations of differentials, unit pivots, cancellation lists, critical faces, topological order and graph parameters are rejected in each of the five cases.

A unit entry $d_q[i,j]=u=\pm1$ permits an integral chain cancellation.
After integer basis changes isolate this entry, the remaining block is
\[
 d'_q[a,b]=d_q[a,b]-d_q[a,j]u^{-1}d_q[i,b]\qquad(a\ne i,\ b\ne j).
\]
The isolated source and target form a contractible two-term summand.
Their row in $d_{q+1}$ and column in $d_{q-1}$ vanish by $d^2=0$ and can be removed.
Applying this operation uses respectively one, one, two, two and four unit pivots in the five displayed cases.
The residual complexes have one degree-zero generator and, respectively, two generators in degree two, one in degree three, one in degree three, one in degree four and one in degree three; every residual differential is zero.
This proves the asserted free integral homology.
It does not assert that these integer basis changes cancel topological cells.

\begin{lemma}\label{lem:homologywedge}
Let $X$ be a simply connected CW complex with reduced integral homology concentrated in degree $q\ge2$ and equal to $\Z^r$ there.
Then $X\simeq\bigvee^r S^q$.
\end{lemma}
\begin{proof}
Successive applications of Hurewicz show that $\pi_i(X)=0$ for $i<q$ and that the Hurewicz map $\pi_q(X)\to H_q(X;\Z)$ is an isomorphism.
Choose sphere maps representing an integral homology basis and combine them into a map $\bigvee^r S^q\to X$.
This induces an isomorphism on all integral homology groups.
The homological Whitehead theorem for simply connected CW complexes makes it a homotopy equivalence; these are standard consequences of Hurewicz and Whitehead \cite[Theorems~4.32 and~4.37, and Corollary~4.33]{Hatcher}.
\end{proof}

Each of these five Morse CW models has one vertex and no one-cells, so is simply connected.
Lemma~\ref{lem:homologywedge} therefore supplies their ordinary wedge types.
The same argument applies to $\K(25,4)$ using \eqref{eq:row}.
For the remaining cases the CW model has one vertex and cells in only one positive dimension, and hence is already a wedge of spheres.
The differential and the separate connectivity argument are both needed for the mixed cases.

\section{Integral homology bases and induced actions}\label{sec:certification}
\begin{lemma}\label{lem:basis}
Let $C_*$ be an integral chain complex with $H_q(C_*)$ free of rank $r$.
Suppose $P:C_q\to\Z^r$ kills $\partial C_{q+1}$ and cycles $z_j$ satisfy $Pz_j=e_j$.
Then $[z_1],\ldots,[z_r]$ form an integral homology basis, with inverse coordinates induced by $P$ on cycles.
For a chain automorphism $U$, the homology matrix has columns $P(Uz_j)$.
\end{lemma}
\begin{proof}
The induced map $\overline P:H_q(C_*)\to\Z^r$ has section $e_j\mapsto[z_j]$.
Its kernel is a subgroup of a free group and has rank zero, so is zero.
Taking coordinates proves the matrix assertion.
\end{proof}

We give one construction for $(n,s,q,r)=(14,2,2,13),(18,4,3,7),(18,5,3,11),(27,5,5,7)$.
Let $c_1,\ldots,c_r$ be the critical $q$-faces in the appendix's order.
For a $q$-simplex $\sigma$, define
\begin{equation}\label{eq:projection}
 P(\sigma)=
 \begin{cases}
 e_j,&\sigma=c_j,\\
 0,&\sigma\text{ is the upper face of a matched pair},\\
 -[\beta:\sigma]^{-1}\displaystyle\sum_{\substack{\eta\prec\beta\\\eta\ne\sigma}}[\beta:\eta]P(\eta),
 &(\sigma,\beta)\in M.
 \end{cases}
\end{equation}
The last dependency follows a matched upward and unmatched downward step in the acyclic Hasse graph, so the recurrence is well-founded.
In each degree, the coefficient of a critical face $c$ in $P(\sigma)$ is the sum of the signed paths from $\sigma$ to $c$ alternating a matched incidence upward with an unmatched incidence downward, including the length-zero path when $\sigma=c$.
Positive-length projection paths start at lower matched faces and end at critical faces; upper matched faces have zero projection.
These degreewise maps form the projection in the algebraic Morse contraction \cite{Skoldberg,JW}, so $P_{q-1}\partial=\partial^M_qP_q$.
In every homology-basis application below there is no critical $(q+1)$-face, so $P$ kills all $(q+1)$-boundaries; in the mixed-dimensional cases of Section~\ref{sec:mixed} the same chain-map property identifies $P\partial$ on a critical $(q+1)$-face with the Morse differential.

To lift a critical face, form a directed graph on $q$-simplices.
If $\alpha\prec\sigma$ is matched to a different $q$-simplex $\sigma'$, insert an edge $\sigma\to\sigma'$ of weight
\begin{equation}\label{eq:liftweight}
 -[\sigma:\alpha][\sigma':\alpha]^{-1}.
\end{equation}
This graph is acyclic.
Starting with coefficient one at $c_j$, propagate coefficients in topological order, adding each source coefficient times its edge weight at the target.
The accumulated coefficients define $z_j$.
This is the signed alternating-path construction \cite{Skoldberg,JW}.
Every simplex reached by a lift path other than $c_j$ itself is the upper face of a matched pair, so $Pz_j=e_j$ holds by construction; the resulting chains and projections are nevertheless checked directly.

\begin{proposition}\label{prop:chainchecks}
For each of the four specified complexes, $\partial z_j=0$, $Pz_j=e_j$, and $P$ kills the boundary of every $(q+1)$-simplex.
The full domains checked are
\begin{center}
\begin{tabular}{@{}lrr@{}}
\toprule
Complex & $q$-simplices & $(q+1)$-simplices\\
\midrule
$\K(14,2)$ &98&35\\
$\K(18,4)$ &504&360\\
$\K(18,5)$ &504&378\\
$\K(27,5)$ &16137&13338\\
\bottomrule
\end{tabular}
\end{center}
\end{proposition}
\begin{proof}
Apply \eqref{eq:projection}--\eqref{eq:liftweight} to the completely enumerated face set and matching.
Apply \eqref{eq:boundary} to every resulting chain, and apply $P$ to each chain and each $(q+1)$-simplex boundary.
The former images are the standard basis vectors and all boundaries map to zero, as the chain-map description of $P$ predicts.
Cycle support sizes are given in the appendix.
An independent check enumerates faces from the graph edge differences and tests the full integer coefficients of $P$ and the lifts, without importing the producer's recurrence or path construction.
Thus the identities hold in the simplicial chain complex itself.
\end{proof}

Proposition~\ref{prop:finite} gives the free ranks needed by Lemma~\ref{lem:basis}.
To apply translation to an oriented simplex, translate every vertex modulo $n$ and sort the image increasingly, multiplying by the sorting permutation's sign.
Thus
\begin{equation}\label{eq:transport}
 T_{ij}=(P(\tau z_j))_i.
\end{equation}
The matrices in Appendix~\ref{app:matrices} use this column convention.
The matching need not be translation invariant: the transfer uses the actual simplicial action, and the CW equivalence is not asserted to respect it.

\section{The cyclic polynomial modules}\label{sec:actions}
\subsection{The rank-two and rank-one cases}
Proposition~\ref{prop:finite} and Section~\ref{sec:mixed} supply the hypotheses of Proposition~\ref{prop:lowrank} for every rank-one and rank-two entry in Theorem~\ref{thm:main}.
The rank-one action is trivial in odd homological degree; the degree-four $\K(22,5)$ instead carries the sign action.
Every rank-two case has even degree, so its module is $\Z[t]/(t^2+t+1)$ and the action has exact order three.
In particular the degree-two $(15,4),(15,5)$ and degree-four $(21,4),(24,4)$ examples share the same Eisenstein translation lattice.

\subsection{Four unimodular cyclic bases}
If a free rank-$r$ homology lattice $H_q$ with operator $T$ has a vector $v$ for which
\[
 B=(v,Tv,\ldots,T^{r-1}v),\qquad\det B=\pm1,
\]
and a monic degree-$r$ polynomial $f$ satisfies $f(T)v=0$, then
\begin{equation}\label{eq:cyclic}
 \Z[t]/(f)\longrightarrow H_q,\qquad \overline g\longmapsto g(T)v
\end{equation}
is an isomorphism.
The orbit vectors form an integer basis; commutation with $T$ shows $f(T)=0$ on that basis.
Both sides of \eqref{eq:cyclic} are free of rank $r$.

\begin{proposition}\label{prop:cyclic}
For the matrices of \eqref{eq:transport}, the following are exact identities.
\begin{center}
\begin{tabular}{@{}lcccl@{}}
\toprule
Complex & $r$ & $v$ & $\det B$ & $f(t)$, with $f(T)v=0$\\
\midrule
$\K(14,2)$ &13&$e_1$&$-1$&$1+t+\cdots+t^{13}$\\
$\K(18,4)$ &7&$e_2$&$-1$&$(t-1)\Phi_9(t)$\\
$\K(18,5)$ &11&$e_1$&$1$&$(t^9-1)\Phi_6(t)$\\
$\K(27,5)$ &7&$e_2$&$-1$&$(t-1)\Phi_9(t)$\\
\bottomrule
\end{tabular}
\end{center}
\end{proposition}
\begin{proof}
Construct the matrices by signed geometric transport on the verified cycles.
The appendix gives the three odd-degree matrices in the fixed critical-face order, and hence specifies their $B$ by their columns.
For $\K(14,2)$, Proposition~\ref{prop:composite} supplies an explicit cycle, and the same projection and lift recurrences determine $B$.
Successive integer multiplication gives the orbit columns and the indicated annihilating relation.
Exact elimination gives the displayed determinants.
An independent calculation checks the orbit columns, determinants by rational elimination, and the polynomial identity.
Together with Proposition~\ref{prop:chainchecks}, these are identities for the geometric homology action.
\end{proof}

Equation~\eqref{eq:cyclic} proves the four module identifications.
The $\Phi_9$ factor supplies a primitive ninth-root constituent, so both occurrences of $L_3$ have action of order exactly nine.
For $L_2$, the rational factors $t^9-1$ and $\Phi_6$ give order $\operatorname{lcm}(9,6)=18$.
In particular the seven-dimensional lattices in degrees three and five are isomorphic as $\Z[t]$-modules: both are inflated from the same $C_9$ action.

\section{A common factor-ideal description of the gluing}\label{sec:gluing}
The following elementary lemma first treats the two-factor integer obstructions; Lemma~\ref{lem:multifactor} then handles the three-factor comparison.
It is a two-factor Chinese remainder calculation within the arithmetic framework of \cite{MS}; the graph-specific input is the cyclic-lattice identification.
The general resultant theory controls which primes can obstruct factor splitting, whereas the explicit quotient below retains the additive group structure.

\begin{lemma}\label{lem:gluing}
Let $a,b\in\Z[t]$ be monic and coprime over $\Q[t]$, and put $L=\Z[t]/(ab)$.
Then $A=\ker(a:L\to L)=bL$, $D=\ker(b:L\to L)=aL$, $A\cap D=0$, and
\begin{equation}\label{eq:quotient}
 L/(A\oplus D)\cong\Z[t]/(a,b).
\end{equation}
There is an exact sequence of $\Z[t]$-modules
\begin{equation}\label{eq:extension}
 0\longrightarrow\Z[t]/(b)\xrightarrow{\ \overline g\mapsto ag\ }L
 \longrightarrow\Z[t]/(a)\longrightarrow0.
\end{equation}
It splits if and only if $(a,b)=\Z[t]$.
\end{lemma}
\begin{proof}
If $ag$ is divisible by $ab$ in $\Z[t]$, cancellation gives $g=bh$.
This proves $A=bL$; the other kernel and the injection in \eqref{eq:extension} follow the same way.
The image of the injection is $aL$, the kernel of reduction modulo $a$.
Over $\Q$, multiplication by $a$ is invertible on the $b$-factor and multiplication by $b$ on the $a$-factor.
Thus their kernels have zero intersection over $\Q$.
Since $L$ is free over $\Z$, $A\cap D=0$ integrally.
The quotient by their sum is precisely the quotient by the ideal $(a,b)$, proving \eqref{eq:quotient}.

A module section of reduction modulo $a$ has image annihilated by $a$, hence contained in $A$.
If a section exists, $A+D=L$, so $(a,b)=\Z[t]$.
Conversely, this ideal equality gives the integral Chinese remainder decomposition and a section.
\end{proof}

Apply the lemma first with $a=t-1$ and $b=\Phi_9(t)$.
Reduction modulo $t-1$ gives
\[
 \Z[t]/(t-1,\Phi_9(t))=\Z/(\Phi_9(1))=\Z/3\Z.
\]
In particular $A_3=\Z\Phi_9$ and $D_3=(t-1)L_3\cong\Z[\zeta_9]$.
The exact sequence
\[
 0\longrightarrow\Z[\zeta_9]\longrightarrow L_3
 \xrightarrow{\ \overline g\mapsto g(1)\ }\Z\longrightarrow0
\]
does not split as a $\Z[t]$-module sequence: every invariant element has augmentation divisible by three.
This holds for both the $C_{18}$ and $C_{27}$ realizations.

Next take $a=t^9-1$ and $b=\Phi_6(t)$.
Modulo $\Phi_6$, one has $t^2=t-1$, hence $t^3=-1$ and $t^9-1=-2$.
Consequently
\[
 \Z[t]/(t^9-1,\Phi_6(t))
 \cong\F_2[t]/(t^2+t+1),
\]
whose additive group is $(\Z/2\Z)^2$.
Thus $A_2=\Phi_6 L_2\cong\Z[C_9]$ and $D_2=(t^9-1)L_2\cong\Z[\zeta_6]$ have direct sum of index four.
The exact sequence
\[
 0\longrightarrow\Z[\zeta_6]\longrightarrow L_2
 \longrightarrow\Z[C_9]\longrightarrow0
\]
is nonsplit as a $\Z[t]$-module sequence.
Both underlying abelian-group sequences split because their quotients are free; the obstruction concerns a section compatible with translation.
It is the exponent-two quotient, rather than a cyclic quotient of order four, that records this gluing.

\subsection{Three-factor gluing and a cycle-square comparison}
\begin{lemma}\label{lem:multifactor}
Let $f=\prod_{i=1}^m f_i$ be a product of monic polynomials that are pairwise coprime over $\Q[t]$, and put $L=\Z[t]/(f)$ and $g_i=f/f_i$.
Then $U_i=\ker(f_i:L\to L)=g_iL\cong\Z[t]/(f_i)$, the sum of the $U_i$ is direct, and
\[
 L/\bigoplus_i U_i\cong\Z[t]/(g_1,\ldots,g_m).
\]
The factor-kernel sum need not exhaust $L$, despite the rational factor decomposition.
\end{lemma}
\begin{proof}
The identity $f_ip=fh=f_ig_ih$ implies $p=g_ih$, proving the kernel formula and the isomorphism by multiplication by $g_i$.
Over $\Q$ the kernels occupy distinct factors of the Chinese remainder decomposition, so their sum is direct over $\Q$, hence over $\Z$.
The quotient by their sum is the quotient by the displayed ideal.
This is an elementary factor-ideal consequence of classical integral Chinese remainder theory \cite{MS}.
\end{proof}

\begin{lemma}\label{lem:crtdual}
With the monic factors in Lemma~\ref{lem:multifactor}, let
\[
 P:\Z[t]/(f)\longrightarrow\bigoplus_i\Z[t]/(f_i),\qquad
 a\longmapsto(a\bmod f_i)_i
\]
be the reduction map.
The finite abelian groups $\operatorname{coker}P$ and
$L/\bigoplus_i\ker(f_i:L\to L)$ have the same Smith invariant factors.
\end{lemma}
\begin{proof}
For a monic $h\in\Z[t]$ of degree $d$, let $\lambda_h$ extract the coefficient
of $t^{d-1}$ from the representative of degree below $d$.
The integer pairing $\beta_h(a,b)=\lambda_h(ab\bmod h)$ is perfect:
its monomial Gram matrix has zero entries for $i+j<d-1$ and ones
for $i+j=d-1$, so reversing the columns gives a triangular matrix with
diagonal ones.
Put $B=\bigoplus_i\Z[t]/(f_i)$ and
$E:B\to L$, $(b_i)\mapsto\sum_i(f/f_i)b_i$.
By Lemma~\ref{lem:multifactor}, its image is the factor-kernel sum.
Division by the monic $f_i$ gives
\[
 \lambda_f((f/f_i)ab_i\bmod f)
 =\lambda_{f_i}((a\bmod f_i)b_i\bmod f_i).
\]
Thus, for the unimodular Gram matrices $J_f,J_B$,
$J_fE=P^{\mathsf T}J_B$.
The embedding and the transpose of the reduction are integrally equivalent,
and transposition preserves Smith form.
This is an elementary duality within classical integral CRT theory
\cite{MS}; the conclusion is about abelian groups, without identifying
their geometric actions.
\end{proof}

\begin{proposition}\label{prop:composite}
For $\K(14,2)=\Ind(C_{14}^2)$, let
\begin{equation}\label{eq:cross14}
 z=([1]-[0])*([5]-[4])*([9]-[8]),\qquad w=[z]\in\Ht_2(\K(14,2);\Z),
\end{equation}
where $*$ denotes the oriented join of chains, in the displayed factor order.
Then $w,tw,\ldots,t^{12}w$ is an integral basis, the annihilator is $1+t+\cdots+t^{13}$, and
\[
 \rho_*p(t)w=t^5p(t^{-1})w.
\]
The three factor-kernel lattices have the quotient in \eqref{eq:index14}.
Under $w'=t^3w$, the reflection becomes $\rho_*p(t)w'=t^{-1}p(t^{-1})w'$.
Consequently this integral $D_{14}$ homology representation agrees with that of the circular-arc nerve $N(28,14)$ after the latter factors through $D_{14}$.
This comparison asserts no equivariant homotopy equivalence between the complexes.
\end{proposition}
\begin{proof}
The six vertices in \eqref{eq:cross14} induce exactly the three edges $\{0,1\},\{4,5\},\{8,9\}$.
Their independence complex is a join of three zero-spheres, whose fundamental cycle is the displayed eight-term chain.
This is a classical cross-cycle construction \cite[Section~2]{SpecialCycles}.
The chain $z$ equals the first critical-face lift in Appendix~\ref{app:cells}.
The complete matching has one vertex and thirteen two-cells, so gives the ordinary wedge $\bigvee^{13}S^2$; that type also follows from the prior cycle-power recurrence \cite{CyclePowers}.
The face counts are $(14,63,98,35)$.
The common projection and lift recurrences verify all thirteen cycles, the entire projection on the $98$ triangles, and its vanishing on all $35$ tetrahedron boundaries.
Thus Lemma~\ref{lem:basis} applies.
Signed translation gives a Krylov basis with determinant $-1$ and the relation $\sum_{j=0}^{13}t^jw=0$, as in Proposition~\ref{prop:cyclic}.
Signed vertex negation gives $\rho_*w=t^5w$; semilinearity extends this over the orbit basis.
Both geometric signed maps are checked on every simplex and on every homology column.

For the gluing, take $(f_1,f_2,f_3)=(\Phi_2,\Phi_7,\Phi_{14})$ in Lemma~\ref{lem:multifactor}.
Their product is $1+t+\cdots+t^{13}$.
In its monomial basis form the $13$ by $13$ matrix $E$ whose columns are the coefficients of $t^jg_i$, in order $i=1,2,3$ and $0\le j<\deg f_i$.
These are the embeddings of the three kernel bases.
Exact integer elimination gives $|\det E|=448$, $14E^{-1}$ integral, and
\[
 \operatorname{rank}_{\F_2}E=7,\qquad\operatorname{rank}_{\F_7}E=12.
\]
The determinant and exponent bound restrict the quotient to elementary $2$- and $7$-primary parts; the two ranks give their multiplicities six and one.
This proves \eqref{eq:index14}.
This quotient is obtained from the kernel-embedding matrix $E$, independently of a cokernel calculation for the Chinese remainder reduction map $L_{14}\to\bigoplus_{d=2,7,14}\Z[t]/(\Phi_d)$.

The change $w'=t^3w$ gives $\rho_*w'=t^{5-6}w'=t^{-1}w'$.
Propositions~6.1 and~7.2 of the public nerve preprint \cite{Nerve}, with $(n,k,\ell)=(28,14,1)$, give an integral orbit basis $\alpha_0,\ldots,\alpha_{12}$, relation $\sum_{i=0}^{13}\alpha_i=0$, and actions $\tau\alpha_i=\alpha_{i+1}$, $\rho\alpha_i=\alpha_{-i-1}$.
The map $\alpha_i\mapsto t^iw'$ is therefore an integral isomorphism intertwining both generators.
The nerve homology rotation has order fourteen, so its $D_{28}$ action factors through $D_{14}$ with those generators.
This is an already known abstract dihedral representation, here realized by a cycle-square cross-cycle.
\end{proof}

In the original orbit generator, the reflection traces are $\operatorname{tr}(\rho_*t^k)=(-1)^{k+1}$ for $0\le k<14$.
All fourteen values agree with the reduced signed simplicial Lefschetz sums in degree two.
For a monomial shift $w_a=t^aw$, the twist exponent is $5-2a$ modulo $14$, so no such shift changes the reflection to positive pure polynomial inversion.
This is a statement about these orbit generators, not a classification under arbitrary units of $L_{14}$.
The same odd-twist phenomenon is already present in the nerve representation just compared.

\begin{proof}[Proof of Theorem~\ref{thm:main}]
Proposition~\ref{prop:finite} and the CW arguments following it give the homotopy types and free integral homology groups.
Proposition~\ref{prop:lowrank} identifies all the rank-one and rank-two modules, as applied in Section~\ref{sec:actions}.
Propositions~\ref{prop:chainchecks} and \ref{prop:cyclic}, with Lemma~\ref{lem:basis} and \eqref{eq:cyclic}, identify the four higher-rank geometric actions as $L_{14},L_3,L_2,L_3$ and give their orders.
Lemma~\ref{lem:gluing} and its two applications prove \eqref{eq:index3} and \eqref{eq:index2}; Lemma~\ref{lem:multifactor} and Proposition~\ref{prop:composite} prove \eqref{eq:index14}.
Over $\Q$ the factors are coprime, giving the stated rational decompositions.
For each corresponding split integer module, the indicated factor-annihilator kernels sum to the entire lattice.
Their index is preserved by a module isomorphism, so none of the three glued lattices is isomorphic to the displayed split module.
\end{proof}

Rational characters determine the rational decompositions but cannot distinguish these lattices from the corresponding split lattices.
For example, on either realization of $L_3$,
\[
 \tr(t^k)=
 \begin{cases}
 7,&9\mid k,\\
 -2,&3\mid k,\ 9\nmid k,\\
 1,&3\nmid k.
 \end{cases}
\]
The same traces hold on $\Z\oplus\Z[\zeta_9]$.
The quotients \eqref{eq:index3}--\eqref{eq:index2} retain information omitted by the character.

\section{Geometric reflections on the glued lattices}\label{sec:dihedral}
The same signed-chain construction retains more than the cyclic action.
Let $\rho$ be the graph automorphism $x\mapsto-x$ modulo $n$.
The group notation $D_m$ follows the convention fixed in the introduction.
Integral dihedral actions on circular-arc nerve complexes were computed in \cite[Section~7 of the public preprint]{Nerve}.
Here we determine the geometric reflection on the three odd-degree higher-rank cases $(18,4),(18,5),(27,5)$ already identified, using the same cycles and coordinates.

\begin{theorem}\label{thm:dihedral}
For $(18,4),(18,5),(27,5)$ respectively, let $w=[\sum_jv_jz_j]$ with $v=e_2,e_1,e_2$ as in Proposition~\ref{prop:cyclic}, and put $w'=t^7w$.
Under the cyclic-module isomorphism sending $1$ to $w'$, the geometric reflection is
\[
\begin{array}{ccl}
\Ht_3(\K(18,4);\Z)\cong L_3:&\quad&\rho_*p(t)=p(t^{-1}),\\
\Ht_3(\K(18,5);\Z)\cong L_2:&&\rho_*p(t)=p(t^{-1}),\\
\Ht_5(\K(27,5);\Z)\cong L_3:&&\rho_*p(t)=-p(t^{-1}).
\end{array}
\]
These formulas specify integral $D_{18},D_{18},D_{27}$ modules respectively.
The two rank-seven homology actions factor through $D_9$.
As $D_9$ modules with the indicated generators, they are nonisomorphic even after tensoring with $\Q$.
Every reflection has trace $1$ in the first two cases and $-1$ in the third.
\end{theorem}

\begin{proof}
Both defining polynomials have constant term $-1$, so $t$ is a unit in $L_3$ and $L_2$.
They satisfy $t^r f(t^{-1})=-f(t)$, for $r=7,11$ respectively.
Thus substitution $p(t)\mapsto p(t^{-1})$ is a well-defined involution on each quotient.

Apply $\rho$ to every vertex of an oriented simplex and multiply by the sign required to put the image in increasing order, as in \eqref{eq:transport}.
Lemma~\ref{lem:basis} gives the matrix with columns $P(\rho z_j)$.
For the generators fixed above, the resulting exact identities are
\begin{equation}\label{eq:reflectiontwists}
\rho_*w=t^5w\quad(18,4),\qquad
\rho_*w=t^{14}w\quad(18,5),\qquad
\rho_*w=-t^5w\quad(27,5).
\end{equation}
These are obtained using the projection and cycle-lift recurrences \eqref{eq:projection}--\eqref{eq:liftweight} on the critical-face orders in Appendix~\ref{app:cells}.
An independent check computes the signed simplicial reflection, verifies that it commutes with every boundary map, and compares all columns $P(\rho z_j)$.
In the respective $[z_j]$ coordinates, the shifted generators are
\[
\begin{array}{ll}
(18,4):&(-1,0,-1,0,0,0,0)^{\mathsf T},\\
(18,5):&(0,0,0,0,-1,0,-1,1,0,0,0)^{\mathsf T},\\
(27,5):&(1,0,1,0,0,0,0)^{\mathsf T}.
\end{array}
\]
Multiplication by $t^7$ is an integral automorphism, so these generators still have unimodular orbit bases.
The relation $\rho_*t=t^{-1}\rho_*$ extends \eqref{eq:reflectiontwists} over the entire cyclic basis.
For $w'=t^7w$, it gives $\rho_*w'=t^{5-14}w'=w'$ in the first case, $\rho_*w'=t^{14-14}w'=w'$ in the second, and $\rho_*w'=-t^{5-14}w'=-w'$ in the third, using $t^9=1$ on $L_3$.
This proves the displayed semilinear formulas.
The geometric relations also give $\rho_*^2=1$ and $\rho_*t\rho_*=t^{-1}$; no invariance of the Morse matching is required.

To compute reflection traces, work over $\Q$.
On $L_3\otimes\Q=\Q\oplus\Q(\zeta_9)$, inversion has trace one on the first factor and zero on the cyclotomic factor, whose primitive-root eigenspaces are paired by inversion.
On $L_2\otimes\Q=\Q[C_9]\oplus\Q(\zeta_6)$, inversion fixes one vector of the usual $C_9$ basis and pairs the other eight; the cyclotomic contribution again has trace zero.
The same calculation applies to inversion composed with any rotation: on the $C_9$ basis the index map has the form $j\mapsto a-j$, with exactly one fixed index, and on the nontrivial cyclotomic eigenspaces it exchanges pairs.
The negative sign reverses every reflection trace for $\K(27,5)$.
As an independent check, for every reflection $g$ the reduced signed simplicial Lefschetz sum satisfies $\widetilde L(g)=(-1)^q\tr(g_*|\Ht_q)$, for all $18,18,27$ reflections respectively.
The sign is negative in each of the three odd homological degrees.
Since $t$ has order nine on the rank-seven lattices, their homology representations factor through the quotient sending $D_{18}$ or $D_{27}$ to $D_9$ with the specified generators.
Their traces on $\rho$ are $1$ and $-1$, proving rational, hence integral, nonisomorphism.
\end{proof}

The two rank-seven cases therefore agree as cyclic lattices, with the same index-three gluing, but differ as geometric dihedral representations.
The factorization through $D_9$ concerns homology only.
This theorem asserts neither a quotient action on the complexes nor an equivariant homotopy equivalence.

\section{Odd-vertex rank-four family comparison}\label{sec:rankfour}
The two occurrences of the rank-seven lattice in Theorem~\ref{thm:dihedral} are distinguished by rational reflection traces.
A classical even-degree family exhibits a complementary distinction: the rational dihedral representations agree, but their integral lattices do not.
The family below uses the earlier circular-graph classification \cite{Adamaszek,Nerve}, recalled in Theorem~4.3 of \cite{RipsCircle}, and the cross-polytopal maps and integral bases of Adamaszek and Adams \cite[Lemma~8.2 and Proposition~8.4]{RipsCircle}.
It provides a comparison for the preceding finite action calculations.

\begin{proposition}\label{prop:rankfour}
For $\ell\ge0$, put $n=5(2\ell+1)$.
Then
\[
 \K(n,3)\simeq\bigvee^4 S^{2\ell},\qquad
 \Ht_{2\ell}(\K(n,3);\Z)\cong L_5:=\Z[t]/(\Phi_5(t)),
 \quad \Phi_5(t)=1+t+t^2+t^3+t^4,
\]
with $t$ acting by the original one-step vertex translation.
All other reduced homology groups vanish.
There is a cyclic generator $w'$ for which the original reflection acts by
\begin{equation}\label{eq:rankfourreflection}
 \rho_*\bigl(f(t)w'\bigr)=(-1)^{\ell+1}f(t^{-1})w'.
\end{equation}
The homology action factors through $D_5$.
The two resulting integral $D_5$-module types, according to the parity of $\ell$, are nonisomorphic, although their rationalizations are isomorphic.
Their $C_5$-coinvariants are $\Z/5\Z$, with opposite reflection actions.
\end{proposition}

\begin{proof}[Proof of Proposition~\ref{prop:rankfour}]
We first transport the graph and the generators, then use the classical integral basis to compute reflection.
For $\ell=0$, the graph is complete on five vertices and its independence complex consists of five points.
Its reduced degree-zero homology is the augmentation ideal of the vertex permutation module.
Writing $v=e_0-e_1$, the four successive translates $v,tv,t^2v,t^3v$ are the edge differences of a spanning path, hence an integral basis, with $\Phi_5(t)v=0$.
Negation gives $\rho v=-t^{-1}v$; for $w'=t^2v$, one has $\rho w'=-w'$.
This proves the action formula at $n=5$; the coinvariant argument below applies unchanged.
For the rest of the geometric construction suppose $\ell\ge1$.
Let $u=(n+1)/2$.
Multiplication by $u$ is a permutation of $\Z/n\Z$, with inverse multiplication by $2$.
It sends the forbidden differences $\pm1,\pm3$ to the four differences of cyclic distance $(n-1)/2$ or $(n-3)/2$.
The remaining differences have cyclic distance at most $(n-5)/2=5\ell$.
Hence
\[
 \K(n,3)\cong\operatorname{Cl}(C_n^{5\ell}).
\]
The rotation $\tau$ becomes $\sigma^u$, where $\sigma(x)=x+1$ in the clique coordinates, and $\rho$ remains vertex negation.
In particular $u\equiv3\pmod5$.
The classical classification \cite{Adamaszek,Nerve}, as recalled in Theorem~4.3 of \cite{RipsCircle}, gives the displayed ordinary homotopy type at winding fraction $\ell/(2\ell+1)$.

We recall the integral basis in a form that also specifies its chain representatives.
In clique coordinates, let
\[
 F_a=[a,a+5,\ldots,a+10\ell].
\]
Vertices in these displayed lists are reduced modulo $n$; the displayed order specifies their orientation.
Each $F_a$ is a maximal simplex, so the cochain $\beta_a=F_a^\vee$ is a cocycle.
For $a<b<a+5$, set $A_j=a+5j$, $B_j=b+5(j+\ell)$ modulo $n$, and define
\begin{equation}\label{eq:rankfourcycle}
 z_{a,b}=
 \sum_{\epsilon\in\{0,1\}^{2\ell+1}}
 (-1)^{\epsilon_0+\cdots+\epsilon_{2\ell}}
 [V_{0,\epsilon_0},\ldots,V_{2\ell,\epsilon_{2\ell}}],
 \qquad V_{j,0}=A_j,\quad V_{j,1}=B_j.
\end{equation}
The antipodal pairs $(A_j,B_j)$ describe the cross-polytope from Lemma~8.2 of \cite{RipsCircle}: the $A_j$ are its even-indexed vertices, and $B_j$ is the antipode of $A_j$.
The all-$B$ list is a cyclic shift of the displayed $F_b$; a cyclic shift of $2\ell+1$ vertices preserves orientation, so its coefficient is $-1$.
To express these displayed orientations in the global increasing-order chain basis, multiply each term by its sorting sign.
Every term is a simplex, and cancellation of each omitted pair gives $\partial z_{a,b}=0$.
Its two terms supported in a single residue class modulo $5$ are $F_a-F_b$.
Write $\alpha_{a,b}=[z_{a,b}]$.
Proposition~8.4 of \cite{RipsCircle} gives the basis
$\alpha_{1,5},\ldots,\alpha_{4,5}$, dual to $\beta_1,\ldots,\beta_4$.
Equivalently, their pairing is the identity, and the known free rank four makes Lemma~\ref{lem:basis} applicable.
This is an integral basis.

Identify these classes with the augmentation ideal
\[
 I_5=\left\{\sum_{i=0}^4 c_i e_i:\ c_i\in\Z,\ \sum_i c_i=0\right\}
 \subset\Z^5
\]
by $\alpha_{a,b}\mapsto e_{a\bmod5}-e_{b\bmod5}$.
The relations in the proof of Proposition~8.4 of \cite{RipsCircle} justify this identification for all permitted $a,b$.
Rotation sends $e_i-e_j$ to $e_{i+1}-e_{j+1}$ in clique coordinates.
Negation reverses the displayed ordering of $F_a$, up to a cyclic shift.
Reversal of $2\ell+1$ vertices has sign $(-1)^{\ell(2\ell+1)}=(-1)^\ell$, whereas the cyclic shift has sign $+1$.
Mixed-residue terms remain mixed under both symmetries, so the maximal-face cochains give
\[
 \rho_*(e_i-e_j)=(-1)^\ell(e_{-i}-e_{-j}).
\]
Thus this formula computes the geometric reflection, including its orientation sign.

Pull $\alpha_{0,1}$ back to the original graph and call the resulting class $w$.
Since $\tau$ shifts the residue index by $u\equiv3\pmod5$, the four classes
$w,tw,t^2w,t^3w$ correspond to the edge differences on
$\{0,1\},\{3,4\},\{1,2\},\{4,0\}$.
These edges form a spanning tree on five vertices; their differences form an integral basis of $I_5$.
Also $\Phi_5(t)w=0$, and $\rho_*w=(-1)^{\ell+1}t^3w$.
For $w'=t^4w$, the relation $\rho_*t=t^{-1}\rho_*$ and $t^5=1$ give
$\rho_*w'=(-1)^{\ell+1}w'$.
This proves the cyclic-module identification and \eqref{eq:rankfourreflection}.
The factorization through $D_5$ is at the homology level.

Let $L_5^+$ and $L_5^-$ denote the two choices of sign in \eqref{eq:rankfourreflection}.
Their coinvariant group is
\[
 L_5/(t-1)L_5=\Z[t]/(\Phi_5,t-1)\cong\Z/5\Z.
\]
Reflection induces $+1$ and $-1$, respectively, on this quotient.
No automorphism of $\Z/5\Z$ intertwines these two actions, proving integral nonisomorphism.
Multiplication by $a=t-t^{-1}$ commutes with $t$ and anticommutes with polynomial inversion.
It therefore intertwines $L_5^+\otimes\Q$ and $L_5^-\otimes\Q$.
Its determinant is
\[
 \prod_{j=1}^4(\zeta_5^j-\zeta_5^{-j})
 =\prod_{j=1}^4(1-\zeta_5^j)=\Phi_5(1)=5.
\]
It is consequently invertible over $\Q$, but not over $\Z$.
\end{proof}

The two finite inputs added to the comparison are as follows.
The face and pair counts use the same conventions as Proposition~\ref{prop:finite}.
\begin{center}
\begin{tabular}{@{}lrrcl@{}}
\toprule
Complex & Faces & Ordinary pairs & Degree & Reflection in the $w'$ basis\\
\midrule
$\K(15,3)$ &486&240&$2$& polynomial inversion\\
$\K(25,3)$ &31206&15600&$4$& negative polynomial inversion\\
\bottomrule
\end{tabular}
\end{center}
In each case the ordinary matching has one critical vertex and four critical cells in the indicated positive degree.
Complete face enumeration and a reversed-Hasse acyclicity check independently reproduce these Morse models.
The signed rotation and reflection commute with every augmented simplicial boundary.
Pulling \eqref{eq:rankfourcycle} and the maximal-face cochains back by multiplication by $2$ gives the verified integral pairing and action coordinates.
In each case the normalized orbit basis $w',tw',t^2w',t^3w'$ is unimodular relative to the pulled-back ordered classical basis $\alpha_{1,5},\ldots,\alpha_{4,5}$.
Both cyclic matrices are the companion matrix of $\Phi_5$, and the two reflection matrices differ by a sign.
Their rational characters agree; the coinvariant reflection separates the integral lattices.

The abstract types themselves are already present in the nerve family.
For $N(5(\ell+1),5\ell)$, Propositions~6.1 and~7.2 of the public preprint \cite{Nerve} give the same cyclic lattice with
$\rho\alpha_i=(-1)^{\ell+1}\alpha_{-i-1}$.
Choosing the cyclic generator $t^2\alpha_0$ makes reflection equal to $(-1)^{\ell+1}$ times inversion.
This agreement concerns the common $D_5$ homology quotient; it asserts no equivariant homotopy equivalence between the nerve and the independence complex.
The family comparison is a consequence of the classical models, with a direct signed-chain realization on the two specified graphs.

\section{Even-vertex family via a natural nerve reduction}\label{sec:evenfamily}
The odd-vertex comparison has a bipartite counterpart for every even vertex count at least ten.
The ordinary types are already the $d=4$ case of \cite[Theorem~1.2]{NS}.
We derive the actual one-step actions from classical integral nerve representations.
The reduction transports the part-preserving symmetries; an integral centralizer and the signed trace recover the remaining generator.

\begin{proposition}\label{prop:evenfamily}
Let $m\ge5$, write $m=5q+a$ with $0\le a\le4$, and put $X_m=\K(2m,3)$.
Let $T$ and $R$ act on integral homology by the vertex maps $x\mapsto x+1$ and $x\mapsto-x$, respectively.
If $a=0$, then $X_m\simeq\bigvee^4 S^{2q-1}$ and
\[
 \Ht_{2q-1}(X_m;\Z)\cong\Z[t]/(\Phi_{10}(t)),
 \qquad \Phi_{10}(t)=t^4-t^3+t^2-t+1,
\]
where $T$ acts as $t$ and an integral cyclic generator $w$ can be chosen so that
\[
 R\bigl(f(t)w\bigr)=(-1)^q f(t^{-1})w.
\]
This homology action factors through $D_{10}$.
The two parity choices give nonisomorphic integral $D_{10}$-modules with isomorphic rational representations.
If $a>0$, then $X_m\simeq S^{2q}$ and
\[
 \Ht_{2q}(X_m;\Z)\cong\Z,\qquad T=-1,\quad R=(-1)^q.
\]
In both cases all other reduced integral homology groups vanish.
\end{proposition}

\begin{proof}[Proof of Proposition~\ref{prop:evenfamily}]
We first identify a natural integral suspension isomorphism for $T^2$ and $R$.
The known nerve basis then identifies the integer centralizer of $T^2$; the reduced Lefschetz number chooses the correct square root for $T$.

\smallskip\noindent\emph{The ordered cover and its natural actions.}
Set $A=\{a_i=2i\}$ and $B=\{b_j=2j+1\}$, with indices modulo $m$.
Both parts are independent, and the neighbors of $a_i$ are
$b_{i-2},b_{i-1},b_i,b_{i+1}$.
Let $U=\Delta(B)$ be the full simplex on $B$ and let
\[
 V=\operatorname{SC}_{X_m}(A)
   =\bigcup_{i\in\Z/m\Z}\operatorname{st}_{X_m}(a_i),
\]
where $\operatorname{st}$ denotes the closed simplicial star.
Every face containing a vertex of $A$ belongs to $V$, so $X_m=U\cup V$.
Both subcomplexes are contractible by the star-cluster lemma \cite[Lemma~3.2 of the public preprint]{Barmak}.
Indeed, any nonempty intersection of the stars is the closed star of the corresponding face of $A$, hence a cone, and the cover nerve is a simplex.
The maximal faces of $U\cap V$ are the cyclic blocks
\[
 \{b_{i+2},b_{i+3},\ldots,b_{i+m-3}\}.
\]
Each has $m-4$ vertices; with $b_j$ labeled $j$, this is exactly $N(m,m-5)$, whose maximal faces are the cyclic intervals $[j,j+m-5]$.
For $m=5$ it consists of five isolated vertices.
The classical bipartite suspension construction is due to Jonsson; Barmak recalls it in Theorem~3.7 and gives the two-contractible-subcomplex argument in Lemma~3.3 of the public preprint \cite{Barmak}.
It yields the ordinary equivalence $X_m\simeq\Sigma N(m,m-5)$.
For the fixed ordered cover $(U,V)$, reduced integral Mayer--Vietoris gives an isomorphism
\[
 \delta:\Ht_d(X_m;\Z)\longrightarrow\Ht_{d-1}(N(m,m-5);\Z).
\]
An automorphism preserving both parts preserves each cover member, so naturality makes $\delta$ equivariant for such maps.
On the $B$ labels, $T^2$ is $s:j\mapsto j+1$ and $R$ is $s^{-1}e:j\mapsto-j-1$, where $e:j\mapsto-j$.
The one-step translation exchanges the parts, so we recover it separately.

\smallskip\noindent\emph{The rank-four lattice and the missing generator.}
Suppose $m=5q$.
Then $(m-5)/m=(q-1)/q$.
Theorem~3.6 and Propositions~6.1 and~7.2 of the public nerve preprint \cite{Nerve} give a wedge of four $(2q-2)$-spheres and integral classes $\alpha_i$, indexed modulo $5$, with
\[
 \alpha_0,\ldots,\alpha_3\text{ an integral basis},\qquad
 \sum_{i=0}^4\alpha_i=0,\qquad
 s\alpha_i=\alpha_{i+1},\quad e\alpha_i=(-1)^q\alpha_{-i-1}.
\]
Via $\delta$, identify the homology of $X_m$ with
$O=\Z[z]/(1+z+z^2+z^3+z^4)$, taking $\alpha_0$ to $1$ and $T^2$ to multiplication by $z$.
This is a free rank-one $O$-module over the integers.
Since $T$ commutes with $T^2$, it is multiplication by some $\lambda\in O$ satisfying $\lambda^2=z$.
The ring $O$ is an integral domain, $z^5=1$, and $(z^3)^2=z$; hence
$(\lambda-z^3)(\lambda+z^3)=0$ forces $\lambda=\pm z^3$.
The actual one-step vertex translation fixes no nonempty independent face setwise: its only nonempty invariant vertex subset is the full vertex set.
Thus it has trace zero on every nonnegative simplicial chain group and trace one on the augmentation group, so its reduced Lefschetz number is $-1$.
The homology degree $2q-1$ is odd, giving $\operatorname{tr}(T)=1$.
Multiplication by $z^3$ has cyclotomic trace $-1$, and therefore
\[
 T=-z^3,\qquad z=T^2.
\]
Consequently $\Z[T]=O$ and the minimal polynomial is $\Phi_{10}$.
For completeness, $\Phi_{10}(-z^3)=0$ follows by permuting the five powers of $z$ in their norm relation; the resulting surjection $\Z[t]/(\Phi_{10})\to O$ is an isomorphism between free abelian groups of rank four.

\smallskip\noindent\emph{Reflection normalization and the integral distinction.}
Naturality and the nerve formulas give
\[
 R\alpha_0=(-1)^q z^{-2}\alpha_0,\qquad
 R(f(z)\alpha_0)=(-1)^q f(z^{-1})z^{-2}\alpha_0.
\]
Choose $w=\delta^{-1}(z^4\alpha_0)$.
Since $z^5=1$, one has $Rw=(-1)^q w$.
Inversion takes $T=-z^3$ to $T^{-1}$, proving the stated formula in $t$ and the integral orbit basis $w,Tw,T^2w,T^3w$.
The order of $T$ is $10$, dividing $2m=10q$, which gives the $D_{10}$ quotient on homology.
Moreover,
\[
 H/(T^2-1)H\cong O/(z-1)O\cong\Z/5\Z.
\]
Reflection acts there by $(-1)^q$, distinguishing the two integral types.
Over $\Q$, multiplication by $z-z^{-1}$ is invertible, commutes with $T$ and anticommutes with inversion, so it intertwines the two signs.

\smallskip\noindent\emph{The remaining residues.}
For $1\le a\le4$, the ratio $(m-5)/m$ lies strictly between $(q-1)/q$ and $q/(q+1)$.
Theorem~3.6 and Proposition~7.3 of the public nerve preprint \cite{Nerve} give $N(m,m-5)\simeq S^{2q-1}$, with $s=1$ and $e=(-1)^q$ on its integral homology.
The suspension and naturality give $X_m\simeq S^{2q}$ and $R=(-1)^q$.
On this rank-one group $T=\pm1$; the same reduced Lefschetz number now gives $\operatorname{tr}(T)=-1$, so $T=-1$.
\end{proof}

Together with Proposition~\ref{prop:rankfour}, this gives odd- and even-vertex rank-four comparisons: the actual rotations have cyclotomic types $\Phi_5$ and $\Phi_{10}$, respectively.
Both reflection distinctions come from order-five coinvariants and disappear over $\Q$.
The even-vertex formulas are consequences of the cited suspension and integral nerve theory; they add a uniform action comparison without adding finite computational inputs or an equivariant suspension equivalence.

\section{Odd sphere branches and the complete jump-three comparison}\label{sec:oddgeneric}
The rank-four odd family occurs at the singular ratios in the classical circular-clique classification.
The complementary ratios give single odd-dimensional spheres.
We determine their actual reflection by the orbit trace of Proposition~\ref{prop:orbittrace}, rather than transporting an action through an ordinary homotopy equivalence.

\begin{lemma}\label{lem:linearjump3}
Let $J_L$ be the graph on $\{0,\ldots,L-1\}$ whose edges have difference $1$ or $3$, with $J_0$ the empty graph, and put
\[
 F_L=\sum_{\sigma\in\Ind(J_L)}(-1)^{|\sigma|}.
\]
Then $F_0=1$, $F_1=0$, $F_2=F_3=F_4=-1$, and $F_L=-F_{L-5}$ for $L\ge5$.
In particular $F_{5\ell+r}=(-1)^{\ell+1}$ for $\ell\ge0$ and $2\le r\le5$.
\end{lemma}
\begin{proof}
Write $I_H(x)=\sum_{\sigma\in\Ind(H)}x^{|\sigma|}$ for a graph $H$.
The initial values follow from $I_{J_0}=1$, $I_{J_1}=1+x$, and
$J_2=K_{1,1}$, $J_3=K_{1,2}$, $J_4=K_{2,2}$, using
$I_{K_{a,b}}(x)=(1+x)^a+(1+x)^b-1$.

For $L\ge5$, independent sets avoiding both $1$ and $3$ cancel in the signed sum by toggling $0$.
Every remaining set avoids $0$ and either contains $1$, or avoids $1$ and contains $3$.
Removing the specified vertex in the first class leaves the induced graph
$E=J_L[\{3\}\cup\{5,\ldots,L-1\}]$; in the second class it leaves
$D=J_L[\{5\}\cup\{7,\ldots,L-1\}]$.
All listed vertices outside $\{0,\ldots,L-1\}$ are omitted.
Thus $F_L=-I_E(-1)-I_D(-1)$.
In $E$, the only possible neighbor of $3$ is $6$.
Deletion of $3$ leaves $J_{L-5}$, up to translation, and deletion of its closed neighborhood leaves $D$.
The elementary deletion identity
\[
 I_E(x)=I_{E-3}(x)+xI_{E-N_E[3]}(x)
\]
therefore gives $I_E(-1)=F_{L-5}-I_D(-1)$ and proves the recurrence, including $L=5,6$ when $6$ is absent.
Since $F_5=-F_0=-1$, iteration gives the final formula.
\end{proof}

\begin{proposition}\label{prop:oddgeneric}
Let $n\ge7$ be odd and not divisible by $5$.
There is a unique integer $\ell\ge0$ with $10\ell+5<n<10\ell+15$.
Then
\[
 \K(n,3)\simeq S^{2\ell+1},\qquad
 \Ht_{2\ell+1}(\K(n,3);\Z)\cong\Z,
 \qquad \tau_*=1,\quad\rho_*=(-1)^{\ell+1}.
\]
All other reduced integral homology groups vanish.
\end{proposition}
\begin{proof}
The unit relabeling $u=(n+1)/2$ used in Proposition~\ref{prop:rankfour} gives
\[
 \K(n,3)\cong\operatorname{Cl}(C_n^{(n-5)/2}),
 \qquad\tau\mapsto\sigma^u,\quad\rho\mapsto(x\mapsto-x).
\]
The inequalities in the statement are exactly
\[
 \frac{\ell}{2\ell+1}<\frac{n-5}{2n}<\frac{\ell+1}{2\ell+3}.
\]
The classical clique classification \cite[Corollary~6.7]{Adamaszek}, also stated as Theorem~8.2 of the public preprint \cite{Nerve} and Theorem~4.3 of \cite{RipsCircle}, proves the ordinary sphere type.
The integral rotation on its rank-one homology is $\pm1$; its odd order forces $\tau_*=1$.

Put $n=2h+1$.
The reflection orbits are $\{0\}$ and $\{i,-i\}$ for $1\le i\le h$.
The last two orbits, $i=h,h-1$, contain edges and give looped quotient vertices: their internal differences are $n-1$ and $n-3$.
There are no other looped vertices.
After removing these forbidden vertices, label the remaining orbits by $0,\ldots,h-2$.
For distinct labels $i<j$, quotient adjacency means that $j-i$ or $i+j$ is congruent to one of $\pm1,\pm3$ modulo $n$.
In these ranges the only possibilities are $j-i\in\{1,3\}$ or $i+j\in\{1,3\}$.
The latter produces only $(0,1)$, $(0,3)$ and $(1,2)$, already edges of the former kind whenever their labels occur.
Thus the usable quotient graph is $J_L$, with $L=h-1=(n-3)/2$.

Write $n=10\ell+r$ with $r\in\{7,9,11,13\}$.
Then $L=5\ell+(r-3)/2$, and Lemma~\ref{lem:linearjump3} gives $F_L=(-1)^{\ell+1}$.
Proposition~\ref{prop:orbittrace} identifies the reduced Lefschetz number of the actual reflection with
\[
 \widetilde\chi(\Ind(J_L))=-F_L=(-1)^\ell.
\]
The sole homology degree is odd, so $\tr(\rho_*)=(-1)^{\ell+1}$, proving the claimed integral action.
\end{proof}

These statements give a complete action comparison for every $n\ge5$.
For $n=6,8$, the graph is $K_{3,3}$ or $K_{4,4}$; the independence complex is the disjoint union of the two part simplices.
It is homotopy equivalent to $S^0$, with rotation exchanging components and reflection preserving them.
Let $\iota$ denote polynomial inversion $f(t)\mapsto f(t^{-1})$ in the normalized rank-four cyclic bases.
The nonzero reduced integral homology of $\K(n,3)$ and its geometric actions are summarized below, with integer parameters $\ell\ge0$ and $q\ge1$; all omitted degrees vanish.
\begin{center}
\begin{tabular}{@{}lccl@{}}
\toprule
Vertex count & Degree & Rotation module & Reflection\\
\midrule
$5(2\ell+1)$, $\ell\ge0$ &$2\ell$ &$\Z[t]/(\Phi_5)$ &$(-1)^{\ell+1}\iota$\\
odd $10\ell+5<n<10\ell+15$ &$2\ell+1$ &$\Z$, $t=1$ &$(-1)^{\ell+1}$\\
$10q$, $q\ge1$ &$2q-1$ &$\Z[t]/(\Phi_{10})$ &$(-1)^q\iota$\\
$10q+2a$, $q\ge1$, $1\le a\le4$ &$2q$ &$\Z$, $t=-1$ &$(-1)^q$\\
$6,8$ &$0$ &$\Z$, $t=-1$ &$1$\\
\bottomrule
\end{tabular}
\end{center}
The table assembles the cited ordinary classifications, classical integral bases, natural Mayer--Vietoris and elementary signed traces.
The action formulas follow from these classical models and the signed trace calculation.

\section{Two further cycle-square lattices}\label{sec:cyclesquare}
The Morse-cycle reconstruction of Sections~\ref{sec:cw} and~\ref{sec:certification} also applies when the ordinary matching uses a specified vertex priority instead of increasing vertex order.
These examples extend the factor-kernel comparison to three odd-order factors and to an order-eight constituent.

\begin{proposition}\label{prop:cyclesquare}
For one-step translation on the indicated reduced homology groups, there are integral cyclic generators $w$ giving
\[
 \begin{array}{c|c|c|c}
 K & q & H=\Ht_q(K;\Z) & \rho_*w=h(t)w\\ \hline
 \K(15,2)&2&\Z[t]/((t^5-1)\Phi_3(t))&h(t)=t^6-t^3-t\\
 \K(16,2)&3&\Z[t]/((t-1)\Phi_8(t))&h(t)=t^2.
 \end{array}
\]
Here $\Phi_3=t^2+t+1$ and $\Phi_8=t^4+1$.
Reduced homology is zero in other degrees, and the ordinary homotopy types are $\bigvee^7 S^2$ and $\bigvee^5 S^3$, respectively.
Their translation actions have orders $15$ and $8$; the latter is inflated from $C_8$ to $C_{16}$.
With $U_d=\ker(\Phi_d(t):H\to H)$ and $\Phi_1=t-1$, the factor-kernel quotients are
\[
 \Ht_2(\K(15,2);\Z)/(U_1\oplus U_3\oplus U_5)\cong\Z/15\Z,
 \qquad
 \Ht_3(\K(16,2);\Z)/(U_1\oplus U_8)\cong\Z/2\Z.
\]
On either cyclic lattice, reflection is the semilinear map $p(t)w\mapsto h(t)p(t^{-1})w$.
For $\K(16,2)$, replacing $w$ by $tw$ gives pure polynomial inversion.
\end{proposition}
\begin{proof}
Use the face order and vertex-stage matching of Section~\ref{sec:cw}, replacing its vertex order by
\[
 \begin{split}
 \K(15,2):\ &0,3,6,9,12,1,4,7,10,13,2,5,8,11,14,\\
 \K(16,2):\ &0,7,13,4,15,12,6,9,14,10,3,5,8,11,2,1.
 \end{split}
\]
The respective face counts, including the empty face, are $309$ and $453$.
The complete ordinary matchings have $150$ and $223$ pairs, a critical vertex $0$, and the following critical faces in lexicographic order:
\[
 \begin{split}
 \K(15,2):\ &(1,4,10),(1,4,11),(1,5,10),(1,8,13),\\
             &(2,5,11),(2,7,10),(4,7,14),\\
 \K(16,2):\ &(1,5,8,12),(1,5,10,14),(2,5,11,14),\\
             &(2,6,11,14),(3,8,11,15).
 \end{split}
\]
Topological orders verify all $885$ and $1392$ directed nonempty Hasse incidences, respectively.
Forman's theorem therefore gives the stated ordinary wedges and free homology ranks.
The ordinary types also follow from the earlier cycle-power splitting theorem \cite{CyclePowers}.

Let $z_1,\ldots,z_r$ be the critical-face lifts determined by \eqref{eq:projection}--\eqref{eq:liftweight} with these matchings.
Their support sizes are $(8,8,8,10,8,8,8)$ and $(16,16,57,16,16)$.
The projection is defined on all $140$ simplices in degree $q$ in either case and kills all $75$ or $16$ incoming simplex boundaries.
Every lift has zero boundary and projects to its specified standard basis vector.
Thus Lemma~\ref{lem:basis} identifies the lifts as an integral homology basis.
In their displayed order, choose the coordinate vectors
\[
 v_{15}=(-1,-1,0,1,0,0,1)^T,\qquad
 v_{16}=(-1,-1,-1,-1,1)^T,
\]
and put $w=[\sum_j v_jz_j]$ in the corresponding homology group.
The signed geometric translation and negation give the following integer identities:
\[
 \begin{array}{c|c|c|c}
 n&r&\det(v,Tv,\ldots,T^{r-1}v)& f(t)\text{ with }f(T)=0\\ \hline
 15&7&-1&(t^5-1)\Phi_3(t)\\
 16&5&-1&(t-1)\Phi_8(t).
 \end{array}
\]
In the orbit bases, the first reflection columns are the coefficients of $t^6-t^3-t$ and $t^2$, respectively.
These finite identities are determined by the graph, matching, critical lists, recurrences and vectors just specified.
Independent checks enumerate the entire face set from edge differences, verify every directed Hasse incidence, every projection and lift identity, every signed simplicial chain-map identity and every homology action column, and check the orbit determinants and polynomial relations by rational elimination.
Equation~\eqref{eq:cyclic} then proves the module identifications.
The rational cyclotomic factors give the exact action orders.
The relation $\rho\tau\rho=\tau^{-1}$ gives the reflection formula on every orbit vector.
Both defining polynomials are reciprocal up to sign and have unit constant term, so substitution by $t^{-1}$ is well-defined on their integer quotient rings.
The geometric identity $\rho_*^2=1$ gives $h(t)h(t^{-1})=1$ in each quotient; thus the displayed multiplier is a unit.
For the second case, $\rho_*(tw)=t^{-1}t^2w=tw$.

For the first gluing quotient use Lemma~\ref{lem:multifactor} with $(f_1,f_2,f_3)=(\Phi_1,\Phi_3,\Phi_5)$.
Put $A=\Phi_3$ and $B=\Phi_5$.
The identity $(1+t^3)A-tB=1$ implies that
\[
 (AB,(t-1)B,(t-1)A)=(t-1,15):
\]
the left ideal contains $t-1$, and reduction modulo $t-1$ takes $AB$ to $A(1)B(1)=15$.
Conversely each left generator lies in $(t-1,15)$.
Lemma~\ref{lem:multifactor} therefore gives $\Z/15\Z$ directly.
As an independent integer check, the columns $t^j f/f_i$, $0\le j<\deg f_i$, give a kernel-embedding matrix $E$ with $|\det E|=15$, $15E^{-1}$ integral, and ranks six over $\F_3$ and $\F_5$.
For the second quotient, Lemma~\ref{lem:gluing} gives directly
\[
 \Z[t]/(t-1,t^4+1)=\Z/2\Z.
\]
These are factor-kernel quotients inside the integral homology lattices, rather than just rational factor decompositions.
\end{proof}

The ordinary cycle-square types, induced-matching cross-cycles and supersymmetric cohomology descriptions are prior work \cite{CyclePowers,SpecialCycles,HuijseSchoutens,HuijseLadders,HuijseThesis}.
The proposition records the integer translation and reflection on these two graphs, including an integral orbit generator and the gluing quotients.

\section{Two-generator lattices and a recurring summand}\label{sec:higherorbit}
A cyclic orbit need not span the full homology lattice.
The next examples require two generators and connect the cycle-square calculations to the lattice $L_3$ already occurring in Theorem~\ref{thm:main}.
The following elementary criterion extends the orbit-basis argument.

\begin{lemma}\label{lem:multiorbit}
Let $H$ be a free abelian group of rank $r$ with an automorphism $T$.
For monic polynomials $f_a$ of degrees $d_a$ with $\sum_a d_a=r$, choose $u_a\in H$ such that $f_a(T)u_a=0$.
If the concatenated columns
\[
 B=(u_1,Tu_1,\ldots,T^{d_1-1}u_1,\ u_2,\ldots,T^{d_2-1}u_2,\ldots)
\]
have determinant $\pm1$ in an integral basis of $H$, then
\[
 H\cong\bigoplus_a\Z[t]/(f_a(t)),\qquad p_a(t)\longmapsto p_a(T)u_a.
\]
Suppose additionally that $T^n=1$ and $RTR=T^{-1}$, $R^2=1$.
If $Ru_a=h_a(T)u_a$ for every $a$, then each displayed summand is $R$-stable and its reflection is $p(t)\mapsto h_a(t)p(t^{-1})$.
For a finite-order $T$, write its rational characteristic polynomial as $\prod_{d\mid n}\Phi_d^{m_d}$.
The minimum number of cyclic generators of $H\otimes\Q$ is $\max_d m_d$.
\end{lemma}
\begin{proof}
The proposed module map is well-defined by the annihilator identities and sends the standard polynomial bases to $B$.
Its determinant proves that it is an isomorphism over $\Z$.
The dihedral relation extends each seed reflection to every polynomial in that seed.
For the last assertion, $t^n-1$ is square-free over $\Q$, so the rational module is a direct sum of $\Q[t]/(\Phi_d)$ vector spaces of dimensions $m_d$.
One cyclic generator contributes at most one dimension over each of these fields.
Conversely the rational Chinese remainder theorem combines one basis vector from each constituent into one generator, and repeating this $\max_d m_d$ times generates the module.
\end{proof}

Write $\mathcal P_n^\epsilon=\Z[t]/(t^n-1)$ with reflection $p(t)\mapsto t^\epsilon p(t^{-1})$.
Its underlying cyclic module is regular; the superscript specifies the reflection, with the geometric generators $\tau,\rho$ kept fixed.
For $L_3$ in the next proposition, reflection is pure inversion as in Theorem~\ref{thm:dihedral}.

\begin{proposition}\label{prop:twoorbit}
Reduced homology of each of $\K(17,2),\K(18,2),\K(19,2)$ is concentrated in degree three.
There are integral $D_n$-module identifications
\[
 \begin{array}{c|c|c}
 K&\text{ordinary homotopy type}&\Ht_3(K;\Z)\\ \hline
 \K(17,2)&\bigvee^{18}S^3&\mathcal P_{17}^0\oplus\Z\\
 \K(18,2)&\bigvee^{25}S^3&\mathcal P_{18}^1\oplus L_3\\
 \K(19,2)&\bigvee^{20}S^3&\mathcal P_{19}^0\oplus\Z .
 \end{array}
\]
The $\Z$ summands have trivial translation and reflection.
Each integral module needs exactly two cyclic generators.
In particular,
\[
 \Ht_3(\K(18,2);\Z)\cong
 \mathcal P_{18}^1\oplus\Ht_3(\K(18,4);\Z)
\]
with the same one-step translation and vertex-negation conventions on both graphs.
This is an isomorphism of homology representations, with no claimed equivariant decomposition of the complexes.
\end{proposition}
\begin{proof}
Use the ordinary nonempty-face matching of Section~\ref{sec:cw} and the projection/lift conventions of Section~\ref{sec:certification}.
For $n=17,18$ the vertex priority is $0,1,\ldots,n-1$; for $n=19$ it is
\[
 2,3,9,0,6,5,8,1,14,7,13,4,10,17,18,11,12,16,15.
\]
The complete face counts, including the empty face, are $664,973,1426$.
The matchings have respectively $322,473,702$ pairs, one critical vertex, and $18,25,20$ critical three-faces.
The critical vertices are $0,0,2$ respectively.
Directed Hasse acyclicity is checked on every one of the $2176,3384,5244$ nonempty incidences.
Forman's theorem gives the displayed ordinary wedges; these ordinary types are also covered by the prior cycle-power classification \cite{CyclePowers}.

Order the critical three-faces lexicographically, and let $z_1,\ldots,z_r$ be their lifts from \eqref{eq:projection}--\eqref{eq:liftweight}.
The critical faces are the unmatched faces of the fully specified matching, so this also fixes the indexing without an external data file.
The projection is defined on all $238,378,570$ three-faces and kills all $51,126,266$ incoming four-face boundaries.
Every lift has zero boundary and projects to its standard basis vector; Lemma~\ref{lem:basis} makes their classes an integral basis.

In each case put $u=[z_1]$.
For $n=18$, put $v=[z_1+z_3+z_8]$.
For $n=17,19$, put $v=[\sum_j b_jz_j]$ with the following coordinate vectors:
\[
\begin{split}
 b_{17}={}&(3,-3,3,-3,3,-3,3,-3,3,-3,\\
           &\hspace{12mm}-2,3,1,-3,1,3,-2,1)^{\mathsf T},\\
 b_{19}={}&(0,0,-1,-1,1,-1,1,1,0,0,\\
           &\hspace{12mm}1,-1,0,-1,3,-2,-1,0,0,1)^{\mathsf T}.
\end{split}
\]
Let $B$ concatenate the $n$ successive translates of $u$ with the first $d$ translates of $v$.
The signed geometric actions satisfy
\[
\begin{array}{c|c|c|c|c|c}
 n&d&\det B&f_v(t)&\rho_*u&\rho_*v\\ \hline
 17&1&1&t-1&t^4u&v\\
 18&7&-1&(t-1)\Phi_9(t)&t^5u&t^4v\\
 19&1&1&t-1&t^{16}u&v .
\end{array}
\]
Here $(t^n-1)u=0$ and $f_v(t)v=0$.
These are full integer identities in the lifted-cycle coordinates: the checker derives every column of both actions from the signed vertex maps, rather than merely testing abstract group relations.
The concatenated orbit columns and their determinants give Lemma~\ref{lem:multiorbit}.
The reflection columns lie in their respective blocks.
Replacing $u$ by $t^2u,t^2u,t^8u$ for $n=17,18,19$ respectively changes its reflection multiplier to $1,t,1$.
Replacing $v$ by $t^2v$ for $n=18$ changes its multiplier to $1$.
For a seed with $Ru=T^k u$, the relation $RT=T^{-1}R$ gives $R(T^a u)=T^{k-2a}(T^a u)$.
These shifts are integral automorphisms, giving the stated dihedral formulas.
The second summand for $n=18$ is therefore the same $L_3$ with inversion already identified on $\K(18,4)$.

Finally the three rational factorizations have multiplicities
\[
 \Phi_1^2\Phi_{17},\qquad
 \Phi_1^2\Phi_2\Phi_3\Phi_6\Phi_9^2\Phi_{18},\qquad
 \Phi_1^2\Phi_{19}.
\]
The last part of Lemma~\ref{lem:multiorbit} excludes a single cyclic generator.
The displayed integer orbit bases supply two, proving minimality.
Complete signed Lefschetz checks independently recover every rotation and reflection trace.
\end{proof}

The proposition locates the earlier glued rank-seven action as a direct summand of a larger geometric action.
Splitting off a regular summand does not split its cyclotomic factors over $\Z$, and the $L_3$ summand retains its index-three factor-kernel quotient.
Thus regular summands, cyclic glued lattices and their geometric reflections belong to the same reconstruction framework.
The ordinary cycle-square types, cross-cycle constructions, ground-state counts and momentum interpretations have substantial precedents \cite{CyclePowers,SpecialCycles,HuijseSchoutens,HuijseLadders,HuijseThesis,HuijseCritical,ChepigaLadder}.

\section{Integer saturation of a classical cross-cycle family}\label{sec:tiling}
The orbit-basis criterion requires an integer basis.
A family of classical cross-cycles can span the full rational homology while missing a primitive integer class.
The next fixed example makes that distinction explicit.
The cross-cycle construction itself is established \cite[Section~1.2 and Proposition~3.1]{JonssonCross}; the additional assertion concerns the integral embedding of the specified family.

For an ordered list of five distinct vertices $a=(a_1,\ldots,a_5)$, write
\[
 c(a)=\bigwedge_{j=1}^5\bigl([a_j]-[a_j+1]\bigr),
\]
where the wedge denotes the oriented join of chains, using the written factor order and increasing simplex orientations after expansion; vertices are taken modulo $22$.
If the circular gaps between successive $a_j$ are $4$ or $5$, the five edges $\{a_j,a_j+1\}$ form an induced matching in $G(22,2)$ and $a$ is a maximal independent set.
Thus these are cross-cycles in the precise sense of the cited construction.
Put $N(t)=1+t+\cdots+t^{21}$, and let $T=\tau_*$, $R=\rho_*$ on $H=\Ht_4(\K(22,2);\Z)$.

\begin{proposition}\label{prop:tilingindex}
The ordinary type is $\K(22,2)\simeq\bigvee^{43}S^4$.
Let
\[
 u=[c(0,4,8,12,17)],\qquad v=[c(0,4,8,13,17)],\qquad
 L=\Z[t]u+\Z[t]v\subset H.
\]
Their forty-four translates comprise the adjacent-pair cross-cycle family with five circular gaps in $\{4,5\}$.
They span $H\otimes\Q$, but
\[
 H/L\cong\Z/2\Z,
\]
with trivial induced $D_{22}$ action.
Their unique primitive integer relation is
\begin{equation}\label{eq:tilingrelation}
 N(T)u+2N(T)v=0.
\end{equation}
The geometric reflections satisfy $Ru=-T^9u$ and $Rv=-T^{13}v$.
For $u'=T^4u$ and $v'=T^6v$, both reflections have multiplier $-t$.
Consequently, as an integral dihedral module,
\begin{equation}\label{eq:tilingsubgroup}
 L\cong\Z[t]/(N(t))\ \oplus\ \Z[t]/(t^{22}-1).
\end{equation}
The two factors have generators $u'+2v'$ and $v'$, respectively; reflection on both is $p(t)\mapsto-tp(t^{-1})$.
Let $z$ be the critical-cell lift specified below. Returning to the original seeds $u,v$, put $w=u+v$ and $e=z-3w$.
These give an isomorphism of integral cyclic modules
\begin{equation}\label{eq:tilingambient}
 H\cong A\oplus P,\qquad A=\Z[t]/(N(t)),\quad P=\Z[t]/(t^{22}-1),
\end{equation}
with generators $e,w$, respectively.
In these coordinates the specified subgroup has the precise position
\begin{equation}\label{eq:tilingparity}
 L=\ker(\epsilon_2:A\longrightarrow\Z/2\Z)\oplus P,
 \qquad \epsilon_2\Bigl(\sum_i a_it^i\Bigr)=\sum_i a_i\pmod2.
\end{equation}
The minimum number of cyclic generators of $H$ is exactly two.
Equation~\eqref{eq:tilingambient} specifies cyclic coordinates; Proposition~\ref{prop:tilingdihedral} gives compatible ambient dihedral coordinates.
\end{proposition}
\begin{proof}
For this example use the following face-outer matching, rather than the vertex-outer rule of Section~\ref{sec:cw}.
List all faces, including the empty face, by cardinality and lexicographic order.
For each available face $\sigma$, inspect vertices $21,20,\ldots,0$ and pair it with the first available face $\sigma\cup\{x\}$, if one exists with $x\notin\sigma$; mark both unavailable.
Remove the first pair $(\varnothing,\{21\})$ to pass to ordinary chains.
There are $4489$ faces including the empty face, $2222$ ordinary pairs, one critical vertex $21$ and forty-three critical four-faces.
Acyclicity is checked on all $19162$ nonempty Hasse incidences.
Forman's theorem gives the ordinary wedge, already covered by the cycle-power classification \cite[Theorem~1.1, specialized to $r=2$]{CyclePowers}.

Use the projection and lift rules of Section~\ref{sec:certification}, with these matched incidences.
Order the critical four-faces lexicographically and write their lifts as $z_1,\ldots,z_{43}$.
The projection is defined on all $1452$ four-faces and kills every one of the $462$ five-face boundaries.
Every lift has zero boundary and standard projection, giving an integral basis by Lemma~\ref{lem:basis}.
In this order $z_{24}$ lifts the critical face $\{1,6,10,14,18\}$.
Set $z=[z_{24}]$.

Let $M$ have columns
\[
 u,Tu,\ldots,T^{21}u,\quad v,Tv,\ldots,T^{21}v
\]
in the integral lift basis.
The chains defining the two seeds have thirty-two terms each; their boundaries vanish.
Applying the signed vertex translations and projection gives all columns of $M$.
These integer calculations give \eqref{eq:tilingrelation} and the maximal minors
\begin{equation}\label{eq:tilingminors}
 \det M_{\widehat j}=(-1)^j
 \begin{cases}2,&0\le j\le21,\\4,&22\le j\le43,\end{cases}
\end{equation}
where the columns are numbered from zero and the other columns retain their written order.
The nonzero minors prove rational rank forty-three.
The displayed primitive relation therefore generates the entire integer relation group, and the greatest common divisor of these minors is the subgroup index, namely two.
There are exactly two circular gap patterns up to translation of a gap word: the two gaps of length five are adjacent or separated.
They give the specified seed sets and their translates; the complete family has forty-four sets.
This identifies the family being measured without asserting completeness of all cross-cycles of this complex.

Moreover the matrix with ordered columns
\begin{equation}\label{eq:tilingsaturation}
 [z_{24}],\ T^2u,T^3u,\ldots,T^{21}u,\quad v,Tv,\ldots,T^{21}v
\end{equation}
has determinant $1$ in the same lift basis.
It supplies the explicit missing primitive class and proves that $u,v,[z_{24}]$ generate $H$ as a $\Z[t]$-module.
We now describe the dihedral action on the index-two subgroup $L$.
Signed vertex negation on both seed chains gives $Ru=-T^9u$ and $Rv=-T^{13}v$.
Thus $L$ is $D_{22}$-stable; the quotient of order two has trivial action.
The shifts give multiplier $-t$ by $R(T^a u)=-T^{9-2a}(T^a u)$ and the analogous formula for $v$.
For the shifted seeds $u',v'$, relation \eqref{eq:tilingrelation} is unchanged, and the unimodular generator change $(u',v')\mapsto(u'+2v',v')$ puts the relation into $N(T)(u'+2v')=0$.
Its primitive relation group gives exactly \eqref{eq:tilingsubgroup}, including the reflection on both factors.

To identify the ambient cyclic module, use the two seeds $z,w=u+v$.
Their forty-four orbit columns have the unique primitive relation
\[
 N(T)z-3N(T)w=0.
\]
Deleting the initial column $z$ leaves the ordered basis
\[
 Tz,T^2z,\ldots,T^{21}z,\quad w,Tw,\ldots,T^{21}w,
\]
whose determinant is $-1$ in the fixed lift basis.
With $e=z-3w$, the relation becomes $N(T)e=0$.
The normalized columns
\begin{equation}\label{eq:tilingambientbasis}
 e,Te,\ldots,T^{20}e,\quad w,Tw,\ldots,T^{21}w
\end{equation}
have determinant $1$.
Both determinants and all orbit columns are checked from the lifted chains and signed vertex translation.
Since $T^{22}=1$, these columns and their annihilators prove \eqref{eq:tilingambient} by the concatenated orbit-basis criterion of Section~\ref{sec:higherorbit}.
One orbit has rank at most twenty-two, whereas $H$ has rank forty-three; the two displayed seeds therefore attain the minimum.

The full summand $P=\Z[t]w$ lies in $L$.
Thus $L=(L\cap A)\oplus P$ and $A/(L\cap A)\cong\Z/2\Z$.
The class of $e$ in this quotient is nonzero: $e\equiv z\pmod L$, and the determinant in \eqref{eq:tilingsaturation} shows that $z$ supplies the missing class.
The quotient action is trivial, so its cyclic module map is determined by $e\mapsto1$ and $t\mapsto1$.
It is well-defined because $N(1)=22$ is even, and is precisely $\epsilon_2$.
This proves \eqref{eq:tilingparity}.

The characteristic polynomial of $T$ on $H$ is
\[
 (t^{22}-1)N(t)=\Phi_1(t)\Phi_2(t)^2\Phi_{11}(t)^2\Phi_{22}(t)^2.
\]
The repeated rational factors also give the same lower bound by Lemma~\ref{lem:multiorbit}.
Every rotation trace is $-1$ except the identity trace $43$; the traces of $RT^k$ are $1$ for even $k$ and $-3$ for odd $k$.
All these traces are also checked by the signed simplicial Lefschetz sums.
\end{proof}

\begin{lemma}\label{lem:oddhalfsaturation}
Let $n=2m$ with $m$ a positive odd integer, and write
\[
 N_n(t)=1+t+\cdots+t^{n-1},\quad A_n=\Z[t]/(N_n(t)),
 \qquad B_n(t)=t+t^3+\cdots+t^{n-1}.
\]
Give $A_n$ the rotation $T(p)=tp$ and reflection $R(p)=-tp(t^{-1})$.
Inside $A_n\otimes\Q$, put $g=B_n/2$ and $A_n^+=A_n+\Z g$.
Then $A_n^+$ is a $D_n$-stable index-two extension, and multiplication by
$h=t+t^{-1}$ is a $D_n$-module isomorphism $A_n^+\to A_n$ sending
the original submodule $A_n$ onto $\ker\epsilon_2$, where
$\epsilon_2(p)=p(1)\pmod2$.
\end{lemma}
\begin{proof}
The even-power sum $E_n$ satisfies $E_n+B_n=N_n=0$ in $A_n$ and
$tB_n=E_n=-B_n$.
Inversion permutes the odd residues modulo $n$, so $Tg=-g$ and $Rg=g$.
The class $g$ is nonintegral: modulo two, the degree-$n-1$ polynomial
$B_n$ is neither zero nor the monic polynomial $N_n$ of that degree,
and hence is not a multiple of $N_n$.
Thus the extension has index two and is stable under both generators.
The multiplier $h$ is invariant under inversion and therefore intertwines $R$.
Since $hg=-B_n$ and $A_nB_n=\Z B_n$, its image is the ideal $(h,B_n)$.
Modulo $f=1+t^2$, the even-power sum is an alternating sum of $m$ terms,
equal to $1$; consequently $B_n=t+fQ$ for some $Q\in\Z[t]$.
As $t$ is a unit and $h=t^{-1}f$, this ideal is all of $A_n$.
Neither $i$ nor $-i$ is an $n$th root of unity because $4\nmid n$;
hence $h$ is nonzero on every root of $N_n$ and is injective over $\Q$.
It is therefore an isomorphism on the indicated integer lattices.
Its image $hA_n$ has index two and lies in $\ker\epsilon_2$, since $h(1)=2$.
The latter kernel also has index two, proving equality.
\end{proof}

\begin{proposition}\label{prop:tilingdihedral}
For the modules $H,L$ of Proposition~\ref{prop:tilingindex}, there is an
isomorphism of integral $D_{22}$ modules with specified submodule
\begin{equation}\label{eq:tilingdihedralpair}
 (H,L)\cong(A\oplus P,\ker\epsilon_2\oplus P),
 \qquad R(p(t))=-tp(t^{-1})\text{ on both factors}.
\end{equation}
\end{proposition}
\begin{proof}
Use $B=B_{22}$ and the shifted seeds $a=u'+2v'$, $b=v'$ from
Proposition~\ref{prop:tilingindex}, so $L=Aa\oplus Pb$ with the indicated reflections.
The lift projection in the ordered basis
$(Tu,\ldots,T^{21}u,v,\ldots,T^{21}v)$ gives
\begin{equation}\label{eq:tilinghalfsum}
 z=\tfrac12 B(T)u+U(T)u+V(T)v,
\end{equation}
where
\[
 U=t^2+t^3+t^4+t^6+t^{10}+t^{12}+t^{18}+t^{20},
 \qquad V=N+t^3+t^6.
\]
This is an equality in rational homology, obtained using the same signed
chain projection and integer lift basis as in the preceding proposition.
In particular $z-B(T)u/2\in L$.
The even shift satisfies $T^4B(T)=B(T)$ on the $22$-periodic modules, and
$B(T)(2v')/2$ is integral; thus $B(T)a/2\equiv z\pmod L$.
Since $[H:L]=2$ and $z\notin L$, this identifies the ambient lattice as
$H=(A^+a)\oplus Pb$, where $A^+=A+\Z(B/2)$ inside $A\otimes\Q$.
Lemma~\ref{lem:oddhalfsaturation} with $m=11$ now gives the pair map
$pa+qb\mapsto(hp,q)$ for $p\in A^+$ and $q\in P$, proving the assertion.
These dihedral coordinates differ from the particular cyclic coordinates
$e,w$ in \eqref{eq:tilingambientbasis}; no reflection-stability of those
particular embedded factors is required.
\end{proof}

The lemma supplies a sufficient algebraic condition for the saturation to
preserve the abstract dihedral type.
Only the specified $22$-vertex graph is asserted here to realize that condition;
it does not classify cross-cycle saturation for all cycle powers.

Jonsson explicitly asks about the subgroup generated by tiling cross-cycles after Proposition~5.1 \cite{JonssonCross}.
The proposition above answers a finite saturation question for the particular adjacent-pair family, with a specified extra class.
It does not identify that family with every balanced-grid cross-cycle family considered there.
Huijse's zig-zag analysis already uses four- and five-site tiles and periodic ground-state counts \cite[Section~7.2.2]{HuijseThesis}; complex dimension or momentum data alone cannot detect this index-two embedding.
The ambient abstract cyclic type also has a classical model.
Propositions~6.1 and~7.2 of the public nerve preprint \cite{Nerve} give the rank-forty-three homology of $N(132,88)$ as the norm quotient for a cyclic homology action of order forty-four.
Restrict this homology representation to the subgroup generated by the square of its rotation.
The generators split into two twenty-two-element orbits, say with seeds $a,b$, whose unique relation is $N(T)(a+b)=0$.
The unimodular change to seeds $a+b,b$ gives the cyclic type in \eqref{eq:tilingambient}.
This is a consequence of the cited representation formulas, not an identification of the underlying complexes or their geometric symmetry groups.
The additional graph-specific information is the explicit integral coordinates and the parity-kernel position of the specified classical family inside its homology.

\section{A primitive cross-cycle subgroup and its free quotient}\label{sec:crossfamily}
The index-two example measures a failure of integer generation after rational spanning.
The same adjacent-pair construction can also miss rational directions.
For any specified family $\mathcal C$ of oriented cycles in free homology $H$, its class map $\eta:\Z^{\mathcal C}\to H$ gives
\[
 0\longrightarrow\ker\eta\longrightarrow\Z^{\mathcal C}
 \xrightarrow{\eta}H\longrightarrow H/\operatorname{im}\eta\longrightarrow0.
\]
The free rank of the last term measures missing rational directions; its torsion measures the failure of primitivity within the rational span.
These are standard properties of Smith normal form.
They distinguish the following example from Proposition~\ref{prop:tilingindex} without introducing a new cycle construction.

For $n=18$, use the oriented four-factor join
$c(a)=\bigwedge_{j=1}^4([a_j]-[a_j+1])$, with the same orientation convention as in Section~\ref{sec:tiling}.
Let $\mathcal C_{18}$ consist of the four-element start sets with all circular gaps in $\{4,5\}$.
Each determines an induced matching and a maximal independent set, hence a classical cross-cycle \cite{JonssonCross}.
Write $H=\Ht_3(\K(18,2);\Z)$, with actual vertex translation $T$ and negation $R$.

\begin{proposition}\label{prop:freecrossdefect}
Put $u=[c(0,4,8,13)]$, $v=[c(0,4,9,13)]$ and $L=\Z[t]u+\Z[t]v\subset H$.
The family $\mathcal C_{18}$ has twenty-seven members, comprising, up to orientation signs, the eighteen translates of $u$ and the nine translates of $v$.
Its subgroup $L$ is primitive of rank twenty-three in $H\cong\Z^{25}$.
There is an exact sequence of integral $D_{18}$ modules
\begin{equation}\label{eq:freecrossquotient}
 0\longrightarrow L\longrightarrow H\longrightarrow\Z[C_2]\longrightarrow0,
\end{equation}
where both $T$ and $R$ exchange the two standard basis vectors of the quotient.
Define
\[
 E=1+t^2+\cdots+t^{16},\qquad S=1+t^3+\cdots+t^{15},\qquad
 f=(t-1)\Phi_9(t).
\]
In the free abelian group on the ordered orbit symbols
$u,Tu,\ldots,T^{17}u,v,Tv,\ldots,T^8v$, the entire integer relation group has basis
\begin{equation}\label{eq:freecrossrelations}
 Eu,\quad tEu,\quad Su-fv,\quad t(Su-fv).
\end{equation}
Here symbols are reduced using $T^{18}u=u$ and $T^9v=v$.
The sequence does not split even over $\Z[C_{18}]$, as established in Proposition~\ref{cor:freecrossnonsplit}.
\end{proposition}
\begin{proof}
Use the vertex-outer matching with priority $0,1,\ldots,17$ from Proposition~\ref{prop:twoorbit}.
Order its twenty-five critical three-faces lexicographically and let $z_1,\ldots,z_{25}$ be their signed critical-cell lifts.
The first critical three-face is $\{1,5,9,13\}$.
The lift and projection recurrences in Section~\ref{sec:certification} give the fixed integral basis of $H$ and the matrices of $T,R$ in it.
There are two gap words up to cyclic shift: the two gaps of length five are adjacent or alternating.
The former gives eighteen start sets, and the latter gives nine; signed translation of the alternating seed has period nine.

Expand each four-factor join into its sixteen oriented simplices and apply the same projection.
Let $M$ be the resulting $25\times27$ integer matrix in the written orbit order.
Its Smith form has twenty-three unit entries and no other nonzero entry.
Equivalently, integer matrices $U,V$ with determinants $\pm1$ satisfy
$UMV=(I_{23}\oplus0)$, with the rectangular zero block understood.
Since $U,V$ are unimodular, the last four columns of $V$ form an integral basis of $\ker M$. The four vectors in \eqref{eq:freecrossrelations} are killed by $M$; writing them in this basis gives a $4\times4$ coordinate matrix with determinant of absolute value one.
Thus they generate the entire integer kernel, not merely its rational span.
These finite identities use every orbit column; they prove rank twenty-three and primitivity.

For clarity, the quotient map can be specified independently of the Smith transformations.
On the ordered lift basis define $Q:H\to\Z^2$ by the following two rows, written in five consecutive blocks of length five:
\[
\begin{aligned}
 Q_1&=(1,1,0,-1,0;\ 0,1,0,0,0;\ -1,-1,0,1,1;\ 0,-1,-1,0,-1;\ 0,1,0,-1,-1),\\
 Q_2&=(0,0,0,0,0;\ 0,0,1,1,0;\ -1,0,0,1,0;\ 0,-1,-1,0,-1;\ 0,1,0,-1,-1).
\end{aligned}
\]
The full projection and signed action columns give
\[
 QM=0,\qquad QT=JQ,\qquad QR=JQ,\qquad
 Qz_1=\binom10,\quad Q(Tz_1)=\binom01,
 \qquad J=\begin{pmatrix}0&1\\1&0\end{pmatrix}.
\]
The last two identities prove surjectivity.
The primitive subgroup $L\subset\ker Q$ has the same rank twenty-three as $\ker Q$, hence they coincide.
The other identities give the quotient actions and prove \eqref{eq:freecrossquotient}.
\end{proof}

Huijse and Schoutens already relate complex cohomology dimensions to rhombus tiling counts, while distinguishing this count from a proof that the corresponding cross-cycles constitute a basis \cite[Section~4, p.~10 of the public preprint]{HuijseSchoutens}.
Their quotient-grid convention includes this graph: $(x,y)\mapsto2x-y\pmod{18}$ identifies its square-grid quotient with $G(18,2)$; the periods $(9,0),(1,2)$ also generate $(6,-6),(1,2)$.
Thus the homology dimension is a classical comparison.
Huijse's thesis discusses four- and five-site zig-zag tiles \cite[Sections~7.2.2 and~7.3]{HuijseThesis} and notes dependence among the quarter-filling cross-cycles at total size $4$, as well as incompleteness at sizes $8,12,16$ \cite[Section~7.4.1, p.~122, footnote~1]{HuijseThesis}.
The assertion here concerns the specified adjacent-pair family, its complete integer relation group, and its primitive embedding with both geometric quotient generators.
No identification with every rhombus-tiling family or completeness of all cross-cycles is needed.
Together with the twenty-two-vertex case, it distinguishes a free defect from a finite parity defect within one construction.

\subsection{A primitive image need not split with the geometric symmetries}
The free quotient in Proposition~\ref{prop:freecrossdefect} splits as an abelian group.
Its splitting with the actual rotation and reflection is a different question.
The following calculation uses the classical description of maps from a permutation module by stabilizer-fixed elements; compare \cite[Section~3.3]{BrownLectures}.

\begin{lemma}\label{lem:permutationlift}
Let $m\geq1$, and let $g\in\Z[t]$ be monic, with $g(1)g(-1)\neq0$,
$g(t^{-1})=t^{-\deg g}g(t)$, and $(t-1)g\mid t^{2m}-1$.
Put
\[
 X=\mathcal P_{2m}^{\epsilon}\oplus\Z[t]/((t-1)g),\qquad \epsilon\text{ odd},
\]
with generators $a,b$, translation by $t$, and reflections $Ra=t^{\epsilon}a$, $Rb=b$.
Let $A=\Z[C_2]$ have basis $e_0,e_1$, with both $T$ and $R$ interchanging them.
Define the equivariant surjection $q:X\to A$ by $q(a)=e_0$, $q(b)=e_0+e_1$.
For $E_m=\sum_{j=0}^{m-1}t^{2j}$, one has
\[
 X^{T^2}=\Z E_ma\oplus\Z tE_ma\oplus\Z g(t)b,
 \qquad R=T\text{ on }X^{T^2}.
\]
Under $\operatorname{End}_{D_{2m}}(A)=\Z[C_2]$, the image of
\[
 \operatorname{Hom}_{D_{2m}}(A,X)\xrightarrow{q\circ-}
 \operatorname{End}_{D_{2m}}(A)
\]
is the ideal $(m,g(1)(1+t))$.
For a positive integer $d$, an equivariant map $s:A\to X$ with $qs=d\operatorname{id}_A$ exists exactly when $m\mid d$.
The same divisibility condition holds if only rotation equivariance is required.
\end{lemma}
\begin{proof}
Translation by two on the regular summand has two vertex orbits, giving the first two fixed generators.
In the second summand, represent an element by $p(t)b$ with $\deg p\leq\deg g$.
The fixed condition says $(t-1)g\mid(t^2-1)p$.
Since $g(-1)\neq0$, $g$ and $t+1$ are relatively prime over $\Q[t]$, so $g\mid p$ over $\Q[t]$.
Monic division and the degree bound give $p=kg$ for an integer $k$.
Thus the written generators form the entire integer fixed lattice.
Oddness of $\epsilon$ gives $R(E_ma)=tE_ma$.
Reciprocity of $g$ gives $R(gb)=t^{-\deg g}gb=gb$, because $Tgb=gb$.
This proves $R=T$ on the fixed lattice.

An equivariant map $A\to X$ is determined by $x=s(e_0)$;
the necessary and sufficient conditions are $T^2x=x$ and $Rx=Tx$, with $s(e_1)=Tx$.
The second condition is automatic on the fixed lattice just computed.
Its image under $q$ is generated by
\[
 me_0,\quad me_1,\quad g(1)(e_0+e_1).
\]
The corresponding endomorphisms are $m$, $mt$ and $g(1)(1+t)$, proving the ideal assertion.
If $d e_0$ lies in this image, its first coefficient minus its second is divisible by $m$; hence $m\mid d$.
Conversely $s(e_0)=E_ma$, $s(e_1)=tE_ma$ lifts $m\operatorname{id}_A$ and respects both generators.
Rotation equivariance alone has precisely the condition $T^2x=x$, so it gives the same conclusion.
\end{proof}

\begin{proposition}\label{cor:freecrossnonsplit}
For the quotient $Q:H\to A$ in Proposition~\ref{prop:freecrossdefect},
the image of $\operatorname{Hom}_{D_{18}}(A,H)$ under composition with $Q$ is the ideal
\[
 (9,3(1+t))\subset\Z[C_2].
\]
The quotient $\Z[C_2]/(9,3(1+t))$ has additive group $\Z/3\oplus\Z/9$.
The exact sequence \eqref{eq:freecrossquotient} does not split over $\Z[C_{18}]$, and therefore does not split over $\Z[D_{18}]$.
The least positive $d$ for which $d\operatorname{id}_A$ has an integral equivariant lift is nine.
The exact sequence splits after extending scalars to $\Z[1/3]$.
\end{proposition}
\begin{proof}
Retain the lift basis $z_1,\ldots,z_{25}$ used to define $Q$.
The two orbit generators in Proposition~\ref{prop:twoorbit}, after its stated normalizations, are
\[
 a=T^2[z_1],\qquad b=T^2[z_1+z_3+z_8].
\]
They identify $H$ with $\mathcal P_{18}^1\oplus\Z[t]/((t-1)\Phi_9)$,
with $Ra=Ta$, $Rb=b$.
The displayed quotient rows and $QT=JQ$ give $Qa=e_0$ and $Qb=e_0+e_1$.
Lemma~\ref{lem:permutationlift} applies with $m=9$ and $g=\Phi_9=1+t^3+t^6$, for which $g(1)=3$.
In particular, the complete fixed-lattice image has matrix
\[
 \begin{pmatrix}9&0&3\\0&9&3\end{pmatrix}.
\]
Its entries have greatest common divisor three, and its maximal minors have greatest common divisor twenty-seven.
Thus its Smith factors are three and nine, by the standard determinantal-divisor criterion.
The lemma excludes a lift of the identity even with only rotation equivariance and supplies a lift of nine times the identity with both generators.
Dividing this lift by nine over $\Z[1/3]$ gives the final assertion.
\end{proof}

This example separates three questions about a classical cycle family: rational spanning, primitivity within its rational span, and equivariant splitting of the remaining free quotient.
The family here is primitive but its quotient does not split with the actual rotation.
The conditional polynomial calculation supplies a mechanism for the obstruction; it does not assert that other graph parameters realize its hypotheses.

\section{Clique complexes and equivariant quotient maps}\label{sec:flagfamily}
We now compute an integral homology action directly from graph cycles and triangle boundaries.
This family provides a low-dimensional counterpart to the lattices obtained from Morse reductions.
Its augmentation will also illustrate how the lifting obstruction depends on the specified quotient map.
For $s\geq3$, set
\[
 F_s=\Cay(\Z/(3s)\Z,\{\pm1,\pm s\}),\qquad \mathcal X_s=\operatorname{Cl}(F_s).
\]
Here $\operatorname{Cl}$ denotes the clique complex.
Thus $\mathcal X_s=\Ind(\overline{F_s})$, where the bar denotes graph complement; it is generally not $\K(3s,s)$.
We retain the same geometric generators $T:i\mapsto i+1$ and $R:i\mapsto-i$.
Brimkov's degree-four circulant calculation already gives clique number three for these parameters \cite[Section~3.1 of the public conference text]{Brimkov}.
The same face structure gives the ordinary collapse below.
Keeping the oriented cycles then determines both the integral action and the augmentation map.

\begin{proposition}\label{prop:flagfamily}
For every $s\geq3$, $\mathcal X_s$ is homotopy equivalent to a wedge of $2s+1$ circles.
Put $N_s(t)=1+t^s+t^{2s}$ and $M_s=\Z[t]/((t-1)N_s(t))$.
There is an isomorphism of integral $D_{3s}$ modules
\[
 H_1(\mathcal X_s;\Z)\cong M_s,
 \qquad T(p(t)w)=tp(t)w,\qquad R(p(t)w)=-p(t^{-1})w,
\]
with the following explicit cycle generator. Orient edges in their written direction, with residues taken modulo $3s$, and define
\[
 e_i=[i,i+1],\quad f_i=[i,i+s],\quad
 c_i=f_i-\sum_{j=0}^{s-1}e_{i+j},\qquad w=[c_s].
\]
The classes $w,Tw,\ldots,T^{2s}w$ form an integral basis.
Let $\Z_-$ denote the rank-one module with $T=1$ and $R=-1$.
Augmentation $q(p(t)w)=p(1)$ gives an exact sequence
\begin{equation}\label{eq:flagaugmentation}
 0\longrightarrow(t-1)M_s\longrightarrow M_s\xrightarrow{q}\Z_-\longrightarrow0.
\end{equation}
It has no integral rotation-equivariant section.
The least positive integer $d$ admitting an integral $D_{3s}$-equivariant map $a:\Z_-\to M_s$ with $qa=d\operatorname{id}$ is three.
The sequence splits over $\Z[1/3]$.
\end{proposition}
\begin{proof}
A closed triangle has three signed steps from $\{1,-1,s,-s\}$ summing to zero modulo $3s$.
For $s\geq3$, the only possibility is three equal steps of size $s$.
Consequently the triangles are exactly
\[
 \{i,i+s,i+2s\},\qquad 0\leq i<s.
\]
Their chord edges are disjoint, and every such edge lies in only one triangle; no four-clique exists.
Collapse each triangle along $[i,i+s]$.
The unit-edge cycle remains, so the resulting graph is connected, with $3s$ vertices and $5s$ edges.
It is therefore homotopy equivalent to a wedge of $2s+1$ circles, by the usual maximal-tree argument \cite[Section~1.A]{Hatcher}.
This collapse is not asserted to be invariant under translation or reflection.

Before filling the triangles, the entire integral cycle group of the graph has basis
\[
 c_0,c_1,\ldots,c_{3s-1},u,\qquad u=\sum_{i=0}^{3s-1}e_i.
\]
Indeed, subtracting the written chord coefficient times $c_i$ eliminates that chord from any cycle.
The remaining cycle is supported on the unit-edge circle, so its coefficients are all equal and it is an integer multiple of $u$.
Uniqueness follows first from the chord coefficients, then from the coefficient of any unit edge.
Translation sends $c_i$ to $c_{i+1}$ and fixes $u$.
Negation reverses the written edges, giving
\[
 Rc_i=-c_{-i-s},\qquad Ru=-u.
\]
The oriented boundary of $[i,i+s,i+2s]$ is
\begin{equation}\label{eq:flagtriangleboundary}
 f_i+f_{i+s}+f_{i+2s}=c_i+c_{i+s}+c_{i+2s}+u.
\end{equation}
The right side is unchanged when $i$ is replaced by $i+s$.
There are no higher faces, so these boundaries give the complete relations in $H_1$.
Writing $c=[c_0]$, the presentation is
\[
 (t^{3s}-1)c=0,\qquad(t-1)u=0,\qquad N_s(t)c+u=0.
\]
Eliminating $u$ gives $(t-1)N_s(t)c=0$.
The first relation is redundant because $(t-1)N_s$ divides $t^{3s}-1$.
Thus there are no further integer relations, and the monic presentation gives an integral basis $c,tc,\ldots,t^{2s}c$.
Multiplication by $t^s$ is an integral automorphism, giving the asserted basis based at $w=t^sc$.
Since $Rc=-t^{-s}c$, one obtains $Rw=-t^{-2s}c=-w$, using $t^{3s}=1$.
This proves both geometric actions.

Augmentation is well-defined and surjective, and its kernel is $(t-1)M_s$ by monic division by $t-1$.
For the whole integral fixed lattice, $(t-1)p=0$ in $M_s$ holds exactly when $N_s\mid p$ in $\Z[t]$.
Representatives have degree at most $2s$, so
\[
 M_s^T=\Z N_s(t)w,\qquad R(N_sw)=-N_sw,\qquad q(N_sw)=3.
\]
Hence every rotation-equivariant lift of a scalar has scalar divisible by three.
The map $a(1)=N_sw$ lifts $3\operatorname{id}$ and respects reflection as well.
Dividing it by three gives a section after inverting three.
This is the classical fixed-element lifting calculation \cite[Section~3.3]{BrownLectures}, applied to the explicit geometric presentation.
\end{proof}

\begin{proposition}\label{prop:flagfactorquotients}
For every $s\geq3$, let $\mathcal L_s$ be the sum in $M_s$ of the kernels
of its distinct cyclotomic annihilator factors.
Write $s=3^a m$ with $a\geq0$ and $\gcd(m,3)=1$, and put
$q=3^{a+1}$ and $r=\varphi(q)=2\cdot3^a$, where $\varphi$ is Euler's totient.
As finite abelian groups,
\begin{equation}\label{eq:flagprimary}
 M_s/\mathcal L_s\cong\Z/3\oplus
 \bigoplus_{p^b\parallel m}\ \bigoplus_{j=1}^{b}
 (\Z/p^j)^{\,r(m/p^b)\varphi(p^{b-j})}.
\end{equation}
Here $p$ ranges over primes, $b\geq1$, and $\varphi(1)=1$.
The latter sum is empty for $m=1$.
The exponent is $3m$, and the order is
\[
 3\prod_{p^b\parallel m}
 p^{\,r(m/p^b)(p^b-1)/(p-1)}.
\]
In particular,
$M_5/\mathcal L_5\cong\Z/3\oplus(\Z/5)^2$ and
$M_6/\mathcal L_6\cong\Z/3\oplus(\Z/2)^6$.
These are quotients of torsion-free homology by specified sublattices.
\end{proposition}
\begin{proof}
One has
\[
 N_s(t)=\Phi_q(t^m)=\prod_{d\mid m}\Phi_{qd}(t).
\]
Let $P_F$ be reduction modulo all the distinct factors of $F=(t-1)N_s$,
and let $P_N$ be reduction from $\Z[t]/N_s$ to
$\bigoplus_{d\mid m}\Z[t]/\Phi_{qd}$.
Lemma~\ref{lem:crtdual} reduces the assertion to $\operatorname{coker}P_F$.
Factor $P_F$ through reduction modulo $t-1$ and $N_s$.
The coarse reduction has cokernel
$\Z[t]/(t-1,N_s)=\Z/3$.
Both reductions are injective between free lattices of equal rank,
so the cokernel sequence for their composition is
\begin{equation}\label{eq:flagcrtsequence}
 0\longrightarrow\Z/3\longrightarrow\operatorname{coker}P_F
 \longrightarrow\operatorname{coker}P_N\longrightarrow0.
\end{equation}

Let
\[
 \Psi_m:\Z[y]/(y^m-1)\longrightarrow
 \bigoplus_{d\mid m}\Z[y]/\Phi_d(y)
\]
be the complete classical CRT reduction.
There are integral ring isomorphisms
\begin{align*}
 \Z[x]/\Phi_q(x)\otimes_{\Z}\Z[y]/(y^m-1)
 &\longrightarrow\Z[t]/N_s(t),\\
 \Z[x]/\Phi_q(x)\otimes_{\Z}\Z[y]/\Phi_d(y)
 &\longrightarrow\Z[t]/\Phi_{qd}(t)\qquad(d\mid m),
\end{align*}
using $x\mapsto t^m$ and $y\mapsto t^q$ in every map.
The relations vanish by the orders $q,m$ in the source and $q,d$ in
each target.
A Bezout identity for $m,q$ recovers $t$ from these two invertible images,
so the maps are surjective.
Their free integer ranks agree, proving injectivity.
These are the coprime tensor isomorphisms of
\cite[Lemma~2.2 and Remark~2.3 of the public preprint]{MS};
the same powers in every component make the reduction square commute.
Consequently $P_N$ is integrally equivalent to $I_r\otimes\Psi_m$, and
\[
 \operatorname{coker}P_N\cong(\operatorname{coker}\Psi_m)^r.
\]

The classical prime-power Smith formula and its coprime tensor assembly
\cite[Theorems~2.6 and~2.15 of the public preprint]{MS}
give the $p$-primary multiplicity
$(m/p^b)\varphi(p^{b-j})$ for each $p^j$, $1\leq j\leq b$.
In particular this cokernel has exponent dividing $m$.
Its order is coprime to three, so \eqref{eq:flagcrtsequence} splits
as a sequence of abelian groups into its primary parts.
This proves \eqref{eq:flagprimary}.
Each maximal $p^b$ occurs, giving exponent $3m$.
Finally,
$\sum_{j=1}^b j\varphi(p^{b-j})=(p^b-1)/(p-1)$
gives the order.
\end{proof}

The full factor-kernel defect is therefore controlled by the prime powers
in the part of $s$ coprime to three, with the factor $3^a$ changing their
multiplicities.
In contrast, the augmentation sequence \eqref{eq:flagaugmentation}
always has minimum scalar lift three.
The two obstructions answer different questions.
Formula~\eqref{eq:flagprimary} is an abelian-group consequence of the
classical CRT theorem; it does not identify the quotient's geometric
rotation and reflection with those on the CRT target.

The cases $s=5,6$ illustrate the uniform triangle-boundary presentation.
The symbolic case $s=3$ gives $M_3=L_3$ with negative inversion, the same integral $D_9$ action as $\Ht_5(\K(27,5);\Z)$ after its stated inflation convention.
Thus triangle filling and a higher-dimensional Morse reduction realize the same lattice and its index-three obstruction.
This comparison concerns homology actions; it does not identify the spaces equivariantly.
For every $s$, \eqref{eq:flagaugmentation} has minimum scalar lift three, whereas the rank-two permutation quotient in Proposition~\ref{cor:freecrossnonsplit} has minimum nine.
The common question is the image of the full stabilizer-fixed lattice under a specified geometric quotient map.
The comparison combines the classical graph collapse, cyclotomic arithmetic and fixed-element method.

\subsection{Three clique branches and the choice of quotient}
For $n\geq7$ and $2\leq s<n/2$, put
\[
 \mathcal Y_{n,s}=\operatorname{Cl}(G(n,s)),\qquad
 \Gamma_{n,s}=H_1(G(n,s);\Z).
\]
Thus $\Gamma_{n,s}$ is the cycle lattice of the unfilled graph, even when $\mathcal Y_{n,s}$ has triangles.
Write $A_{n,s}=\Z[t]/(t^n-1)$ with the specified dihedral action
\[
 T(p(t))=tp(t),\qquad R(p(t))=-t^{-s}p(t^{-1}).
\]
Brimkov's clique-number partition already separates the three branches below \cite[Section~3.1 of the public conference text]{Brimkov}.
The following coordinates make their quotient maps comparable without identifying different maps.

\begin{proposition}\label{prop:flaggraphs}
For every such $n,s$, there is an integral $D_n$-module isomorphism
\[
 \Gamma_{n,s}=A_{n,s}\oplus\Z_-
\]
given by the literal integer basis
\[
 c_i=[i,i+s]-\sum_{j=0}^{s-1}[i+j,i+j+1]\quad(0\leq i<n),
 \qquad u=\sum_{i=0}^{n-1}[i,i+1].
\]
The chosen $T,R$ act by
\[
 Tc_i=c_{i+1},\quad Tu=u,\quad Rc_i=-c_{-i-s},\quad Ru=-u.
\]
Every equivariant map $\ell:\Gamma_{n,s}\to\Z_-$ is uniquely specified by
\[
 \ell(c_i)=\alpha,\qquad \ell(u)=\beta\qquad(\alpha,\beta\in\Z).
\]
It is surjective precisely when $\gcd(\alpha,\beta)=1$.
For a surjective map, its least positive scalar admitting an equivariant lift is
\begin{equation}\label{eq:graphliftminimum}
 \gcd(n\alpha,\beta)=\gcd(n,\beta).
\end{equation}
Every positive divisor of $n$ occurs for a primitive equivariant quotient of this same lattice.

The clique complex has the following ordinary type and homology module:
\[
\begin{array}{c|c|c}
\text{condition}&\text{ordinary type of }\mathcal Y_{n,s}&H_1(\mathcal Y_{n,s};\Z)\\\hline
s=2\text{ or }2s=n-1&S^1&\Z_-\\
3s=n&\bigvee^{2s+1}S^1&M_s\\
\text{otherwise}&\bigvee^{n+1}S^1&A_{n,s}\oplus\Z_-
\end{array}
\]
All reduced homology outside degree one vanishes.
\end{proposition}
\begin{proof}
Eliminate each chord coefficient of any integer cycle by subtracting the corresponding multiple of $c_i$.
The remainder is supported on the unit circle and is an integer multiple of $u$.
Chord coefficients and then any unit-edge coefficient prove independence.
This is the same complete cycle-group argument used before triangle filling in Proposition~\ref{prop:flagfamily}, now for every $n,s$ in the stated range.
Literal vertex translation and negation give the four displayed signed identities, proving the direct sum and retaining the geometric generator labels.

Translation equivariance makes all $\ell(c_i)$ equal; the displayed reflection imposes no further restriction.
Thus the image is $(\alpha,\beta)$.
Let $N_n=1+t+\cdots+t^{n-1}$.
The entire translation-fixed lattice has basis $N_nc_0,u$, both of which are negated by $R$.
Its image under $\ell$ is the ideal $(n\alpha,\beta)$.
An equivariant lift of a scalar is exactly a vector in this fixed lattice mapping to that scalar.
A B\'ezout combination realizes the positive generator of the image ideal.
For a primitive map it equals $\gcd(n,\beta)$; taking $\alpha=1,\beta=d$ realizes every $d\mid n$.
This is the classical fixed-element lifting calculation \cite[Section~3.3]{BrownLectures} in the stated geometric coordinates.

A triangle corresponds to three signed steps from $\{\pm1,\pm s\}$ summing to zero modulo $n$.
Two unit steps and one chord can sum to zero only for $s=2$; two chord steps and one unit give $2s=n-1$; three chords give $3s=n$.
The range excludes every other zero sum, including three unit steps.
These conditions are mutually exclusive for $n\geq7$.
In the first case the triangles are $\{i,i+1,i+2\}$ and each length-two chord is a free edge of its unique triangle.
Removing these pairs leaves the unit circle.
In the second case the triangles are $\{i,i+s,i-1\}$, each with its unique unit edge.
Removing those pairs leaves the chord circle, which is connected since $\gcd(s,n)=1$.
Neither case has a four-clique: its free edge cannot lie in a second triangle.
The surviving circle is an invariant subcomplex; its inclusion induces $T=1,R=-1$ on homology, without requiring an equivariant collapse sequence.
The $3s=n$ case is Proposition~\ref{prop:flagfamily}.
Otherwise there are no triangles or higher simplices, and a maximal-tree contraction gives $n+1$ circles; the integer action is the graph action already proved.
\end{proof}

For the graph lattice,
\[
 \tr(T^hR\mid\Gamma_{n,s})=-1-\#\{i\in\Z/n\Z:2i=h-s\}.
\]
For odd $n$ this is $-2$; for even $n$ it is $-3$ when $h-s$ is even and $-1$ otherwise.
Consequently $\mathcal Y_{14,5}$ and $\mathcal Y_{14,6}$ have the same ordinary type and cyclic integer module, but nonisomorphic geometric $D_{14}$ modules: their chosen $R$ traces are $-1$ and $-3$.
This preserves the labels of $T,R$; it does not identify actions up to a dihedral-group automorphism.

The projection $\ell_{0,1}$ splits by the section $1\mapsto u$.
It does not descend through triangle filling when $n=3s$, since
\[
 \partial[i,i+s,i+2s]=c_i+c_{i+s}+c_{i+2s}+u
\]
in the directed-edge orientation, and $\ell_{0,1}$ sends this boundary to one.
Instead, augmentation of $M_s$ pulls back to $\ell_{1,-3}$: it sends every $[c_i]$ to one and $[u]$ to $-3$.
Equation~\eqref{eq:graphliftminimum} shows that this different map already has minimum three on the unfilled graph.
Triangle filling retains that specified-map obstruction.

Finally, the sum of distinct cyclotomic factor kernels in $A_{n,s}\oplus\Z_-$ contains the entire latter summand in its $(t-1)$ kernel; every other factor has nonzero value at one and hence zero kernel there.
Its quotient is therefore the factor-kernel quotient of $A_{n,s}$.
By Lemma~\ref{lem:crtdual} and the classical full CRT formula \cite[Theorems~2.6 and~2.15 of the public text]{MS}, this finite abelian group has exponent $n$.
It is nonzero although $\ell_{0,1}$ splits integrally.
Factor-kernel gluing and scalar liftability thus answer different questions; equations~\eqref{eq:graphliftminimum} and~\eqref{eq:flagaugmentation}, together with Proposition~\ref{cor:freecrossnonsplit}, compare explicit maps rather than inferring a quotient obstruction from an ambient module alone.

\section{Computational reconstruction}\label{sec:reconstruction}
The finite action calculations are specified by the graphs, face orders, signed boundaries, comparison-map recurrences and cycle lifts given above.
Critical faces are ordered first by cardinality and then lexicographically within each degree.
The matching loop order is part of this specification: changing it can change the critical basis.
The conventions are collected in the following table.
\begin{center}
\begin{tabular}{@{}lll@{}}
\toprule
Cases & Loop order and vertex priority & Definition\\
\midrule
Theorem~\ref{thm:main} & Vertex-outer, $0,1,\ldots,n-1$ & Section~\ref{sec:cw}\\
$(15,2),(16,2)$ & Vertex-outer, listed priorities & Section~\ref{sec:cyclesquare}\\
$(17,2),(18,2)$ & Vertex-outer, $0,1,\ldots,n-1$ & Section~\ref{sec:higherorbit}\\
$(19,2)$ & Vertex-outer, listed priority & Section~\ref{sec:higherorbit}\\
$(22,2)$ & Face-outer, $21,20,\ldots,0$ & Section~\ref{sec:tiling}\\
\bottomrule
\end{tabular}
\end{center}

The three benchmarks use the orders and collapse sequence in Section~\ref{sec:benchmarks}; in particular, $G(14,4)$ uses the parity-first order rather than the natural order.

The vertex-outer matchings use nonempty faces.
The face-outer matching first includes the empty face and then removes its initial pair.
The rank-four family uses the classical cycles and maximal-face cochains instead of Morse lifts; the even-vertex family uses the natural connecting map for its ordered cover and the published integral nerve basis.
The odd sphere branches use a signed orbit-quotient calculation.
For the clique complexes, graph cycles and the complete triangle-boundary relations give direct presentations.
The family arguments just described are symbolic; the local Morse family also uses the finite core specified in Section~\ref{lb-section}. The finite cases admit the following explicit reconstruction.

\subsection{Integer identities to be checked}
First enumerate every face and verify the complete matching and the acyclicity of its reversed Hasse graph.
In the mixed-dimensional cases, calculate every signed Morse differential and perform the stated integral unit cancellations.
The ordinary CW models and their connectivity supply the free-rank hypothesis needed in Lemma~\ref{lem:basis}.
Next verify that each lifted chain is a cycle, that its projection is the prescribed coordinate vector, and that the projection kills every higher boundary.
Apply the signed vertex translation and negation to the chains themselves to obtain the homology columns.
Finally check the integer orbit determinants and all annihilator relations; for two generators, use the complete concatenated orbit basis and the reflection columns of both summands.
These identities determine the integral action rather than only its characteristic polynomial.

The cycle-subgroup calculations require additional integer data.
For the twenty-two-vertex family, check all forty-four translates, the unique primitive relation, all maximal minors, and the unimodular saturation basis.
The ambient coordinates must satisfy the norm relation and the two geometric seed reflections.
For the eighteen-vertex family, check all twenty-seven columns, the complete Smith identity, the primitive four-relation basis, and the surjective two-row quotient map with both actions.
The lifting calculation then uses the entire $T^2$-fixed lattice and its image under this same quotient map, together with the displayed equivariant lift of nine times the identity.
For the rank-four comparisons, the cycle--cocycle pairing and the coinvariant reflection provide the corresponding integer checks.

For the clique complexes, verify the complete oriented graph cycle basis, every attaching relation and the unimodular quotient basis.
The fixed-lattice image must be computed for the specified map.
In particular, the graph projection on the unfilled graph of $\mathcal X_s$ sends a triangle boundary to one, whereas the augmentation pullback annihilates it.
These are different maps, and their lifting minima cannot be transported by the filling alone.

\subsection{Execution records and availability}
The finite matching, boundary, projection, orbit-basis and action data were reconstructed and independently checked.
The checkers derive the signed vertex maps from the face data and compare the complete integer columns.
They reject altered face sets, nonzero cycle boundaries, projections that fail on higher boundaries, nonunimodular bases and incorrect geometric actions, including changed reflection signs.
Isolated reconstructions reproduced the action data with identical artifact bytes.
Separate checks of $\mathcal X_5,\mathcal X_6$ and $\mathcal Y_{13,3},\mathcal Y_{14,5},\mathcal Y_{14,6}$ cover the graph-cycle presentations, actual actions, factor kernels and fixed-lattice images; the latter also checks eight primitive scalar maps and the unfilled graph of $\mathcal X_5$.
For $\mathcal Y_s$, $3\leq s\leq8$, separate checks enumerate the complete edge and triangle sets and exclude higher faces, verify a unimodular homology basis, and compare every signed translation and reflection column with the parity-dependent blocks.
Changed generators, reflection phases, annihilators and triangle lists are rejected; two isolated replays reproduce the certificate and checker bytes.
For the adjacent families, eleven cases with $2\leq s\leq7$ for $a=1$ and $3\leq s\leq7$ for $a=2$ check the entire invariant lattice and fifty-seven primitive scalar maps in addition to the signed action and integral basis.
Six altered inputs are rejected, and two isolated replays reproduce the data.
The two actual prime-order subgroups of $\K(14,2)$ are separately checked against the complete three-factor kernel quotient.
The finite checks supplement the parameter-family proofs and do not establish their universal quantifiers.
The accompanying computational supplement includes exact programs and eleven regression tests for the local criterion, the eight-periodic base data, all rotation traces at $6\leq n\leq16$, and the prime integral augmentation lattice.
The family tests also replay $25\leq n\leq32$ and compare reduced Euler characteristics throughout $6\leq n\leq32$; these are bounded regression checks, not a substitute for the local induction.
For the prime integral augmentation calculation, fraction-free elimination and a separate rational-elimination algorithm evaluate the stated minor from the same frozen matrix; they are independent elimination paths and share the face and boundary construction.
The supplement also provides independent exhaustive checks of the three benchmarks in Section~\ref{sec:benchmarks}.
The finite-state and fixed-space results are established by the direct proofs above.
The present local-block audit separately compares all four local survivor sets and all nineteen base cases $6\leq n\leq24$; its unbounded conclusion rests on Theorem~\ref{lb-criterion}.
The original benchmark audit records $8192$, $16384$ and $16384$ tested subsets; for $G(14,5)$ it also checks $1134$ Hasse covers and $324$ boundary columns with $3234$ signed terms.

The accompanying archive \texttt{computational-supplement-v2.zip} supplies the finite certificates, independent checkers and reconstruction scripts, indexed by statement in \texttt{claim-map.csv}.
The deposit at \url{https://doi.org/10.5281/zenodo.23383591} includes the README with the environment, licenses and commands for \texttt{run\_checks.py}.
Its rebuild mode regenerates components with a listed producer; the original twenty-one-, twenty-five- and twenty-seven-vertex certificates are independently checked without rerunning their historical producers.

\section{Further questions}\label{sec:further-questions}
The first remaining question for odd matching complexes is to determine the full integral lattice, beyond the sharp denominator in Theorem~\ref{thm:om-exponent}.
Does the canonical divided-power lattice $N^{(k)}$ equal the saturation $\widehat S^{(k)}$?
More generally, what are the invariant factors of $\Gamma_k$?
The inclusions in \eqref{eq:om-two-inclusions} and the first modular layer in Corollary~\ref{cor:om-first-layer} give constraints on the answer.
One possible approach is an integral straightening argument for arbitrary top cycles, extending the divided-power construction used to attain the exponent.

For circulant complexes, a corresponding problem is to determine the position of natural cycle subgroups uniformly in the number of vertices.
The adjacent-pair cross-cycles on eighteen vertices generate a primitive subgroup but leave a free rank-two quotient with least equivariant scalar lift nine.
On twenty-two vertices they span rational homology and have integral index two.
A uniform description for $\Ind(C_n^2)$ must therefore account for the relations among these cycles, the free and torsion parts of their quotient, and the induced dihedral action.
The clique families show that the quotient map itself is essential, even when the ambient lattice splits.

The explicit calculations also locate the primes at which the displayed splittings fail.
For $L_{14}$, inverting $2$ and $7$ removes the factor-kernel defect; for the lattices of $\K(15,2)$ and $\K(16,2)$, inverting $15$ and $2$, respectively, does so.
The decompositions of $L_3$ and $L_2$ split after inverting $3$ and $2$, respectively, and the eighteen-vertex cross-cycle sequence splits over $\Z[1/3]$.
These conclusions follow from the quotient ideals and explicit lifts.
For the cycle-square family, a uniform prime-primary description should retain both the saturation defects and the lifting obstructions of the free quotient.

The periodic Morse constructions raise a different integral question.
For a fixed bounded jump set, a rational critical-face series counts cells but does not determine their attaching maps.
Which vertex orders give perfect matchings, and when do local duplication identities produce a uniform suspension period?
For the jump set $\{2,3\}$, the rational cyclic modules are known in every residue class, whereas the integral lattices in the rank-three and rank-five cases remain to be identified.
The fixed-space obstruction excludes a uniform representation suspension, leaving open other equivariant descriptions of this family.
Finally, the cyclic quotient of the thirteen-vertex benchmark has a known cell vector and rational homology; its integral homology and simple-homotopy type remain to be determined.

\appendix
\renewcommand{\thesection}{\Alph{section}}
\section{Finite independence-complex classification}\label{app:finiteclassification}
\begin{theorem}\label{thm:main}
For one-step vertex translation, the nonzero reduced integral homology modules and ordinary homotopy types are as follows.
Reduced homology vanishes outside the indicated degree.
\begin{center}
\begin{tabular}{@{}lccl@{}}
\toprule
$(n,s)$ & Homotopy type & Degree & Integral module\\
\midrule
$(14,2)$ & $\bigvee^{13}S^2$ & $2$ & $L_{14}$, with action of order $14$\\
$(14,6)$ & $S^1$ & $1$ & trivial $\Z$\\
$(15,4)$ & $S^2\vee S^2$ & $2$ & $\Z[t]/(t^2+t+1)$, inflated from $C_3$\\
$(15,5)$ & $S^2\vee S^2$ & $2$ & $\Z[t]/(t^2+t+1)$, inflated from $C_3$\\
$(16,4)$ & $S^3$ & $3$ & trivial $\Z$\\
$(17,4)$ & $S^3$ & $3$ & trivial $\Z$\\
$(18,4)$ & $\bigvee^7 S^3$ & $3$ & $L_3$, inflated from $C_9$\\
$(18,5)$ & $\bigvee^{11} S^3$ & $3$ & $L_2$, with action of order $18$\\
$(18,6)$ & $S^3$ & $3$ & trivial $\Z$\\
$(19,4)$ & $S^3$ & $3$ & trivial $\Z$\\
$(19,5)$ & $S^3$ & $3$ & trivial $\Z$\\
$(21,4)$ & $S^4\vee S^4$ & $4$ & $\Z[t]/(t^2+t+1)$, inflated from $C_3$\\
$(22,5)$ & $S^4$ & $4$ & $\Z[t]/(t+1)$, sign action of order $2$\\
$(22,6)$ & $S^3$ & $3$ & trivial $\Z$\\
$(23,5)$ & $S^3$ & $3$ & trivial $\Z$\\
$(24,4)$ & $S^4\vee S^4$ & $4$ & $\Z[t]/(t^2+t+1)$, inflated from $C_3$\\
$(25,4)$ & $S^5$ & $5$ & trivial $\Z$\\
$(27,5)$ & $\bigvee^7 S^5$ & $5$ & $L_3$, inflated from $C_9$\\
\bottomrule
\end{tabular}
\end{center}
The action on $L_3$ has order nine.
For $A_3=\ker(t-1:L_3\to L_3)$ and $D_3=\ker(\Phi_9(t):L_3\to L_3)$,
\begin{equation}\label{eq:index3}
 L_3/(A_3\oplus D_3)\cong\Z/3\Z.
\end{equation}
For $A_2=\ker(t^9-1:L_2\to L_2)$ and $D_2=\ker(\Phi_6(t):L_2\to L_2)$,
\begin{equation}\label{eq:index2}
 L_2/(A_2\oplus D_2)\cong(\Z/2\Z)^2.
\end{equation}
For $U_d=\ker(\Phi_d(t):L_{14}\to L_{14})$, $d=2,7,14$, one has
\begin{equation}\label{eq:index14}
 L_{14}/(U_2\oplus U_7\oplus U_{14})\cong\Z/7\Z\oplus(\Z/2\Z)^6.
\end{equation}
Hence $L_{14}$ is not the direct sum of its three cyclotomic integer factors.
Consequently $L_3\not\cong\Z\oplus\Z[\zeta_9]$ and $L_2\not\cong\Z[C_9]\oplus\Z[\zeta_6]$ as $\Z[t]$-modules, although their rationalizations have these corresponding decompositions.
\end{theorem}

The finite independence-complex calculations associated with the expanded core table and its action comparisons use twenty-seven parameter inputs: the eighteen pairs in Theorem~\ref{thm:main}, $(15,3),(25,3)$, the cycle squares $(15,2),(16,2),(17,2),(18,2),(19,2),(22,2)$, and $(21,5)$.
The twenty-six pairs displayed in that collection represent twenty-five graphs; the three benchmarks of Section~\ref{sec:benchmarks} are listed separately.
Indeed, multiplication by $5$ identifies $G(19,4)$ with $G(19,5)$, carrying one-step translation to five-step translation; both homology actions are trivial.
The additional input $G(21,5)$ is similarly isomorphic to $G(21,4)$ and adds no graph representative or displayed theorem pair.
The five clique-complex inputs are $\mathcal X_5,\mathcal X_6$ and $\mathcal Y_{13,3},\mathcal Y_{14,5},\mathcal Y_{14,6}$, where $\mathcal Y_{n,s}=\operatorname{Cl}(G(n,s))$.
The parameter-family proofs are independent of these finite checks.

\section{Critical faces and chain supports}\label{app:cells}
For the eighteen cases of Theorem~\ref{thm:main}, each ordinary matching has the critical vertex $\{0\}$.
Their positive-dimensional critical faces are listed in matching order, which also gives the homology-basis order for the four higher-rank cases of that theorem.
The two further cycle-square critical lists and vertex priorities are specified separately in Proposition~\ref{prop:cyclesquare}.
Their translation matrices and orbit-basis checks are collected in Appendix~\ref{app:matrices}.
For the additional low-rank cases, the critical faces are
\[
\begin{array}{ll}
\K(14,6):&\{1,4\},\\
\K(15,4):&\{1,7,10\},\{2,8,11\},\\
\K(17,4):&\{2,8,11,13\},\\
\K(19,4):&\{2,4,11,13\},\\
\K(19,5):&\{2,5,9,12\},\\
\K(23,5):&\{1,4,12,15\},\\
\K(24,4):&\{1,7,10,16,19\},\{2,8,11,17,20\}.
\end{array}
\]

For $\K(14,2)$ they are
\begin{align*}
 &\{1,5,9\},\{1,5,10\},\{1,6,10\},\{1,6,11\},\{2,6,10\},\{2,6,11\},\{2,7,11\},\\
 &\{2,7,12\},\{3,7,12\},\{3,8,12\},\{3,8,13\},\{4,8,13\},\{4,9,13\}.
\end{align*}
Their cycle support sizes are $8,8,8,8,8,8,8,14,12,12,12,12,12$.
The first lift is the explicit cross-cycle in \eqref{eq:cross14}.

For $\K(18,4)$ they are
\[
 \{1,7,10,16\},\{2,4,11,13\},\{2,8,11,17\},\{3,5,12,14\},
 \{4,6,13,15\},\{5,7,14,16\},\{6,8,15,17\}.
\]
For $\K(18,5)$ they are
\begin{align*}
 &\{1,4,12,15\},\{1,9,12,16\},\{2,5,9,12\},\{2,5,13,16\},\\
 &\{2,10,13,17\},\{3,5,12,14\},\{3,6,10,13\},\{3,6,14,17\},\\
 &\{4,6,13,15\},\{5,7,14,16\},\{6,8,15,17\}.
\end{align*}
For $\K(21,4)$ they are
\[
 \{3,5,12,14,17\},\qquad \{4,6,13,15,18\}.
\]
For $\K(25,4)$ the four-cell is $\{1,7,10,16,19\}$.
The five-cells, in the order of \eqref{eq:row}, are
\[
 \{5,7,14,16,21,23\},\qquad \{6,8,15,17,22,24\}.
\]
For $\K(27,5)$ they are
\begin{align*}
 &\{1,9,12,16,19,25\},\{2,5,13,16,20,22\},\{2,10,13,17,20,26\},\\
 &\{3,5,12,14,21,23\},\{4,6,13,15,22,24\},\\
 &\{5,7,14,16,23,25\},\{6,8,15,17,24,26\}.
\end{align*}
The cycle support sizes, in these respective orders, are
\[
\begin{array}{ll}
\K(18,4):&16,28,33,34,32,34,38,\\
\K(18,5):&16,20,32,20,24,32,33,24,33,29,37,\\
\K(27,5):&96,224,252,180,245,225,265.
\end{array}
\]

\subsection{Mixed-dimensional critical faces}\label{app:mixed}
In each line, faces of a given dimension are ordered as the columns or rows of the corresponding differential in Section~\ref{sec:mixed}.
Every complex also has the critical vertex $\{0\}$.
\begin{align*}
\K(15,5),\ q=2:&\quad\{1,4,12\},\{2,5,9\},\{3,6,10\};\\
q=3:&\quad\{4,6,10,12\}.\\
\K(16,4),\ q=2:&\quad\{1,7,10\};\\
q=3:&\quad\{5,7,12,14\},\{6,8,13,15\}.\\
\K(18,6),\ q=3:&\quad\{2,5,12,15\},\{3,6,11,13\},\{4,7,12,14\};\\
q=4:&\quad\{5,7,9,12,14\},\{6,8,10,13,15\}.\\
\K(22,5),\ q=4:&\quad\{2,5,9,12,20\},\{3,5,12,14,18\},\{3,6,10,13,21\};\\
q=5:&\quad\{5,7,14,16,18,20\},\{6,8,15,17,19,21\}.\\
\K(22,6),\ q=3:&\quad\{1,4,12,14\},\{1,5,9,13\},\{2,6,10,14\};\\
q=4:&\quad\{1,5,9,14,18\},\{2,6,10,15,19\},\\
&\quad\{3,7,11,16,20\},\{4,8,12,17,21\};\\
q=5:&\quad\{7,9,11,16,18,20\},\{8,10,12,17,19,21\}.
\end{align*}

\section{Translation matrices}\label{app:matrices}
For the two cycle squares in Proposition~\ref{prop:cyclesquare}, the matrices use its listed critical-lift bases.
Columns give images of basis vectors:
\[
T_{15,2}=\begin{pmatrix}
-1&-1&-1&-1&-1&-1&-1\\
0&0&1&0&0&0&0\\
0&0&0&0&0&0&1\\
0&0&0&0&0&1&0\\
1&0&0&0&0&0&0\\
0&1&0&0&0&0&0\\
0&0&0&1&0&0&0
\end{pmatrix},\qquad
T_{16,2}=\begin{pmatrix}
0&0&0&0&1\\
1&0&0&0&0\\
0&0&1&0&0\\
0&-1&0&0&0\\
0&0&1&1&0
\end{pmatrix}
.
\]
Together with the explicit vectors in Proposition~\ref{prop:cyclesquare}, these small matrices specify the two orbit-basis determinant checks without repeating the chain construction.

For each matrix below, columns are the images of the basis classes $[z_j]$ in the preceding critical-face order.
The orbit matrices are completely specified by $B=(v,Tv,\ldots,T^{r-1}v)$, with $v=e_2,e_1,e_2$ for the three displayed matrices respectively.
Their determinants and annihilators are those of Proposition~\ref{prop:cyclic}.
\begingroup\small\renewcommand{\arraystretch}{0.95}
\[
 T_{18,4}=\begin{pmatrix}
0&0&0&0&1&0&-1\\
1&1&-1&0&-1&0&1\\
1&0&0&0&0&0&0\\
0&1&0&0&0&0&0\\
0&0&0&1&0&0&0\\
0&0&0&0&1&0&0\\
0&0&0&0&0&1&0
\end{pmatrix},\qquad
 T_{18,5}=\begin{pmatrix}
0&0&0&0&0&-1&0&-1&0&-1&0\\
0&0&0&0&0&0&0&0&0&0&1\\
0&0&1&1&-1&1&0&0&0&1&0\\
1&0&0&0&0&0&0&0&0&0&0\\
0&1&0&0&0&0&0&0&0&0&0\\
0&0&0&-1&1&0&-1&0&-1&0&0\\
0&0&1&0&0&0&0&0&0&0&0\\
0&0&0&1&0&0&0&0&0&0&0\\
0&0&0&0&0&1&0&0&0&0&0\\
0&0&0&0&0&0&0&0&1&0&0\\
0&0&0&0&0&0&0&0&0&1&0
\end{pmatrix}.
\]
\[
 T_{27,5}=\begin{pmatrix}
0&0&0&0&1&0&-1\\
-1&1&1&0&1&0&-1\\
1&0&0&0&0&0&0\\
0&-1&0&0&0&0&0\\
0&0&0&1&0&0&0\\
0&0&0&0&1&0&0\\
0&0&0&0&0&1&0
\end{pmatrix}.
\]

\endgroup

\section{Guide to the additional action comparisons}\label{app:guide}
\begin{center}
\begin{tabularx}{\textwidth}{@{}lXl@{}}
\toprule
Parameters & Invariant identified & Result\\
\midrule
Cycle-power families & Integral actions and scalar lifts & Section~\ref{sec:cyclepowerfamily}\\
Prime-order subgroups of $\mathcal X_s$ & Geometric gluing and scalar lifts & Theorem~\ref{thm:geometricfamily}\\
Eighteen core pairs & Low-rank and glued lattices & Theorem~\ref{thm:main}\\
$(5(2\ell+1),3)$ & Odd-vertex family comparison & Proposition~\ref{prop:rankfour}\\
$(2m,3)$, $m\ge5$ & Even-vertex family comparison & Proposition~\ref{prop:evenfamily}\\
odd $n\ge7$, $5\nmid n$ & Jump-three sphere actions & Proposition~\ref{prop:oddgeneric}\\
$(15,2),(16,2)$ & Cyclic lattices and factor kernels & Proposition~\ref{prop:cyclesquare}\\
$(17,2),(18,2),(19,2)$ & Two-orbit dihedral summands & Proposition~\ref{prop:twoorbit}\\
$(18,2)$ & Cycle subgroup and nonsplit free quotient & Propositions~\ref{prop:freecrossdefect}, \ref{cor:freecrossnonsplit}\\
$(22,2)$ & Dihedral lattice and parity subgroup & Propositions~\ref{prop:tilingindex}, \ref{prop:tilingdihedral}\\
$\mathcal X_s$, $s\geq3$ & Triangle-boundary family mechanism & Proposition~\ref{prop:flagfamily}\\
Degree-four clique branches & Geometric maps and all scalar lifts & Proposition~\ref{prop:flaggraphs}\\
\bottomrule
\end{tabularx}
\end{center}

\section*{Acknowledgements}
We acknowledge Gewu Intelligence Lab for providing the collaborative research environment in which this project was developed.

\section*{Disclosure of automated assistance}
OpenAI GPT-5.6 Sol, Anthropic Claude Fable 5 and Grok 4.6 assisted with the original circulant arguments and exposition. OpenAI GPT-6.1 Sol developed the saturation, integral-action, finite-state and fixed-space arguments, constructed exact computational checks and drafted the expanded manuscript. OpenAI GPT-6 Astra checked derivations and references and revised the exposition. Anthropic Claude Fable 5.1 reviewed the circulant arguments, re-checked the derivations and references of the complete manuscript, and revised the exposition.

\begingroup
\raggedright
\bibliographystyle{amsplain-fullnames}
\bibliography{references}
\endgroup
\end{document}